\documentclass[12pt,a4paper,twoside]{amsart}
\usepackage{a4wide, amsmath,amsthm, amsbsy,amsfonts,amssymb, txfonts}
\usepackage{supertabular}
\usepackage[mathscr]{eucal}
\usepackage{stmaryrd,mathrsfs,graphicx, amscd, tikz-cd, cancel}
\usetikzlibrary{matrix,arrows,decorations.pathmorphing}
 \usepackage[normalem]{ulem}
\usepackage[
	linkcolor=blue,
	citecolor=red,
	urlcolor=red,
]{hyperref}
\usepackage[active]{srcltx}
\allowdisplaybreaks
\numberwithin{equation}{section}

\def\AA{{\mathbb A}}

\def\CC{{\mathbb C}}  
\def\DD{{\mathbb D}}

\def\HH{{\mathbb H}}
\def\Hcal{{\mathbb H}}

\def\PP{{\mathbb P}}
\def\QQ{{\mathbb Q}} 
\def\RR{{\mathbb R}} 
 
\def\ZZ{{\mathbb Z}}

\def\G{\Gamma}

\def\ct{{\rm cont}}

\def\qlog{\mathit{qlog}}
\def\llog{\mathit{log}}

\def\tot{{\mathit tot}} 

\def\bs{\backslash}
\newcommand{\eps}{\varepsilon}

\newcommand{\p}{\partial}

\def\Acal{{\mathcal A}}

\def\Ccal{{\mathcal C}}
\def\Cscr{{\mathscr C}}

\def\Dscr{{\mathscr D}}

\def\Fcal{{\mathcal F}} 
\def\Hcal{{\mathcal H}} 
\def\Ical{{\mathcal I}} 
 
\def\Kcal{{\mathcal K}}
\def\Lcal{{\mathcal L}}
\def\Lscr{{\mathscr L}}
\def\Mcal{{\mathcal M}}
\def\Ocal{{\mathcal O}}
\def\Pcal{{\mathcal P}}
\def\Pscr{{\mathscr P}}

\def\Rcal{{\mathcal R}}
\def\Scal{{\mathfrak S}}

\def\Vcal{{\mathcal V}}

\def\la{\langle}
\def\ra{\rangle}
\def\half{{\tfrac{1}{2}}}

\def\glie{{\mathfrak{g}}}
\def\hlie{{\mathfrak{h}}}
\def\klie{{\mathfrak{k}}}

\def\mfrak{{\mathfrak{m}}}
\def\cfrak{\mathfrak{c}}

\def\gfrak{\mathfrak{g}}

\def\nfrak{\mathfrak{n}}

\def\Rscr{\mathscr{R}}

\def\Vect{\mathscr{V}\!ect}

\def\ssm{\smallsetminus}

\def\pt{{\scriptscriptstyle\bullet}}

\newcommand\qLog{{}^q\!\!\Lcal og}

\newcommand\ad{\operatorname{ad}}

\newcommand\br{\operatorname{\mathit{br}}}

\newcommand\coker{\operatorname{Coker}}

\newcommand\edge{\operatorname{edg}}

\newcommand\Gr{\operatorname{gr}}

\newcommand\Hom{\operatorname{Hom}}

\newcommand\inj{\operatorname{Inj}}

\newcommand\KS{\operatorname{KS}}
\newcommand\Lie{\operatorname{Lie}}

\newcommand\sym{\operatorname{Sym}}

\newcommand\SL{\operatorname{SL}}

\newcommand\res{\operatorname{Res}}
\newcommand\sign{\operatorname{sgn}}
\newcommand\slin{\operatorname{\mathfrak{sl}}}
\newcommand\spec{\operatorname{Spec}}
\newcommand\supp{\operatorname{supp}}

\newcommand\PBW{\operatorname{PBW}}
\newcommand\PSL{\operatorname{PSL}}

\newcommand\tr{\operatorname{Tr}}
\newcommand\U{\operatorname{U}}

\newtheorem{theorem}{Theorem}[section]
\newtheorem{lemma}[theorem]{Lemma}
\newtheorem{proposition}[theorem]{Proposition}
\newtheorem{corollary}[theorem]{Corollary}

\newtheorem{definition}{Definition}\numberwithin{definition}{section}

\theoremstyle{remark}
\newtheorem{example}[theorem]{Example}
\newtheorem{remark}[theorem]{Remark}

\newtheorem{question}[theorem]{Question}

\title[Conformal blocks and Hodge theory]{Conformal blocks and the cohomology of configuration spaces of curves}
\author{Eduard Looijenga}
\address{\vbox{\noindent Mathematisch Instituut, Universiteit Utrecht\newline
Mathematics Department, University of Chicago}}
\begin{document}
\begin{abstract}
We realize any space of conformal blocks attached to a punctured curve inside the cohomology of a configuration space of
that curve and compare the WZW connection with the Gau\ss-Manin connection.
\end{abstract}
\maketitle
\hfill{\textsl{In memory of Tonny A.~Springer (1926-2011)}}

\section*{Introduction} 
The theory of conformal blocks begins with taking as input some data of a discrete nature: a simple complex Lie algebra 
$\glie$, a positive integer $\ell$ and a finite number of finite dimensional irreducible representations 
$V_1, \dots, V_n$ of $\glie$. Once this is given,  the theory assigns to every compact Riemann surface $C$ and 
set of pairwise distinct  point $p_1,\dots , p_n$ on $C$ a finite dimensional complex vector space  that for the purpose of 
this introduction we shall denote by $\Vcal_\ell(C;  (V_i, p_i)_{i=1}^n)$ (if we  want this vector space  to be nonzero, 
then we must in fact require that each $\glie$-representation $V_i$ is \emph{of level} $\le \ell$, but there is no need 
to recall that condition here).
The construction of this vector space is essentially algebraic and so it is not a surprise that it depends holomorphically 
on the pointed surface $(C; p_1, \dots, p_n)$.  But what makes it remarkable is that a much stronger property is true: 
if we vary the complex structure on 
$C$ and the position of the points in holomorphic family over a  contractible complex manifold $S$, then the 
associated  projective spaces are canonically identified with each other. To be precise, the vector spaces in question 
define a holomorphic vector bundle 
over $S$ which comes endowed with a projectively flat connection (it will be an ordinary  flat connection if we also give a generator of the 
tangent space of $C(s)$ at every $p_i(s)$ and a generator of the associated Hodge line  $\det H^0(C(s), \Omega_{C(s)})$, all depending holomorphically on $s\in S$). This is the famous \emph{WZW connection}, 
which in another setting is called the \emph{Hitchin connection}. It is expected to be a unitary connection and this has been proved to be so when $\glie=\slin (r)$.

This connection is not easily written down, unless $C$ is 
the Riemann sphere in which case it has a very concrete matrix description given by the 
\emph{Knizhnik-Zamolodchikov} form. Its projective flatness gives rise to projective representations of the 
$n$-pointed mapping class group of genus $g=g(C)$ and this, in turn, is at the basis of its topological applications in 
knot theory. It is also via this interpretation that its unitarity  has been checked for $\glie=\slin (r)$ via the theory of quantum representations.   The theory is now well over  30 years old and comes in several  incarnations, 
among them a purely topological one, 
an algebraic one in the guise of the representation theory of quantum groups,  and one grounded in 
analysis. Yet so far it resisted in being understood in conventional algebro-geometric terms, except for the genus zero case: 
the fundamental paper by Ramadas \cite{ramadas}, which concerned the case $\glie=\slin(2)$ and proved its unitary nature,  led, as we showed in \cite{looij2010},  to a description of the spaces 
$\Vcal(\PP^1; (V_i, p_i)_{i=1}^n)$ in terms of the  Hodge theory of local systems of configuration  spaces on points in 
$\PP^1\ssm\{p_1, \dots, p_n\}$ with the KZ-connection being understood as a Gau\ss-Manin connection. 
The expectation was that  this approach would generalize to the other Lie algebras and indeed, 
shortly afterwards Belkale \cite{belkale} showed this to be the  case.

This article purports to begin a similar program for the higher genus case.
The first step amounts to construction  for every integer $N\ge 0$ a linear map
\[
\textstyle \gamma_N: \Vcal_\ell(C;  (V_i, p_i)_{i=1}^n)\hookrightarrow \Hom^{\sign_N}_\CC(\glie^{\otimes N}\otimes V,F^{-1}H^N(\inj_N(C\ssm\{ p_1, \dots, p_n\}); \CC(N))
\]
and a `logarithmic' companion 
\[
\textstyle \gamma'_N: \Vcal_\ell(C;  (V_i, p_i)_{i=1}^n)\hookrightarrow \Hom^{\sign_N}_\CC(\glie^{\otimes N}\otimes V,F^0H^N(\inj_N(C\ssm\{ p_1, \dots, p_n\}); \CC(N)),
\]
where $V:=V_1\otimes\cdots\otimes V_n$ and other notation used  at the right hand side is to be  understood as follows:  $\inj_N(C\ssm\{ p_1, \dots, p_n\})$ stands for the (configuration) space of injective maps 
$\{1, \dots, N\}\hookrightarrow C\ssm\{ p_1, \dots, p_n\}$. Its cohomology carries a mixed Hodge structure to which we here applied a Tate twist  with $\CC(N)$. Such a Tate twist identifies $H^N(\inj_N(C\ssm\{ p_1, \dots, p_n\}); \ZZ(N))$ with the Borel-Moore  homology group
($=$ homology with closed support) $H^{cl}_N(\inj_N(C\ssm \{ p_1, \dots, p_n\}); \ZZ)$, which in this context is somewhat more natural.
This implies that its Hodge filtration has trivial $F^1$ and  that $F^0$ is the smallest possibly nonzero item of this filtration. The companion $\gamma'_N$ has its target represented by a space of homomorphisms that take values in this $F^0$ part, which is represented by the space of  regular $N$-forms on $\inj_N(C\ssm\{ p_1, \dots, p_n\})$ with logarithmic poles
along the boundary divisor in $C^N$. The permutation group $\Scal_N$ acts on this space as well as on  $\glie^{\otimes N}\otimes V$, and the superscript ${}^{\sign_N}$ means the we only consider the linear maps  that transform according to the sign character.  This
is replaced by the trivial character (so that we take the $\Scal_N$-equivariant maps) if we replace the  $N$-forms by polydifferentials  in that  degree.

Two of our main results,  Theorems \ref{thm:globaldual1} and \ref{thm:globaldual2}, say that $\gamma_N$ and $\gamma'_N$ are 
embeddings when $N$ is large and  characterize the image of both (for all $N$). Let us only mention here that $\gamma_N$ determines its logarithmic companion and vice versa and that $\gamma_{N-1}$ is obtained from  $\gamma_{N}$ as a residue of and likewise for $\gamma'_{N-1}$. This leads us to compare the WZW-connection and the Gau\ss-Manin connection: 
Theorem \ref{thm:gmversuswzw} describes their difference and 
shows how this difference factors  through a Kodaira-Spencer homomorphism.

Central to the theory of conformal blocks are the \emph{propagation property}  and the \emph{factorization property}.
The propagation property is absorbed by our geometric approach in the sense that it is embodied in the existence of global 
polydifferentials, but the reader will find the factorization property conspicuously missing. Although there are no serious obstacles for including it in our treatment (this requires considering configuration spaces 
of the smooth parts of nodal curves), including this here would serve no obvious purpose and  result in making the paper even longer than it already is.  Another, 
admittedly  somewhat more personal reason for its omission is that all the known  derivations of a (Verlinde) formula for the dimension of  
$\Vcal_\ell(C;  (V_i, p_i)_{i=1}^n)$ are based on the factorization property and we thus want to keep alive our belief that this De Rham approach should ultimately lead to a more direct verification (that also bypasses the WZW connection to prove that we get a vector bundle over $S$).  

The ultimate goal is to express $\Vcal_\ell(C;  (V_i, p_i)_{i=1}^n)$ entirely be  in terms of the De Rham theory of the 
fundamental groupoid of $C\ssm \{p_1, \dots, p_n\}$, when  restricted to the  points `at infinity' obtained by choosing a nonzero tangent vector at each $p_i$ and to do this in such a manner that  topological quantum field theory emerges as its Betti incarnation (see also Remark \ref{rem:topinterpretation}).  This is a   dream for now,   but it is in this spirit that we would like to see Segal's approach to the WZW theory \cite{segal} acquire its algebro-geometric formulation.
We should mention here that the  book \emph{Chiral Algebras} \cite{bd2004} by Beilinson-Drinfeld takes a full De Rham ($\Dscr$-module) approach to vertex algebras, and in particular  to conformal blocks. It is likely that the 
two approaches are related, but we admit that  at present it is not clear to us how.

The appearance of polydifferentials in this setting is not new. It is already implicit in the early work of 
Fred Cohen, who linked the cohomology of the configuration space of $N$ points in  $\CC$ with the free 
Lie algebra in $N-1$ variables (see his survey-like paper \cite{cohenf95}). 
To the best of our knowledge, the introduction of polydifferentials proper in a related context appeared first in a paper 
by Beilinson-Drinfeld \cite{bd}: polydifferentials  entered there in their description of the topological dual of the universal enveloping algebra
of a centrally extended loop algebra.  We used polydifferentials in a somewhat different manner in the KZ-setting  in \cite{looij2012} (see also \cite{BBM}). But the present article,  which we like to regard  as generalization of the work of Cohen alluded to above, was certainly inspired by the Beilinson-Drinfeld paper.

The road towards the proof of the main result is lengthier than we wish, but hopefully \emph{vaut le voyage}, 
as on our way we introduce some concepts and obtain results that might have an independent interest, 
especially where  they concern the cohomology of configuration spaces of punctured compact Riemann surfaces. 
This includes the computation of the periods  of a remarkable meromorphic  $2$-form on $C^2$ 
whose polar divisor is twice the diagonal that not so long ago appeared in a paper by Colombo-Frediani-Ghigi 
\cite{CFG},  and its higher dimensional  generalization as a meromorphic $N$-form on $C^{\ZZ/N}$ 
whose polar divisor is the sum of the diagonal hypersurfaces defined by the equality of two successive 
coordinates $z_i=z_{i+1}$, $i\in\ZZ/N$. These forms (that we  use to associate to $\gamma_N$ its logarithmic companion)
have such a classical  flavor that they might well have already been known in the 19th century.
\\

Let us now give a brief description of  the contents of each section. 
The central notion  of the first section is that of a \emph{Lie module structure with invariant form}, admittedly a mouthful, 
but thus named because this refers to the motivating example. 
We found this notion the most suitable cloth for presenting our main results in.  
We construct two examples which at first sight have nothing to do with Lie algebras: one is associated
with a complete discrete valuation ring and another with a smooth affine curve. 
It also leads us to the notion of what we have called a \emph{quasi-logarithmic} form and its  polydifferential counterpart.  
 While Subsection \ref{subsect:lierep} is modest in hardcore mathematical content
(as it merely introduces some of the notions in which some of our main results are later conveniently stated),  
it enables us to make at an early stage the connection with the work of Cohen and thus give the reader some idea 
in which direction where we are heading.

Section $2$ is devoted to the Hodge theory of the configuration space of an affine curve. 
In particular, we review  the spectral sequence that converges to the weight filtration  and do the same  for the 
Hodge filtration, both generalizing  other such spectral sequences in the literature (but in the end  all going back to 
Deligne's construction of a mixed Hodge structure on a smooth variety). It leads us to an analogue of the  Arnol'd algebra.

Section $3$ might be regarded as a chapter in the theory of Riemann surfaces and  is largely independent of the 
previous two sections. It is only when we formulate the summarizing Corollary \ref{cor:zetadef} that we find it 
convenient to use the language developed in Section $1$. We here introduce the canonical polydifferentials alluded to above.

In Section $4$ we first review how the vector spaces $\Vcal_\ell(C;  (V_i, p_i)_{i=1}^n)$ are defined and then  
apply the results obtained  in the previous sections to obtain a description of them in terms of polydifferentials.

The final Section $5$ exhibits the WZW connection in these terms and subsequently compares this with the Gau\ss-Manin connection. Our main result here  is  Theorem  
\ref{thm:gmversuswzw}. 
\\

\emph{Acknowledgements.} This paper  is part of a long term project that started more than a decade ago.
During that time I have of course been influenced
by others through their papers or through correspondence. I mention in particular the  paper \cite{bd} by 
Beilinson-Drinfeld  and a correspondence I had in 2011 with Ashok Raina in connection with his paper with 
Biswas \cite{br1}.

Parts of this research was carried out  while I was being supported by the National Science Foundation of the P.R.C.,  the Jump Trading Fund and the Mittag-Leffler Institute. I thank these three institutions for their support.

\subsubsection*{Conventions and notation}\label{subsubsect:conventions}
 Let $I$ be a finite set of $N$ elements.
If $N>0$ and $X$ is a topological space, then the  \emph{$I$-configuration space of} $X$ is the open subspace of $X^I$ 
defined by the injective maps $I\hookrightarrow X$. 
When $I=\{1, \dots, N\}$ we also call this the \emph{$N$-point configuration space}.
We shall therefore denote it by $\inj_I(X)$ resp.\ $\inj_N(X)$ (the literature has not agreed on notation:  
this is also denoted $F(X,N)$, $F_N(X)$, $\textit{Conf}_N(X)$, $\textit{Pconf}_N(X), \dots$).

We write $\sign (I)$ (resp.\ $\sign_N$) for $\wedge^N(\ZZ^I)$ (resp.\  $\wedge^N(\ZZ^N)$), 
the convention being that $\sign(\emptyset)=\ZZ$. So if $I$ is nonempty, then a linear order on $I$ defines a 
generator of $\sign (I)$ and two linear orders yield the same generator if and only if they differ by an even permutation.
Assume now $N>0$ and consider the set  $\tot(I)$ of linear orders on $I$, to be thought of as the set of bijections
$\{1, \dots, N\}\xrightarrow{\cong} I$. This is a principal set of both the symmetric group $\Scal_N$ (acting on the right) 
and the permutation group $\Scal(I)$ (acting on the left). If $A$ is a module over some ring $R$, 
then $A\otimes_R \cdots \otimes_R A$ (with $N$ tensor factors) is a left $R[\Scal_N]$-module and hence 
\[
A^{\otimes_R I}:= R^{\tot(I)}\otimes_{R[\Scal_N]} (A\otimes_R \cdots \otimes_R A)
\]
is a left $R[\Scal(I)]$-module.  It has the usual properties: it is universal for
$R$-multilinear maps with domain $A^I$ and we have canonical  isomorphisms 
$ A^{\otimes_R (I'\sqcup I'')}\cong A^{\otimes_R I'}\otimes A^{\otimes_R I''}$ (defined by juxtaposition) and 
$(A'\oplus A'')^{\otimes_R I}\cong \oplus_{I=I'\sqcup I''} 
(A')^{\otimes_R I'}\otimes (A'')^{\otimes_R I''}$. We could of course have defined $A^{\otimes_R I}$ more directly 
as the quotient of the free $R$-module generated by $A^I$ by the submodule generated by the usual relations. 
We stipulate that $A^{\otimes_R\emptyset}=R$.

This definition must be modified if we are dealing with a $\ZZ/2$-graded commutative algebra 
$A_\pt=A_+\oplus A_-$ obeying the Koszul rule. Then
$A_\pt\otimes_RA_\pt$ has also that structure by the usual (Koszul) rule, but the transposition must carry a sign 
in order that we get an $R$-algebra automorphism: it is then defined by 
$a\otimes_R b\mapsto (-1)^{|a|.|b|}b\otimes_R a$ (with $a, b$ homogenenous). 
So if $N=2$, then  the left $R^{\tot(I)}$-module structure
on $A_\pt\otimes_RA_\pt$ must be given by this action in order to define $A_\pt^{\otimes_R I}$ as a 
$\ZZ/2$-graded commutative algebra. Note that
its odd part contains the $R$-submodule $(A_-)^{\otimes_R I}\otimes\sign(I)$ 
(where we used the $\otimes_RI$-convention above). This generalizes in an evident manner to the case where $I$ 
is an arbitrary finite set. Such a  situation shows up if we apply the K\"unneth formula to $X^I$, where $X$ is a 
nonempty space of finite type for which $H^\pt(X; R)$ has no $R$-torsion: then $H^\pt(X^I; R)=H^\pt(X)^{\otimes_R I}$.

In what follows, $k$ denotes an algebraically closed field of characteristic zero (which is assumed to be $\CC$ 
whenever Hodge theory is involved).

Other nonstandard notation is listed below in order of appearance.\\

\smallskip
{\small
\begin{supertabular}{ll}
\ref{subsect:lierep} & $[ij]$, $I_{ij}$, $r_i$, $r_{ij}$, $r_i^p$, $s_{ij}$, $\Lcal^{s}_{\gfrak,V}$,  $\Lcal_{\gfrak,V}$, $\Rscr^s_P$,  $\Rscr^s$, $\Rscr_P$, 
$\Lcal og_{C,P}$, $\qLog_{C,P}$ \\
\ref{subsect:polydiff} &  $\Fcal^{(I)}$,  $\Omega^{(I)}$\\
\ref{subsect:LiemoduleDVR} & 
$\Ocal$, $K$, $\omega$, $\theta$,  $F^nK$ etc, $\DD$, $\DD^\times$, $\br$, $\omega^{(I)}\la\qlog\ra$, 
$o$,  $\DD^{(I)}$,  $\res_{i\to j}$,  $r'_{ij}$,  $r''_{ij}$, $\omega^{(I)}\la\qlog\ra_o$, $\omega^{(I)}\la\llog\ra$,  
\\
{} &  $\omega^{(I)}\la\qlog\ra_\G$, $\Lcal og_K$,  $\qLog_K$, 
$\eta_\sigma$, $\hat\eta_I$, $\hat\xi$\\
\ref{subsect:Liepointedcurve} & $\Omega_C( P)^{(I)}\la\qlog\ra$\\
\ref{subsect:basicss} & $\Delta_{I,P}$, $\Delta{ij}$, $\pi_i$, $\Delta_\Pcal$, $c(\Pcal)$, $E_\Pcal$, $A_\Pcal$, 
$A^\pt_{I,P}$, $g_{ij}$, $h_{ij}$, $g(Y)$, $h^p(Y)$, $\Acal(I)$, $\Acal_P(I)$\\
\ref{subsect:projcase} & $\mu$, $A^{\pt,\pt}_{I, P}(C)$, $W_q(|I|,P)$, $\pi^i$ \\
\ref{subsect:canbidiff}  & $\zeta$, $D_n$, $E_n$, $\zeta_n$, $\zeta_\G$, $\qLog_C$, $\hat\zeta_I$, $\hat\xi$\\
\ref{subsect:affineLie} & $\check c$, $\cfrak$, $\glie K$, $\widehat{\glie K}$, $\hat\hlie$, $\alpha_i$, $\theta$, $\check n_i$, $c$, $V^+$, $\lambda_i$, $1_\lambda$, $\tilde \Vcal_{\hat\lambda}(K)$, $\glie_-$, $\widehat{\glie K}_-$, 
$\Vcal_{\hat\lambda}(K)$, $\PBW_\pt$, $\rho$, $\check h$, $\bar c$\\
\ref{subsect:dualrep} & $\tilde\Vcal^{\hat\lambda}(K)$, $\Vcal^{\hat\lambda}(K)$, $\tilde\Vcal_\qlog^{\hat\lambda}(K)$, $\tilde\Vcal_\llog^{\hat\lambda}(K)$, $\Vcal_\qlog^{\hat\lambda}(K)$, $\Vcal_\llog^{\hat\lambda}(K)$, $K_P$,
$\tilde\Vcal_{\hat\lambda}(K_P)$, $\tilde\Vcal_{\hat\lambda}(K_P)$, $\Vcal_{\hat\lambda}(K_P)$, etc.\\
\ref{subsect:propagation} & $\Vcal_{\hat\lambda}(C,P)$, $\Vcal^{\hat\lambda}(C,P)$, $\tilde\Vcal_{\hat\lambda}(\hat C)$, $\tilde\Vcal_\ell(\hat C)$ etc.\\
\ref{subsect:global} &  $\Vcal_\qlog^{\hat\lambda}(C)$, $\Vcal_\llog^{\hat\lambda}(C)$\\
\ref{subsect:confblocks} &  $F$, $j$, $\mathring\Ccal$, $\mathring F$,$\vec p$, $\vec P$,  $\Kcal_P$, $\mfrak_p$, $\widehat{\glie\Kcal_P}$, $\tilde\Vcal_{\hat\lambda}(\Ccal/S,P)$, $\Vcal_{\hat\lambda}(\Ccal/S,P)$,  $F^{(N)}$, $\inj_N\!\mathring F$, $\gamma_N$, $\gamma'_N$\\
\ref{subsect:ks} & $\KS_{\Ccal/S,\vec P}$ \\
\ref{subsect:ss} & $\hat\theta$, $\xi'$, $L(D)$, $\Lscr(D)$ (local)\\
\ref{subsect:wzw} & $\widehat{\glie\Kcal_P}$,  $\tilde D$, $\hat D$, $\Lscr^\dagger$, $\Lscr(D)$ (global)\\ 
\end{supertabular}
}

\begin{small}
\tableofcontents
\end{small}
\section{Lie module structures and polydifferentials}\label{sect:polydiff}

\subsection{Lie module structures}\label{subsect:lierep}
In order to motivate the notion that we are about to introduce, take  a Lie algebra $\glie$ over $k$ and associate 
with every \emph{nonempty} finite set $I$ the $k$-vector space $\gfrak^{\otimes I}$. It is clear that the permutation group 
of $I$ acts on that space, and so we may regard this as a functor 
from the groupoid of finite sets to the category of $k$-vector spaces. This functor comes  with  a collection  of 
`contractions' that we will spell out.  If $(i,j)$ is an ordered pair 
of distinct elements of  $I$, and $I_{ij}$ denotes the quotient of $I$ which identifies $i$ and $j$, 
then write $[ij]$ for the common image of $i$ and $j$ in $I_{ij}$. Applying  the Lie bracket to these tensor factors defines a 
linear map
\[
r_{ij}: \gfrak^{\otimes I}\to \gfrak^{\otimes {I_{ij}}}.
\]
An  obvious commutativity relation, the antisymmetry of the Lie bracket and the Jacobi identity then amount to what we 
shall call the 
\begin{description}
\item[(Lie properties)] $r_{ij}=-r_{ji}$, $r_{[ij], k}r_{ij} +r_{[jk], i}r_{jk} +r_{[ki], k}r_{ki}=0$ 
(as a map $\glie^{\otimes I}\to\glie^{\otimes I_{ijk}}$), $r_{ij}r_{kl}=r_{kl}r_{ij}$ 
(as a map $\glie^{\otimes I}\to\glie^{\otimes {I_{ij,kl}}}$; note that $I_{ij,kl}=I_{kl,ij}$),  
assuming that the indices are pairwise distinct. 
\end{description}
If we are also given a representation $V$ of $\glie$, then we can associate with any (possibly empty) finite set $I$ the 
tensor product $\glie^{\otimes I}\otimes V$ instead. Besides the contractions $r_{ij}$ obeying the above properties, 
we have in addition for every $i\in I$ a contraction 
\[
r_i: \gfrak^{\otimes I}\otimes V\to \gfrak^{\otimes {(I\ssm\{ i\}})}\otimes V
\]
defined by letting the $i$th tensor factor of $\glie^{\otimes I}$ act on $V$. They satisfy the
\begin{description}
\item[(representation properties)] $r_{[ij]}r_{ij}=r_i r_j-r_jr_i$ and $r_kr_{ij}=r_{ij}r_k$, where again, the indices are distinct. 
\end{description}

The preceding case can in fact be regarded as corresponding to the trivial representation by allowing $I$ to be empty: take $V=\glie^{\otimes\emptyset}=k$ with each $r_i$  identically zero.

If we are given a collection of representations  of $\glie$ indexed by a finite set $P$, $\{V_p\}_{p\in P}$, then 
 $V:=\otimes_{p\in P} V_p$ can be regarded as a representation of $\glie^P$, hence comes for each $p\in P$ with a structural map $\glie\otimes V\to V$ giving rise to  a contraction $r^p_i: \gfrak^{\otimes I}\otimes V\to \gfrak^{\otimes {(I\ssm\{ i\}})}\otimes V$ satisfying the representation property. The fact that the representations commute is expressed by the
 \begin{description}
\item[(commutativity relations)]
$r^p_ir^q_j=r^q_jr^p_i$ when $p\not=q$ and $i\not=j$.
\end{description} 
If we are further given  a symmetric bilinear form $s: \glie\otimes\glie \to k$ invariant under the adjoint 
representation of $\glie$, then the contractions 
\[
s_{ij}: \gfrak^{\otimes I}\otimes V\to \gfrak^{\otimes {(I\ssm \{i, j\})}}\otimes V \quad (i,j\in I \text{ distinct}).
\]
defined by applying $s$ to the distinct tensor slots named $i$ and $j$ obey:
\begin{description}
\item[(invariant form)]
$s_{ij}=s_{ji}$, $s_{ij}s_{kl}=s_{kl}s_{ij}$,   $s_{[ij],k}r_{ij}+s_{[ik],j}r_{ik}=0$,  assuming that the indices are pairwise distinct,
\end{description}
where the last identity expresses the fact that $s(X_k,[X_i,X_j])+s(X_j, [X_i,X_k])=0$.
The skew-symmetry of the $r_{ij}$ and the symmetry of the $s_{ij}$ imply that
$s_{[ij],k}r_{ij}$ transforms under a  permutation of $(i,j,k)$ according the sign character.
It is clear that any scalar multiple of $s$ also satisfies the invariant form property. 
We denote this system $\Lcal_{\glie,P, V}^s$, but usually omit  $P$ and $V$ if all the representations are trivial 
(all $r^p_i$ are zero), omit $P$ when there is just  one representation and omit $s$ when we are dealing with the 
trivial invariant form (all $s_{ij}$ are zero).

\begin{example}
Here is a simple variation on the preceding.
If  we let $\glie$ act on itself via the adjoint representation and on $V$ via its diagonal embedding in 
$\glie^P$, then the contraction operators 
$r_{ij}$, $r_i^p$, $s_{ij}$ all become $\glie$-equivariant. 
So they descend to operators that 
have as source and target the spaces of $\glie$-co-invariants $((\glie^P)^{\otimes I}\otimes V)_\glie$ 
that satisfy the same identities (we could have done this for the isogeny spaces of any  irreducible representation of 
$\glie$, of course). We write $(\Lcal_{\glie,P, V}^s)_\glie$ for this system. 
It is worth noting that if we take $I=\{1, \dots, N\}$ with $N\ge 1$, that then for every $i\in I$, 
\[
\textstyle (\sum_{j\not=i}r_{ij} +\sum_{p\in P} r_i^p)(X_N\otimes\cdots \otimes X_1\otimes v)=
X_i(X_N\otimes\cdots \otimes X_1\otimes v),
\]
so that we have then forced the identity $\sum_{p\in P}r^p_i+ \sum_{j\not= i}   r_{ij}=0$  when we pass to $\glie$-covariants (this also makes sense for an arbitrary nonempty finite set $I$, provided we make  the identification 
$I\ssm\{i\}\xrightarrow{\cong} I_{ij}$).
\end{example}

\subsubsection*{The category $\Rscr^s_P$} \index{$\Rscr^s_P$!category}
Since other examples of such functors play a central role in this paper, 
the underlying structure deserves  a definition. 
Given a set $P$, we define a pre-additive category $\Rscr^s_P$ (which means that its morphism sets are abelian groups 
and the composition is bilinear) whose objects are finite sets and which contains the groupoid of finite sets as a 
subcategory  (so that  $\hom_{\Rscr^s_P}(I,I)$ contains the group ring of the permutation group $\Scal(I)$ of $I$). 
We assume that for every finite set $I$ we are given:
\begin{itemize}
\item[(i)] for $i, j\in I$ distinct, morphisms  
$r^I_{ij}\in\hom_{\Rscr^s_P}(I,I_{[ij]})$ and $s^I_{ij}\in \hom_{\Rscr^s_P}(I,I\ssm \{i,j\})$, 
\item[(ii)] for $i\in I$ and $p\in P$,  a morphism $r^{I,p}_i\in\hom_{\Rscr^s_P}(I, I\ssm\{i\})$,
\end{itemize}
which obey the identities we listed that are associated with the \textbf{Lie}, the 
\textbf{representation}, the \textbf{commutativity} and the  \textbf{invariant form} properties.
For instance, the transposition $(ij)\in\Scal(I)$ composed with $r_{ij}$ resp.\ $s_{ij}$ must equal 
$r_{ji}=-r_{ij}$ resp.\ $s_{ji}=s_{ij}$.
The category $\Rscr^s_P$ is assumed to be universal for these properties, that is, every morphism is a 
$\ZZ$-linear combination of composites of  bijections and contractions of the  above type. 
We stipulate that $\hom_{\Rscr^s_P}(\emptyset,\emptyset)$ equals $\ZZ$. 
In case $P$ is a singleton resp.\ empty, we denote this category simply by $\Rscr^s_*$ resp.\  $\Rscr^s$. We write $\Rscr_P$
for the full subcategory for which the morphisms $s^I_{ij}$ are all zero. 

\begin{definition}\label{def:liestructure}
Let  $P$ a finite set and $\Cscr$ a pre-additive category.  
A \emph{$P$-fold Lie module in $\Cscr$ with invariant form} or briefly, a 
\emph{$\Rscr^s_P$-module} in  $\Cscr$,  is an additive 
 functor $\Lcal:\Rscr^s_P\to \Cscr$. An \emph{$\Rscr^s_P$-homomorphism} 
 $\Lcal\to \Lcal'$ between two $\Rscr^s_P$-modules is a natural transformation from $\Lcal$ to $\Lcal'$. 
We say that a  $\Rscr^s_P$-module is  \emph{co-invariant} if $\sum_{p\in P}r^p_i+ \sum_{j\not= i} r_{ij}$ 
is identically zero for all $I$ and $i\in I$ (and where we use the identifications $I\ssm \{i\}\cong I_{ij}$ when $|I|\ge 2$).
\end{definition}

In this paper,  $\Cscr$ will be often the category of modules over a ring, or more generally, 
the category of modules over a ringed space.

Since an $\Rscr^s_P$-module $\Lcal$ is already given by its values on the skeleton of $\Rscr^s_P$, 
it suffices to give for every integer $N\ge 0$ an object $\Lcal_N$ of $\Cscr$   that  comes with an action of 
$\Scal_N$ and for which we are given structural morphisms of degree $-1$ and $-2$ as above.

\begin{example}[Fred Cohen's example]\label{rem:cohenalg} Let $\Lie(X_1, \dots, X_N)$ be the free $k$-Lie algebra generated by $X_1, \dots, X_N$.
It is clear that by applying successively $N-1$ operators 
$r_{ij}$ to $X_N\otimes\cdots \otimes X_1\in \Lie(X_1, \dots, X_N)^{\otimes N}$ we  
produce all the Lie monomials in $X_1,\dots, X_N$ of degree $N$. 
Fred Cohen \cite{cohenf95} constructed a linear isomorphism between their $k$-span 
$\Lie_k[N]$ in $\Lie(X_1, \dots, X_N)$  and the homology group with twisted coefficients 
$H_{N-1}(\inj_N(\CC); k)\otimes\sign_N$.  Cohen's identification is conveniently described  as a perfect  duality between 
$\Lie_k[N]$ and $H^{N-1}(\inj_N(\CC); k(N-1))\otimes\sign_N$, where we inserted a Tate twist by $\ZZ(N-1)$ to help 
us keep track of both orientation issues and Hodge structures. This duality is set up by associating with the Lie variable  
$X_i$  the map 
\[
H^{N-1}(\inj_N(\CC); k)\otimes \sign_N\to H^{N-2}(\inj_{N-1}(\CC); k)\otimes \sign_{N-1}(-1)
\] 
that is essentially the residue map along the diagonal defined by $z_i=z_{i+1}$. 
Since the space $\Lie_k[N]$ is already generated by the right nested brackets 
$\ad_{X_{\sigma(N)}}\ad_{X_{\sigma(n-1)}}\cdots \ad_{X_{\sigma(2)}}(X_{\sigma (1)})$, 
where $\sigma$ runs over the permutation group $\Scal_N$, this gives a recipe
for an  iterated  residue map $H^{N-1}(\inj_N(\CC); k(N-1))\otimes \sign_N\to H^0(\CC; k)=k$.

In terms  of our set-up,  Cohen's  identification amount to cyclic relations among the operators $r_{ij}$. 
If we agree to let $r_{kj}r_{ji}$ stand for 
$r_{k,[ji]}r_{ji}$, then we have for a principal $\ZZ/N$-set $I$ with $N\ge 2$: 
\begin{equation}\label{eqn:cyclic}
\textstyle \sum_{i\in I}  r_{N-1+i,N-2+i} \cdots r_{2+i,1+i} r_{1+i,i}=0
\end{equation}
as a relation in $\hom_\Rcal(I,\emptyset)$. For $N=2$  this returns the antisymmetry of $r_{ij}$ and for 
$N=3$ the Jacobi identity; for general $N$ this identity is corresponds to the
fact that if the sum of the cyclic permutations of the Lie word  $\ad_{X_N}\ad_{X_{N-1}}\cdots \ad_{X_2}(X_1)$ is zero,  
an identity equivalent to one due to  Klyaschko \cite{klyachko}.  To see this, note that  the universal enveloping algebra
of the free Lie algebra generated by $X_1,\dots, X_N$ is the free associative Lie algebra on these generators. 
By writing out $\ad_{X_N}\ad_{X_{N-1}}\cdots \ad_{X_2}(X_1)$
in this associative algebra, we see that it takes the form $X_NY-Y X_N$ with $Y$ a sum of expressions of signed 
monomials $\pm X_{\sigma (N-1)}\cdots X_{\sigma(1)}$ with $\sigma\in \Scal_{N-1}$. 
With every monomial $\pm X_N X_{\sigma (N-1)}\cdots X_{\sigma(1)}$ that appears in 
$\ad_{X_N}\ad_{X_{N-1}}\cdots \ad_{X_2}(X_1)$, also appears a cyclic transform  with opposite sign, namely 
$\mp X_{\sigma (N-1)}\cdots X_{\sigma(1)}X_N$ and from this the assertion readily follows. 

To turn  $I$ into a principal $\ZZ/N$-set $I$ is to give a $\sigma\in\Scal(I)$ which acts transitively on $I$ and so each such 
$\sigma$ defines a relation in $\hom_\Rcal(I,\emptyset)$. These are in fact the only relations  in 
$\Rscr(I)$ that do not come from a relation already present in a  proper quotient of $I$.

We also have relations in $\hom_{\Rscr^s}(\{1,2,\dots ,N\}, \emptyset)$ for $N\ge 3$ that involve $s$, namely
\begin{equation}\label{eqn:inversion}
s_{N, N-1}r_{N-1,N-2}\cdots r_{2,1}=(-1)^N s_{1,2} r_{2,3}\cdots r_{N-1, N},
\end{equation}
where again for $N=3$ this is merely the defining property.  This follows from iterated application of the Jacobi identity and the invariance property of $s$. 
(It is perhaps amusing to note that in the case of our motivating example, we get  for $N=4$  a multiple of the curvature tensor defined by the metric that $s$ defines  on a Lie group with Lie algebra 
$\glie$.)
\end{example}

We will encounter a generalization of Cohen's model in  our setting in Subsection \ref{subsect:Liepointedcurve}.
Indeed, this interplay between configuration spaces and Lie theory is the main theme of the present paper. 
For example, we shall see that when $P$ is given as a subset of $\CC$, then the module $\Rscr_P(N)$ 
(so here $s=0$) can be identified with 
$H_{N-1}(\inj_N(\CC))\otimes\sign_N$ or $H_N(\inj_N(\CC^I\ssm P))\otimes\sign_N$, 
according to whether or not $P$ is empty. Here is a topological  example in that spirit.

\begin{example}[the configuration space of an oriented punctured manifold]\label{example:surface}
Let $M$ be an oriented (topological) manifold of finite type of dimension $m\ge 2$. 
Given a finite set $I$, then for $i,j\in I$ distinct, 
consider the space of maps $I\to \mathring M$ that separate every pair $\not=\{i,j\}$. It contains
$\inj_{I_{ij}}(\mathring M)$ as a closed submanifold with  complement $\inj_I(M)$. 
The normal bundle of the submanifold and hence the Gysin map of this pair has a coboundary map 
$H^\pt(\inj_I(M))\to H^{\pt +1-m}(\inj_{I_{ij}}M)$. The order of $i,j$ is irrelevant, 
but this is not longer the case if we twist with the sign character. Indeed,  the associated map
\[
r_{ij}: H^{(m-1)|I|}(M^I)\otimes\sign (I)\to H^{(m-1)|I_{ij}|}(M^{I_{ij}})\otimes\sign(I_{ij})
\]
changes sign if we exchange $i$ and $j$. These maps define in fact a Lie structure
\[
\Lcal_{M}: I\mapsto H^{(m-1)|I|}(\inj (M))\otimes \sign(I).
\]
If we are also given a finite subset $P\subset M$, then for every $p\in P$ and $i\in I$, we have 
a Gysin sequence for the space of injective maps $I\to M$ with the property that $I\ssm \{p\}$
maps to $M\ssm P$. The coboundary of this Gysin sequence defines
\[
r^p_i: H^{(m-1)|I|}((M\ssm P)^I)\otimes\sign (I)\to H^{(m-1)|I\ssm \{i\}|}((M\ssm P)^{I\ssm \{i\}})\otimes \sign(I\ssm \{i\})
\]
and  makes $I\mapsto  H^{(m-1)|I|}((M\ssm P)^I)\otimes\sign (I)$ a $\Rscr_P$-module $\Lcal_{(M,P)}$.  
\end{example}

We will be concerned with an analogue for $m=2$ of the preceding example in the mixed Hodge category, 
with a compact Riemann surface taking the place of $M$.
As a first step,  we will define a  Lie representation  structure associated to a complete discrete valuation ring with 
residue field $k$. We first recall  the notion of a polydifferential.

\subsection{Polydifferentials for a curve}\label{subsect:polydiff} 
Let $C$ be a nonsingular curve over $k$ and  $I$ a nonempty finite set. For every $i\in I$ we denote by 
$\pi_i :C^I\to C$ (or $\pi ^I_i$) the corresponding projection. If $\Fcal$ is a coherent sheaf on $C$, then we write 
$\Fcal^{(I)}$ for the exterior tensor product  $\otimes_{i\in I} \pi_i^*\Fcal$ in the category of coherent sheaves
(so that $\Ocal_C^{(I)}=\Ocal_{C^I}$). The action of the permutation group $\Scal (I)$ on $C^I$ clearly lifts  
to $\Fcal^{(I)}$. Note that we have a $\Scal (I)$-equivariant identification of $\Omega^{(I)}_C$ with the sign twist  
of the canonical sheaf of $C^I$: $\Omega^{(I)}_C\cong \Omega^{|I|}_{C^I}\otimes\sign(I)$. 

If we take  for $\Fcal$ the graded module $\Omega_C^\pt=\Ocal_C\oplus \Omega_C$, then  
$(\Omega^\pt_C)^{(I)}$ is what we call  the \emph{sheaf of polydifferentials}  on $C^I$. 
Observe that there is a natural decomposition 
\[ 
\textstyle \Omega^\pt_C{}^{(I)}=\bigoplus_{J\subset I}\pi_J^*\Omega_C{}^{(J)}, 
\]
where $J$ runs over all the subsets of $I$ and $\pi_J: C^I\to C^J$ is the evident projection. 
For $\alpha\in \Omega^\pt_C{}^{(I)}$, we denote by  $\alpha^J$ for its component in 
$\pi_J^*\Omega_C^{(J)}$ so that $\alpha=\sum_{J\subseteq I} \alpha^J$.  
The graded sheaf  $\Omega_C^\pt$ is, albeit in a rather trivial manner, a commutative graded 
$\Ocal_C$-algebra, and $(\Omega^\pt_C)^{(I)}$ inherits from this the structure of a 
\emph{commutative}  graded (rather than a graded-commutative) $\Ocal_{C^I}$-algebra.

\subsection{Lie module structures attached to a DVR}\label{subsect:LiemoduleDVR}
Residue operators that involve a single factor have a meaning for polydifferentials. We clarify this by first doing the previous
construction on a local ring of a curve, or rather its completion.
So we start off with a complete discrete valuation ring  $(\Ocal,\mfrak)$ with residue field $k$. 
Since $k$ is assumed to be algebraically closed of characteristic zero, a uniformizer $t$ identifies 
$\Ocal$ with $k[[t]]$. We denote its field of fractions  by $K$ and denote by  $d:K \to \omega$  
the universal  $k$-derivation that is continuous for the $\mfrak$-adic topology (in terms 
of our  uniformizer, $\omega$ is simply the one-dimensional $K$-vector space spanned by $dt$ and 
$df=f'dt$). Recall that the residue map $\res: \omega\to k$ is intrinsically defined and gives rise to a perfect duality
\[
(f, \alpha)\in  K\times \omega\mapsto \res (f\omega)\in k
\] 
of topological vector spaces. Both $K$ and $\omega$ come with natural filtrations defined by the 
valuation: $F^nK:=\mfrak^n$ and $F^n\omega=\mfrak^{n-1}d\Ocal$  (with $n\in\ZZ$) so that $d$ maps 
$F^n\Ocal$ to $F^n\omega$ and $F^nK$ and $F^{1-n}\omega$  are each others annihilator with respect to this pairing.
In particular,  $F^1\omega=\Ocal dt$ is the $\Ocal$-module of regular continuous differentials. 
We write $\DD$ for $\spec(\Ocal)$ and denote its closed point $\spec (k)\in \DD$  by $o$   and  its  generic point
$\spec(K)$ by $ \DD^\times$. 

We are going to associate with $\Ocal$  an $\Rscr^s_P$-module. To this end we 
extend the notation that we used for polydifferentials on a curve to this local setting. 
For example, given a finite set $I$, we write $\Ocal^{(I)}$  for the completed  $I$-fold tensor product  over $k$ of 
$\Ocal$ (whose maximal ideal will be denoted $\mfrak_I$) and put $\DD^{(I)}:=\spec(\Ocal^{(I)})$ 
(so this is the formal completion of $\DD^I$ at its closed point).  
For $i\in I$ we have a projection $\pi_i:\DD^{(I)}\to \DD$. When a uniformizer $t$ of $\Ocal$ has been chosen, we 
will write $t_i$ for $\pi_i^*t$. If $i, j\in I$ are distinct, then the diagonal divisor $\Delta_{ij}$ in $\DD^{(I)}$ 
defined by $t_i-t_j$ is often identified with the closed embedding  $\DD^{(I_{ij})}\hookrightarrow \DD^{(I)}$.

Likewise $K^{(I)}$ resp.\ $\omega^{(I)}$ denotes the completed  $I$-fold tensor product of $K$ and $\omega^{(I)}$. 
So $K^{(I)}$ is obtained from $\Ocal^{(I)}$ by inverting  
$\prod_{i\in I} t_i$, where $t\in \mfrak$ is a uniformizer.
Since we have not imposed a linear order on $I$, we cannot regard an element of $\omega^{(I)}$ as an 
${|I|}$-form, but we may and will regard it as a polydifferential.
Taking the  residue in the $i$th factor defines a map of $\Ocal^{(I\ssm \{ i\})}$-modules
\[
r_i : \omega^{(I)}\to \omega^{(I\ssm \{ i\})}.
\]
If $i,j\in I$ are distinct, then $r_i$ and $r_j$ commute, this in contrast to the case of differential forms, 
where they anticommute. Now consider the natural extension
\[
r_i : \omega^{(I)}(\infty \Delta_{ij})\to \omega^{(I\ssm \{ i\})},
\]
where $(\infty \Delta_{ij})$ means that we allow poles of arbitrary order along $\Delta_{ij}$. 
Then $r_i$ and $r_j$ no longer commute. To explain what happens, let 
\[
r_{ij}:=\res_{i\to j}: \omega^{(I)}(\infty \Delta_{ij})\to \omega^{(I_{ij})}
\]
be the  residue operator  that 
takes  the residue at the diagonal divisor $\Delta_{ij}$  relative to the projection which forgets the $i$th component.

\begin{lemma}\label{lemma:commutator}
On $\omega^{(I)}(\infty \Delta_{ij})$ we have $[r_i,r_j]=r_{[ij]}r_{ij}$ and (hence) $r_{ij}=-r_{ji}$. 
\end{lemma}
\begin{proof}
It suffices to check this for $I=\{i,j\}$, with $i,j$ distinct, on  a polydifferential of the form 
\[
\xi=\frac{dt_i dt_j}{t_i^{m_i+1}t_j^{m_j+1}(t_i -t_j)^\ell} \, ,
\] 
where $m_i, m_j, \ell$ are integers.  We must verify that $r_ir_j-r_jr_i=r_{[ij]}r_{ij}$.

When $\ell \le 0$, then it is clear that $r_ir_j=r_jr_i$  and $r_{ij}=0$. So let us assume that $\ell >0$. For a similar reason,
we can assume that at least one of $m_i$ and $m_j$ is $\ge 0$.  Let us assume $m_i\ge 0$.

In order to compute $r_{ij}(\xi)$, we substitute $(t_i,t_j)=(t+s,t)$, so that we have to take the residue with respect to $s$. 
We find
\[
\xi=\frac{dt ds}{(s+t)^{m_i+1} t^{m_j+1}s^\ell}=\frac{dt ds}{t^{m_i+m_j+2}s^\ell (1+s/t)^{m_i+1}}=
\frac{dt ds}{t^{m_i+m_j+2}s^\ell}\sum_{n\ge 0} \tbinom{n+m_i}{n}(-s/t)^n.
\]
Since $r_{ij}(\xi)$ is the coefficient of $s^{-1}ds$, we see only the term $n=\ell -1$ contributing, giving
\[
r_{ij}(\xi)=(-1)^{\ell-1}\tbinom{\ell+m_i-1}{\ell-1}\frac{dt}{t^{m_i+m_j+\ell+1}}.
\]
So $r_{[ij]}r_{ij}(\xi)$ is zero unless $m_i+m_j+\ell=0$ (so that $m_j\le 0$) and in that case, equal to 
$(-1)^{\ell-1}\tbinom{\ell+m_i-1}{\ell-1}$.

In order that $r_jr_i(\xi)$ be nonzero, one needs that $m_j\ge 0$ and then a straightforward computation yields that its value 
can only be  nonzero only when  $m_i+m_j+\ell=0$,  a case that is precluded, because we assumed 
$\ell>0$ and $m_i>0$.
Similarly, $r_ir_j(\xi)$ is nonzero only when $m_i\ge 0$, $m_i+m_j+\ell=0$ and we find that the value is then 
$(-1)^\ell\binom{\ell+m_i-1}{\ell -1}$.  Hence 
$r_ir_j(\xi)-r_jr_i(\xi)= (-1)^{\ell-1}\binom{\ell+m_i-1}{\ell -1}= r_{[ij]}r_{ij}(\xi)$ 
as desired. \end{proof}

We focus on the case when the pole order along each diagonal hypersurface  is at most two. We will associate three types 
of residues with this situation. Let us first describe these in the basic case where we have only two factors, 
so that $I$ consists of two distinct elements $i,j$.
Let $\sigma$ be the  transposition involution in $\Ocal^{(I)}$ which interchanges the factors. Then 
$\Ocal^{(I)}(2\Delta_{ij})/\Ocal^{(I)}$ inherits that action. Its `polar degree one' submodule
$\Ocal^{(I)}(\Delta_{ij})/\Ocal^{(I)}$ is precisely the $(-1)$-eigensubmodule  of $\sigma$, so that the eigendecomposition of
$\Ocal^{(I)}(2\Delta_{ij})/\Ocal^{(I)}$ splits it into a direct sum of two  free $\Ocal$-modules of rank one:
\[
\Ocal^{(I)}(2\Delta_{ij})/\Ocal^{(I)}=\Ocal^{(I)}(\Delta_{ij})/\Ocal^{(I)}\oplus
\big(\Ocal^{(I)}(2\Delta_{ij})/\Ocal^{(I)}\big)^{\sigma},
\]
We also get such a splitting of  $\omega^{(I)}(2\Delta_{ij})/\omega^{(I)}$:
\[
\omega^{(I)}(2\Delta_{ij})/\omega^{(I)}=
\omega^{(I)}(\Delta_{ij})/\omega^{(I)}\oplus
\big(\omega^{(I)}(2\Delta_{ij})/\omega^{(I)}\big)^{\sigma}.
\]
The residue $r_{ij}$ as defined above and  restricted to $\omega^{(I)}(2\Delta_{ij})$ of course
factors through this quotient.
We then define the linear map
\[
(r'_{ij}, r''_{ij}): \omega^{(I)}(2\Delta_{ij})\to 
\omega\oplus k
\]
as the reduction 
\[
\omega^{(I)}(2\Delta_{ij})\to \omega^{(I)}(2\Delta_{ij})/\omega^{(I)}=
\omega^{(I)}(\Delta_{ij})/\omega^{(I)}\oplus
\big(\omega^{(I)}(2\Delta_{ij})/\omega^{(I)}\big)^{\sigma}
\]
post-composed with the appropriate residue maps: on the first summand we let this be $r_{ij}$ and on the second 
summand $r_{ij}$ followed by the residue in  $o$. It follows that the restriction of  $[r_i,r_j]=r_{[ij]} r_{ij}$ to  
$\omega^{(I)}(2\Delta_{ij})$ is also equal to 
$r_{[ij]}r'_{ij} +r''_{ij}$.

The third type of residue is the \emph{biresidue}: it is the map $\br: \omega^{(I)}(2\Delta_{ij})\to K$ 
that in terms of a uniformizer sends 
$g(t_i,t_j)(t_i-t_j)^{-2}dt_idt_j$ to $g(t,t)$.  It is known (and straightforward to 
verify) that this notion is coordinate invariant.

\begin{example}\label{example:basic2}
We explicate these residue operators in terms of a uniformizer. Let $\alpha\in \omega^{(I)}(2\Delta_{ij})$ and write 
$\alpha$ as a sum compatible with the decomposition above: 
\[
\alpha=\alpha'+\alpha''\text{ with } \alpha'=\frac{f(t_i,t_j)dt_idt_j}{t_i-t_j} 
\text{  and  }  \alpha''=\frac{g(t_i,t_j)dt_idt_j}{(t_i-t_j)^2},
\] 
where $f,g$ lie in  $K^{(I)}$ and $g(t_i,t_j)=g(t_j,t_i)$. 
It is clear that $\br(\alpha)=\br(\alpha'')=g(t,t)$  and that 
\[
r'_{ij}(\alpha)=r_{ij}(\alpha')=
\res_{i\to j}\frac{f(t_i,t_j)dt_idt_j}{t_i-t_j}= f(t,t)dt,
\]
which ought to be viewed as a differential on  $\Delta_{ij}$, and we find that 
\[
r''_{ij}(\alpha)=\res_j\res_{i\to j}\frac{g(t_i,t_j)dt_idt_j}{(t_i-t_j)^2}=
\res_j \big(\res_{i\to j}\frac{g(t_i,t_j)dt_i}{(t_i-t_j)^2}\big)dt_j=\res \frac{\partial g}{\partial t_i}(t,t)dt. 
\]
\end{example}

\begin{lemma}\label{lemma:resformula}
Assume that in Example \ref{example:basic2}, $g$ is a constant (in $k$).  If $f_i, f_j\in K$, then
\[
\res\res_{i\to j} f_i(t_i)f_j(t_j)\alpha= \res f_if_j r'_{ij}(\alpha)+g\res (f_jdf_i),
\]
where we recall that Lemma \ref{lemma:commutator} tells us that the right hand side is also $[r_i, r_j](f_i(t_i)f_j(t_j)\alpha)$.
\end{lemma}
\begin{proof}
This is straightforward, for then 
\[
\res_{i\to j} f_i(t_i)f_j(t_j)\alpha = f_i(t)f_j(t)f(t,t)dt +g\frac{\p f_i(t)}{\p t_i} f_j(t)dt=
f_if_j r'_{ij}(\alpha) +g f_jdf_i
\]
and it remains to apply $\res$.
\end{proof}

When $I:=\{i,j\}$, we define the \emph{module of quasi-logarithmic bidifferentials,} denoted 
$\omega^{(I)}\la \qlog\ra$,  to  be the $\Ocal^{(I)}$-submodule of $(F^0\omega)^{(I)}(2\Delta_{ij})$ of forms on which 
both  $r_i$ and $r_j$ take values in $F^0\omega$. This 
means that in terms of a uniformizer an element of  $\omega^{(I)}\la\qlog\ra$ 
can be written as a sum of bidifferentials which have a denominator 
$t_it_j$, $t_i(t_i-t_j)$, $t_j(t_i-t_j)$ or $(t_i-t_j)^2$ (but the identity 
$(t_it_j)^{-1}= t_j^{-1}(t_i-t_j)^{-1}- t_i^{-1}(t_i-t_j)^{-1}$ shows that we can in fact eliminate $t_it_j$ as a denominator), 
equivalently, that in the above decomposition of $\alpha$, we have $f\in t_i^{-1}\Ocal^{(I)}+ t_j^{-1}\Ocal^{(I)}$ and 
$g\in\Ocal^{(I)}$ (more specifically, we could  take 
$g\in k[[t_i+t_j]]$). So the biresidue takes $\omega^{(I)}\la\qlog\ra$ to $\Ocal$.  

We define $\omega^{(I)}\la\qlog\ra_o$ to be $k$-subspace of $\omega^{(I)}\la\qlog\ra$ for which the biresidue is 
constant. In terms of the preceding decomposition of $\alpha$, this means that  we can choose $g$ to be a constant.
We denote the resulting $k$-linear form by $s: \omega^{(I)}\la\qlog\ra_o\to k$.

We also define the \emph{module of logarithmic bidifferentials} $\omega^{(I)}\la \llog\ra$ as the intersection of 
$\omega^{(I)}\la\qlog\ra$ with $\omega^{(I)}(\Delta_{ij})$. 
In this case we can do with the denominators $t_i(t_i-t_j)$ and $t_j(t_i-t_j)$ and so, as the terminology and notation 
suggest,  this module consists of logarithmic bidifferentials.
 
We  generalize these  notions to the case where $(i,j)$ is a distinct pair taken from any finite set $I$.
Let us denote by $\Delta_I$ the  union of all the diagonal divisors in $\DD^{(I)}$, but  let us also agree to 
abbreviate $\omega^{(I)}(n\Delta_I)$ simply by $\omega^{(I)}(n)$. The above definition extends  in an obvious way 
to give on $\omega^{(I)}(2)$ for every ordered, distinct  pair $(i,j)$ in $I$ 
a  biresidue map $\br_{ij}$  and a linear map
\[
(r'_{ij}, r''_{ij}): \omega^{(I)}(2)\to 
\omega^{(I_{ij})}(2)\oplus \omega^{(I\ssm\{ i,j\})}(2)
\]
with the property that 
\[
[r_i,r_j]=r_{[ij]}r_{ij}=r_{[ij]}r'_{ij} +r''_{ij}.
\] 
The definition of $\omega^{(I)}\la\qlog\ra$ and $\omega^{(I)}\la\qlog\ra_o$  is however not entirely straightforward 
as we want them to be stable under the residue maps $r_i$, $r_{ij}$ and $\br_{ij}$ that we have introduced and so the 
most logical way to introduce them is to force this property by  an inductive definition.

\begin{definition}\label{def:}
The $\Ocal^{(I)}$-module of \emph{quasi-logarithmic polydifferentials} $\omega^{(I)}\la\qlog\ra$ (resp.\ the space of 
\emph{special quasi-logarithmic polydifferentials}) $\omega^{(I)}\la\qlog\ra_o$ is for $|I|$ a singleton given as 
$F^0\omega$, for $|I|=2$ as above and for $|I|>2$ inductively defined as the $\Ocal^{(I)}$-submodule (resp.\ $k$-subspace) of
$\alpha\in \omega^{(I)}(2)$ with the property that 
\begin{enumerate}
\item[(i)] $r_i(\alpha)\in \omega^{(I\ssm\{i\})}\la\qlog\ra$ (resp.\  $r_i(\alpha)\in \omega^{(I\ssm\{i\})}\la\qlog\ra_o$),
\item[(ii)] $r_{ij}(\alpha)\in\omega^{(I_{ij})}\la\qlog\ra$ (resp.\ $r_{ij}(\alpha)\in\omega^{(I_{ij})}\la\qlog\ra_o$),
\item[(iii)] $\br_{ij}(\alpha)\in \Ocal^{(I_{ij})}\omega^{(I\ssm\{i,j\})}\la\qlog\ra$ (resp.\  
$\br_{ij}(\alpha)\in \omega^{(I\ssm\{i,j\})}\la\qlog\ra_o$, which we then
also denote as $s_{ij}(\alpha)$).
\end{enumerate}
The $\Ocal^{(I)}$-module of \emph{logarithmic polydifferentials}
$\omega^{(I)}\la\llog\ra$ is the intersection of $\omega^{(I)}\la\qlog\ra$ with $\omega^{(I)}(1)$.
\end{definition}

Perhaps we should explain that in part (iii) of this definition we use the pull--back along the projection 
$\DD^{(I_{ij})}\to \DD^{(I\ssm\{i,j\})}$ to identify polydifferentials on the target 
with polydifferentials on the source.

The following two examples may help to understand how restrictive our conditions are.

\begin{example}\label{example:triangle}
Consider
\[
\alpha:=\frac{fdt_1dt_2dt_3}{(t_1-t_2)(t_2-t_3)(t_3-t_1)}.
\] 
with $f\in K^{(3)}$. We claim that $\alpha$ is quasi-logarithmic if and only  if  $f\in \Ocal^{(3)}$  and is in 
addition special quasi-logarithmic precisely when the 
restriction of of $f$ to the main diagonal is constant.
Indeed, in order that the residues of $\alpha$ along each coordinate hyperplane are quasi-logarithmic it is 
necessary that  $f\in \Ocal^{(3)}$. This also suffices, for 
example
\[
r_{12}(\alpha)=- \frac{f(t_{12}, t_{12}, t_3)dt_{12} dt_3}{(t_{12}-t_3)^2}
\]
is then clearly quasi-logarithmic.  
Note that its biresidue on the diagonal $t_{12}=t_3$ is the restriction of $f$ to the main diagonal. 
Hence for $\alpha$ to be be special quasi-logarithmic, $f$  must be constant on the main diagonal $t_1=t_2=t_3$. 
This also suffices.
\end{example}

\begin{example}\label{example:notallowed}
Consider
\[
\beta:=\frac{fdt_1dt_2dt_3}{(t_1-t_2)^2(t_3-t_2)}.
\] 
with $f\in \Ocal^{(3)}$. We claim that $\beta$ is quasi-logarithmic if and only if  $f\in (t_1-t_2, t_2-t_3)$ (so that  
$\beta$ can then be written as a sum of two polydifferentials with denominators 
$(t_1-t_2)(t_3-t_2)$ and $(t_1-t_2)^2$). For  if $\beta $ is quasi-logarithmic, then 
\[
\br_{12}(\beta)= \frac{f(t_{12},t_{12}, t_3)dt_3}{t_3-t_{12}}
\]
must lie in $t_3^{-1}\Ocal^{(2)}dt_3$. This can only happen when $t_3-t_{12}$ divides $f(t_{12},t_{12}, t_3)$, 
which is indeed equivalent to: $f\in (t_1-t_2, t_2-t_3)$. Conversely, a form of this type is clearly quasi-logarithmic.
\end{example}

\subsubsection*{Graphical representation of the polar divisors}\label{rem:graphical}
The mechanics of this  inductive definition is perhaps best understood with the help of graphs.
Let $\G$ be a graph whose vertex set is the disjoint union of $I$ and  an element that we denote by $\star$,  
is without loops, but may be nonreduced in the sense that multiple edges are allowed. 
With an edge of  $\G$ connecting $i\in I$ with $j\in I\ssm \{i\}$ resp.\ $\star$ we associate the diagonal divisor 
$\Delta_{ij}$ resp.\ the hyperplane divisor 
$\pi_i^{-1}(o)$ and we denote by $\Delta_\G$ the divisor that we get by taking the sum over all the edges of $\G$. 
Note that $r_i$ resp.\ $r_{ij}$  is  zero on $\omega^{(I)}(\Delta_\G)$ unless $\G$ has an edge $e$ that connects 
$i$ with $\star$ resp.\ $j$, in which
case $\omega^{(I)}(\Delta_\G)$ is mapped onto $\omega^{(I)}(\Delta_{\G/e})$. 
We must be careful with these graphical representations though, even if we stick
to reduced graphs, as there are relations of the following type:
if $J\subset I\sqcup \{\star\}$ ia s subset with $|J|\ge 3$ and $\sigma\in\Scal(J)$  
a $|J|$-cycle which makes $J$ the vertex set of a
polygon $\G'$ (so that 
$\{j, \sigma (j)\}$ is an edge of $\G'$ for every $j\in J$), then 
\[
\textstyle \sum_{j\in J} \prod_{j\in J\ssm\{ i\}} (t_j-t_{\sigma(j)})^{-1} \prod_{i\in J}dt_j=0, 
\]
where we define $t_\star$ to be identically zero. 
This is visibly a relation in $\sum_j \omega^{(J)}(\Delta_{\G'_j})$, where $\G'_j$ is the chain obtained 
from $\G'$ by removing the edges containing $j$.
If we multiply this with an element of $\omega^{(I\ssm J)}(\G'')$, where $\G''$ is a reduced graph of the above type 
with vertex set $(I\sqcup \{\star\})\ssm J$, then we get a relation involving reduced graphs with vertex set 
$I\sqcup\{\star\}$. These are essentially the only type of relations that involve reduced polar divisors.

The following proposition appears already in a somewhat different form in \cite{sv}, \S 6.

\begin{proposition}\label{prop:graphical1}
The module $\omega^{(I)}\la\llog\ra$ of  logarithmic polydifferentials is a sum of submodules of the type 
$\omega^{(I)}(\Delta_\G)$, where $\G$ runs over the 
collection of graphs $\G$ with vertex set $I\sqcup\{\star\}$ of which each connected component of $\G$  
is a chain (possibly of zero length) and with $\star$ appearing as the end of some chain. 
\end{proposition}
\begin{proof}
This is clear when $|I|\le 1$. We proceed with induction on the size of $I$. If $|I|\ge 2$, 
then it follows from the preceding
that if $\eta\in\omega^{(I)}\la\llog\ra$, then $\eta$ is contained in a sum of submodules 
$\omega^{(I)}(\G)$, where $\G$ is a reduced
graph with the property for every edge $e$ of $\G$, each connected component of $\G/e$  is a chain and 
$\star$ is an end point of that chain.

Every triangle spanned by a subset of  $I\sqcup \{\star\}$ gives rise to a relation:  
if $\{i,j\}$ and $\{j,k\}$ are edges  of $\G$, then if  
$\G'$ and $\G''$ are the graphs obtained by removing the  bond $\{i,j\}$ resp.\ $\{j,k\}$ and replacing it by $\{i,k\}$, 
we have  $\omega^{(I)}(\Delta_\G)\subset \omega^{(I)}(\Delta_{\G'})+\omega^{(I)}(\Delta_{\G''})$. 
Successive application of this procedure shows that we can in fact take our $\G$ in the collection of the asserted type.
\end{proof}

We can do something similar for the quasi-logarithmic polydifferentials. 
Suppose $\sigma$ is a permutation of $I$. Such a permutation is graphically 
represented by collection of pairwise disjoint oriented polygons whose vertex set is a subset of $I$. 
If we fix a uniformizer $t\in \Ocal$, then with such a $\sigma$ we associate the
polydifferential
\[
 \eta_\sigma:=\prod_{i\in I} \frac{dt_i}{t_i-t_{\sigma (i)}}, 
\]
If we replace $\sigma$ by its inverse on a $\sigma$-invariant subset of $I$, 
then some of the polygons get an opposite orientation and $\eta_\sigma$ may change sign. So 
if $\supp(\sigma)$ is the set of $i\in I$ with $\sigma(i)\not=i$ and if $\G_\sigma$ denotes the graph 
(a disjoint union of  polygons) defined by $\sigma$  on $\supp(\sigma)$, then 
\[
\omega^{(I\ssm \supp (\sigma))}\la \log \ra \cdot\eta_\sigma
\]
is a subspace of $\omega^{(I)}\la\qlog\ra_o$ which only depends on $\G_\sigma$. 
We therefore denote it by $\omega^{(I)}\la\qlog\ra_{\Gamma_\sigma}$. The natural map 
\[
\oplus_{\G}\, \omega^{(I)}\la\qlog\ra_\G \to \omega^{(I)}\la\qlog\ra_o, 
\]
where in the left hand side is the sum over all graphs $\G$ with vertex set contained in $I$ and whose connected 
components are polygons with $\ge 2$ vertices, is easily seen to be surjective. It is in general not injective. 
For example, if we put
\[
\alpha_n(z_1,z_2, z_3, \cdots , z_n):= \frac{dz_1dz_2\cdots dz_n}{(z_n-z_{n-1})(z_{n-1}-z_{n-2})\cdots (z_2-z_1)(z_1-z_n)}
\]
then  $\alpha_4(z_1,z_2, z_3, z_4) +\alpha_4(z_2,z_1, z_3, z_4)+ \alpha_4(z_1,z_3,z_2, z_4)=0$ 
and this relation involves the polygons defined by
the $4$-cycles $(1,2,3,4,1)$, $(2,1,3,4)$ and $(2,1,3,4)$. So there is in general no uniqueness in representing a 
quasi-logarithmic polydifferential as a sum with terms in the
$\omega^{(I)}\la\qlog\ra_\G$. All we can say is that any relation among the 
$\Scal_n$-transforms of $\alpha_n$ the operators
$s_{21}r_{32}\cdots r_{n-2,n-1}r_{n,n-1}$ and their $\Scal_n$-transforms must take the value zero.
\\

\begin{proposition}\label{prop:F_K}
 If $\alpha \in \omega^{(I)}\la\qlog\ra_o$ and $f=\otimes_{i\in I} f_i\in K^{(I)}$ (with $f_i\in K$), then
\begin{gather*}
[r_i,r_j](f\alpha)=  f^{(ij)}\big( r_{[ij]}(\pi_{[ij]}^*(f_if_j)r'_{ij}(\alpha))+\res (f_jdf_i)s_{ij}(\alpha)\big),
\end{gather*}
where  $\pi_{[ij]}: \DD^{(I_{ij})}\to \DD$ is the coordinate with index 
$[ij]$ and $f^{(ij)}:=\otimes_{l\in I\ssm\{i,j\}} f_l \in K^{(I\ssm \{i,j\})}$.

Furthermore, the assignment $I\mapsto \omega^{(I)}\la\llog\ra$  resp.\  $I\mapsto \omega^{(I)}\la\qlog\ra_o$ together 
with the contractions 
$r_{ij}$, $r_i$ (and $s_{ij}$)  define  an $\Rscr_*$-module $\Lcal og_K$ resp.\ an $\Rscr_*^s$-module $\qLog_K$.
\end{proposition}
\begin{proof}
If $i,j,l\in I$ are distinct and $\alpha\in \omega^{(I)}(\Delta_\G)_o$, then $s_{[ij],l}r_{ij}(\alpha)=0$ unless 
$J:=\{i,j,l\}$ is a  triangle  component of $\G$, so  that  Example \ref{example:triangle} appears as a factor. 
This shows that then $s_{[ij],l}r_{ij}$ is as an element of $\omega^{(I)}(\Delta_{\G\ssm J})_o$ invariant under 
cyclic permutations. In particular, the contractions $s_{ij}$ satisfy the invariant form properties.

The remaining  identities only need to be verified in case $I=\{1,2\}$. This we already did (see 
Lemma \ref{lemma:resformula}), except for  the Lie property (a Jacobi identity). 
But this is well-known and follows from 
$ \sum_{i\in \ZZ/3} (t_{i-1}-t_i)^{-1}(t_i-t_{i+1})^{-1}=0$.
\end{proof}

\subsection*{A homomorphism of $\Rscr^s$-modules}
Suppose $\glie$ is a Lie algebra over $k$ endowed with a symmetric bilinear form $s:\glie\otimes\glie\to k$ invariant 
under the adjoint representation. We then have defined
the $\Rscr^s$-module $\Lcal^{s}_{\gfrak}$.  Here is an example of a homomorphism of 
$\Rscr^s$-modules $\Lcal^{s}_{\gfrak}\to \qLog_K$ that we will later meet again, 
but that will depend on the choice of  a uniformizer $t$ for $\Ocal$.  
For such a choice  we define $\alpha_N\in \omega^{(N)}\la \qlog\ra $ and $s_N\in (\glie^{\otimes N})^*$ by  
\[
\alpha_N:= \frac {dt_1dt_2\cdots dt_N}{(t_N-t_{N-1})\cdots (t_2-t_1)(t_1-t_N)}, 
\quad  s_N(X_N\otimes\cdots \otimes X_1):=s([[\cdots [X_N, X_{N-1}]\cdots , X_2], X_1).
\]
So $\alpha_2=-(t_2-t_1)^{-2}dt_1dt_2\in \omega^{(2)}\la \qlog\ra_o$ and $s_2(X_2\otimes X_1)=s(X_2,X_1)$.
In fact,  it it easy to check that $\alpha_N\in \omega^{(N)}\la \qlog\ra_o$ so that 
$s_N\otimes \alpha_N$ defines an element of $\Hom (\glie^{\otimes N},\omega^{(N)}\la \qlog\ra_o)$.

Note that both $s_N$ and $\alpha_N$  are multiplied by $-1$ under the transposition $(N,N-1)$ and by  $(-1)^N$  
under the permutation $\sigma_o$ which reverses the order (defined by  $\sigma_o(i)=N+1-i$).   
\begin{lemma}\label{lemma:etares}
For $N=1,2,\dots$, let 
\[
\textstyle \eta_N:=-\sum_{\sigma\in \Scal_N} \sigma_*(s_N\otimes \alpha_N)\in 
\Hom (\glie^{\otimes N},\omega^{(N)}\la \qlog\ra). 
\]
Then for 
$r_{N, N-1}\eta_N( X_N\otimes\cdots \otimes X_1)$ is zero for $N\le 2$ and is for $N\ge 3$ equal to 
$(\eta_{N-1} ([X_N,X_{N-1}]\otimes X_{N-2}\otimes\cdots \otimes X_1)$.
On the other hand,  $\br_{N-1,N} \eta_N$ is zero unless unless $N=2$ in which case  it is equal to $s$.  
\end{lemma}
\begin{proof}
We compute $r_{N,N-1}\sigma_*(s_N(X_N\otimes\cdots\otimes X_1)\sigma_*\alpha_N$. This is nonzero only if
$\{\sigma^{-1}N, \sigma^{-1}(N-1)\}$ is  a neighboring pair (we here regard the indexing  as by  $\ZZ/N$). 
It is then clear that  if we sum over all $\sigma\in\Scal_N$, we obtain  an element of  
$\Hom(\glie^{\otimes (N-1)}, \omega^{(N-1)}\la \qlog\ra)^{\Scal_{N-1}}$. The terms in which $\alpha_{N-1}$ 
appears are of the type
\[
s_N(X_{N-2}\otimes X_{N-3}\otimes\cdots X_{i}\otimes X_N\otimes X_{N-1}\otimes 
X_{i-1}\otimes\cdots\otimes X_1)-s_N(X_{N-2}\otimes X_{N-3}\otimes\cdots X_{i}\otimes X_{N-1}\otimes X_{N}\otimes 
X_{i-1}\otimes\cdots\otimes X_1)
\]
with $i\in\ \ZZ/N$, where for $i\equiv N\equiv 0$ this expression must be interpreted  as 
$f(X_{N-1}\otimes\cdots \otimes X_1\otimes X_N)- (X_{N}\otimes\cdots \otimes X_1\otimes X_{N-1})$. 
But now note that this difference is 
\[
f_{N-1}(X_{N-2}\otimes X_{N-3}\otimes\cdots X_{i}\otimes [X_N,X_{N-1}]\otimes X_{i-1}\otimes\cdots\otimes X_1),
\]
even for $i=N$. We thus find that  $r_{N,N-1}\eta_N(X_N\otimes\cdots \otimes X_1)=\eta_{N-1} ([X_N,X_{N-1}]
\otimes X_{N-2}\otimes\cdots \otimes X_1)$.

The last clause is obvious.
\end{proof}

We define $\eta_I\in \Hom(\glie^{\otimes I},\omega^{(I)}\la \qlog\ra)$  
for an arbitrary finite set $I$ via a bijection $I\cong\{1, \dots, N\}$, the result being 
independent of this bijection because  $\eta_N$ is $\Scal_N$-invariant.
If $I'\subset I$ is a subset with complement $I''$, then interpret  $\eta_{I'}\otimes \eta_{I''}$ as an element of 
$\Hom(\glie^{\otimes I},\omega^{(I)}\la \qlog\ra)$.
More generally this makes sense  for any partition $\Pcal$ of $I$: the tensor  $\otimes_{P\in\Pcal} \eta_P$ can be 
regarded as an element of $\Hom(\glie^{\otimes I},\omega^{(I)}\la \qlog\ra)$.

\begin{corollary}\label{cor:etahat}
A homomorphism of $\Rscr^s$-modules $\Lcal^{s}_{\gfrak}\to \qLog_K$ is defined by assigning to a finite set 
$I$ the homomorphism
\[
\textstyle \hat\eta_I:=\sum_{\Pcal | I} \otimes_{P\in\Pcal} \eta_P\in \Hom(\glie^{\otimes I},\omega^{(I)}\la \qlog\ra),
\]
where the sum is over all partitions of $I$.
\end{corollary}
\begin{proof}
This in indeed a straightforward consequence of  Lemma \ref{lemma:etares}.
\end{proof}

Let $V$ be a $\glie$-representation. A  consequence of the preceding is that 
$ \hom_{\Rscr^s_*}(\Lcal^{s}_{\gfrak,V},\qLog_K)$ is in a sense a deformation of 
$\hom_{\Rscr_*}(\Lcal_{\gfrak,V},\Lcal og_K)$:

\begin{corollary}\label{cor:etadef}
Let $\xi\in \hom_{\Rscr_*}(\Lcal_{\gfrak,V},\Lcal og_K)$ (so this assigns to every finite set $I$ a linear map 
$\xi_I: \glie^{\otimes I}\otimes V\to \omega^{(I)}\la\llog\ra$ subject to the usual conditions involving the 
contraction operators $r_i$ and $r_{ij}$). Then a  $\Rscr_*^s$-homomorphism  
$\hat\xi: \Lcal^{s}_{\gfrak,V}\to \qLog_K$ is defined by assigning 
to every finite set $I$,  the  sum 
$\sum_{J\subset I} \xi_{I\ssm J}\otimes \hat\eta_J:  \glie^{\otimes I}\otimes V\to \omega^{(I)}\la\qlog\ra$. 
The resulting map 
\[
\xi\in \hom_{\Rscr_*}(\Lcal_{\gfrak,V},\Lcal og_{K})\mapsto \hat\xi \in \hom_{\Rscr_*^s}(\Lcal^{s}_{\gfrak,V},\qLog_K)
\]
is an isomorphism. 
\end{corollary}
\begin{proof}We compute the image of $\hat\xi_I=\sum_{J\subset I} \xi_{I\ssm J}\otimes \hat\eta_J$ under the 
operators $r_i, r_{ij}$ and $s_{ij}$, where $i,j\in I$ are distinct.
Since  $r_i$ is zero on $\hat\eta$, it follows from the defining formula  that $r_i(\hat\xi)=\widehat{r_i (\xi)}$.
If we apply $r_{ij}$ to $\hat\xi_I$, then the image of the term $\xi_{I\ssm J}\otimes \hat\eta_J$ is zero unless $\{i,j\}$ 
is contained in either $I\ssm J$ or $J$ and this then produces $r_{ij}(\xi_{I\ssm J})\otimes \hat\eta_J= 
\xi_{(I\ssm J)_{ij}}\otimes \hat\eta_J$ resp.\  $\xi_{I\ssm J}\otimes r_{ij}(\hat\eta_J)=
\xi_{I\ssm J}\otimes  \hat\eta_{J_{ij}}$. Hence  $r_{ij}(\hat\xi_I)=\widehat{\xi_{I_{ij}}}$. 
The polydifferential $s_{ij}( \xi_{I\ssm J}\otimes \hat\eta_J)$
is nonzero only when $\{i,j\}\subset J$ and is then equal to $\xi_{I\ssm J}\otimes s_{ij}(\hat\eta_J)$. 
With the help of this identity, one finds that 
the $s_{ij}$ interact with $r_{ij}$ as to satisfy the invariant form properties. 
We conclude that $\hat\xi \in \qLog_K$ is as asserted.

The map $\xi\mapsto \hat\xi$ is evidently $k$-linear. We construct its inverse  with induction, that is, we assume given 
$\tilde\xi\in hom_{\Rscr_*^s}(\Lcal^{s}_{\gfrak,V},\qLog_K)$ and that we for an integer $N\ge 0$, we succeeded in finding 
$\xi_0, \dots , \xi_N$ which in this range define an element of $\hom_{\Rscr_*}(\Lcal_{\gfrak,V},\Lcal og_{K})$  for which  
$\hat\xi_i=\tilde\xi_i$ for $i=0, \dots,  N$.  
We can take  $\xi_0=\tilde\xi_0$ and $\xi_1=\tilde\xi_1$  and so we assume that $N\ge 1$. 
The map $\xi_{N+1}$ that we  must construct has its polar part prescribed in terms of the 
$\xi_0, \dots, \xi_{N-1}$. We choose a  map $\xi'_{N+1}$ with this polar part, leaving us the freedom to add to 
$\xi'_{N+1}$ a map which takes its values in the regular polydifferentials of degree $N+1$. 
One such map is $\xi''_{N+1}:=\tilde \xi_{N+1}-\hat\xi'_{N+1}$, which indeed has no polar part. 
So $\xi_{N+1}=\xi'_{N+1}+\xi''_{N+1}$ has the same polar part as $\xi'_{N+1}$, but is in addition such that 
$\hat\xi_{N+1}=\tilde\xi_{N+1}$. This completes the induction step.
\end{proof}

The preceding extends to the case where we have a product of discrete valuation rings whose factors are indexed by 
a nonempty finite set $P$:
$\Ocal_P=\prod_{p\in P} \Ocal_p$. We let $K_p$ denote the quotient field of $\Ocal_p$ so that  
$K_P:=\prod_{p\in P} K_p$ is the fraction ring of $\Ocal_P$. We also put  $\omega_P:=\oplus_{p\in P}\omega_p$. 
The (perfect) residue pairings $\omega_p\times K_p\to k$ of topological vector spaces  combine to give pairing
\[
\textstyle (\alpha,f)\in \omega_P\times K_P\mapsto \sum_{p\in P}\res_p f\alpha\in k
\]
that is topologically perfect as well. We let $\omega_P^{(I)}\la\qlog\ra:=\oplus_{p\in P}
\omega_p^{(I)}\la\qlog\ra $ and similar for the other modules introduced above. 
We then get for every $p\in P$ and $i\in I$ a residue
map $r^p_i$ and we find the following $P$-fold variant of the last clause of Proposition \ref{prop:F_K}.

\begin{corollary}\label{prop:F_KP}
With these contractions $r_{ij}$, $r^p_i$ (and $s_{ij}$) we obtain an $\Rscr_P$-module $\Lcal_{K_P}: \Rscr_P\to \Vect_k$, $I\mapsto \omega_P^{(I)}\la\llog\ra$ and an $\Rscr^s_P$-module
$\qLog_{K_P}: \Rscr^s_P\to \Vect_k$, $I\mapsto \omega_P^{(I)}\la\qlog\ra_o$, both
in the category of $k$-vector spaces $\Vect_k$.
\end{corollary}

\begin{remark}\label{rem:}
It is likely that  the induced functor $k\otimes \Rscr_P\to \Vect_k$ is essentially faithful.
\end{remark}

\subsection{Lie  module structures attached a pointed curve}\label{subsect:Liepointedcurve}
Now let $C$ be a nonsingular projective irreducible curve of genus $g$ and  $P\subset C$ a  finite (possibly empty) subset.

We define inductively the $\Ocal_{C^I}$-submodule $\Omega_C( P)^{(I)}\la\qlog\ra\subset \Omega_C( P)^{(I)}(2)$ 
as in the local case: it is the largest submodule for which each residue taken along $\pi_i^{-1}(p)$ lands in
$\Omega_C( P)^{(I\ssm\{i\})}\la\qlog\ra$. There is no need to introduce a sheaf 
$\Omega_C( P)^{(I\ssm\{i\})}\la\qlog\ra_o$ for the following reason.   
The biresidue of a \emph{global} section of $\Omega_C( P)^{(I)}\la\qlog\ra$ along the diagonal $\Delta_{ij}$ will be 
regular along the fibers of the projection 
$C^{I_{ij}}\to C^{I\ssm\{i,j\}}$ in the sense that it will be a global section of the coherent 
pull-back of 
$\Omega_C( P)^{(I\ssm\{i,j\})}\la\qlog\ra$. 
Since that projection 
is a trivial bundle with connected proper fiber (namely $C$), such a global section can be regarded as  
an element of $ H^0(C^{I\ssm\{i,j\}},\Omega_C( P)^{(I\ssm\{i,j\})}\la\qlog\ra)$. So we have defined a linear map
\[
s_{ij}: H^0(C^I,\Omega_C( P)^{(I\ssm\{i\})}\la\qlog\ra)\to H^0(C^{I\ssm\{i,j\}},\Omega_C( P)^{(I\ssm\{i,j\})}\la\qlog\ra).
\]
The definition of  the subsheaf $\Omega_C( P)^{(I)}\la\llog\ra\subset \Omega_C( P)^{(I)}(\Delta_I)$ of logarithmic 
 polydifferentials is similar. The other contraction operators that we introduced for polydifferentials associated to a 
 complete DVR also have an evident global counterpart.
\begin{gather*}
r^p_i: H^0(C^I, \Omega_C( P)^{(I)}\la\qlog\ra)\to H^0(C^{I\ssm\{i\}}, 
\Omega_C( P)^{(I\ssm \{i\})}\la\qlog\ra) \quad (i\in I, p\in P),\\ 
r_{ij}: H^0(C^I, \Omega_C( P)^{(I)}\la\qlog\ra)\to 
H^0(C^{I_{ij}}, \Omega_C( P)^{(I_{ij})}\la\qlog\ra)\quad (i,j\in I \text{ distinct}).
\end{gather*}
These are not independent: given $i\in I$, then for every $j\in I\ssm\{i\}$, 
the  projection $C^I\to C^{I\ssm\{i\}}$ maps each 
$C^{I_{ij}}$ isomorphically onto $C^{I\ssm\{i\}}$. 
Via this identification, the residue theorem (taken along the fibers of this projection) then
yields:
\[
\textstyle \sum_{j\not= i}  r_{ij} +\sum_{p\in P} r^p_i=0.
\]
The following is now clear.

\begin{proposition}\label{prop:}
The operations just defined give rise to a co-invariant $\Rscr_P$-module
$\Lcal og_{C,P}:  I\mapsto H^0(C^I, \Omega_C( P)^{(I)}\la\llog\ra)$ 
and a co-invariant $\Rscr^s_P$-module
$\qLog_{C,P}:  I\mapsto H^0(C^I, \Omega_C( P)^{(I)}\la\qlog\ra)$, 
both taking values in  the category of finite dimensional $k$-vector spaces. $\square$
\end{proposition}

\section{Cohomology of the configuration space of a curve}
We derive (and partly recall) some results regarding  the cohomology  of the configuration space  of a 
smooth connected complex  curve.
We also address its mixed Hodge structure.

\subsection{The basic spectral sequence}\label{subsect:basicss}
Here $C$ is a smooth connected (but not necessarily complete) curve over $\CC$ of genus $g$ and 
$P$ a finite (but possibly empty) subset of $C$; the inclusion will be denoted $i_P:P\subset C$. 
\subsubsection*{A stratification and its M\"obius functor}
Observe that $\Delta_{I,P}:=C^I\ssm\inj_I(C\ssm P)$
is  a union of smooth divisors which intersect locally like linear hyperplanes in a complex vector space 
(we called this elsewhere \emph{arrangementlike}). 
These divisors are $H^p_i:=\pi^*_i(p)$ where $i\in I$ and $p\in P$ and $\Delta_{ij}:=(\pi_{i}, \pi_{j})^*\Delta_C$, 
where $\{i, j\}$ a 2-element subset of $I$
(here $\pi_i: C^I\to C$ is evaluation in $i$). The local (transversal) type of these intersections is like that of a 
complexified Coxeter hyperplane arrangement  with components of type $A_*$ or $B_*$. 

These divisors define a stratification of $C^I$. A stratum is given by a
 partition $\Pcal$ of $I\sqcup P$ with the property that its restriction to $P$ is the partition into singletons, i.e., 
 each member of $\Pcal$ contains at most one element of $P$.  So if we regard $\Pcal$ as an equivalence relation 
 on $I\sqcup P$, then it is given by a quotient of $I\sqcup P$ of the form $I_\Pcal\sqcup P$, 
 where a point of $I_\Pcal$ represents a member of $\Pcal$ that is disjoint with $P$.
The associated stratum consists of the set of maps $z: I\to C$ for which
$z\sqcup i_P: I\sqcup P\to C$ factors through an injection 
of $I_\Pcal\sqcup P$. We identify it with  $\inj_{I_\Pcal}(C\ssm P)$ so that its  closure is identified with 
$C^{I_\Pcal}$ and write $\Delta_\Pcal : C^{I_\Pcal}\hookrightarrow C^I$ for the associated closed embedding. 
Its codimension equals $c(\Pcal):=|I|-|I_\Pcal|$.

We denote the collection of decompositions $\Pcal$ of $I\sqcup P$ with the above property by $\Pscr (I, P)$.
This set  is partially ordered, where   by convention $\Pcal\le \Pcal'$ means that $\Pcal$ is equal or finer than 
$\Pcal'$: every member of $\Pcal$ is contained in one of $\Pcal'$
(in other words, the formation of the $\Pcal'$-quotient factors through the $\Pcal$-quotient). 
This definition is opposite to the one  suggested by the incidence property of  the associated strata, 
as this is equivalent to the $\Pcal$-stratum having the  
$\Pcal'$-stratum in its closure.
In particular, $c(\Pcal)=0$ implies that $\Pcal$ is the partition $\Pcal_{\min}$ of $I\sqcup P$ 
into singletons (the least element of $\Pscr(I,P)$). 
If $P\not=\emptyset$, then the maximal elements of $\Pscr (I, P)$ are the $\Pcal$ for which 
$I_\Pcal=\emptyset$ (so that $c(\Pcal)=|I|$), the associated strata being the singletons contained in $P^I$. 
When $P=\emptyset$, there is a greatest element $\Pcal_{\max}$: it has $I$ as a single equivalence class 
(so that $c(\Pcal)=|I|-1$) and the associated stratum is  the main diagonal.

\begin{definition}\label{def:}
The \emph{M\"obius functor} is the contravariant functor
$E$ from $(\Pscr (I,P), \le)$ to the category of free abelian groups of finite rank that is (inductively) characterized by 
the following two properties:
\begin{enumerate}
\item [(i)] when $c(\Pcal)=0$, i.e., when $\Pcal=\Pcal_{\min}$, then $E_\Pcal=\ZZ$,
\item [(ii)] for $c(\Pcal)\ge  1$, the following sequence is exact
\[
0\to E_\Pcal\to \oplus_{\Pcal'\in \Pcal (1)} E_{\Pcal'}\to \oplus_{\Pcal''\in \Pcal (2)} E_{\Pcal''},
\]
\end{enumerate}
where $\Pcal (n)$ denotes the set of  $\Pcal'\in\Pscr$ with $\Pcal'\le \Pcal$ and $c(\Pcal')=c(\Pcal)-n$. 
\end{definition} 

Note that when $c(\Pcal)=1$, then 
$\Pcal (2)=\emptyset$ and so property (ii) then  just tells us that  $E_\Pcal\to E_{\Pcal_{\min}}=\ZZ$ is an isomorphism. 
Such a $\Pcal$ is in fact 
given by a $2$-element subset of $I\sqcup P$ that is either of the form $\{i, j\}\subset I$ (defining the divisor $\Delta_{ij}$) 
or of the form $\{ i, p\}$ with $i\in I$ and $p\in P$ (defining the divisor $H^p_i$), the other members of $\Pcal$ 
being singletons in $I$. 

\subsubsection*{An injective resolution} Given an injective resolution $\Ical^\pt$ of the constant sheaf $\ZZ$ on 
$C^I$, then $\Delta_\Pcal^*\Ical^\pt=\Delta_\Pcal^!\Ical^\pt$
is an injective resolution of the constant sheaf on $C^{I_\Pcal}$ and as observed in \cite{looij1991}
\[
\cdots \to\oplus_{c(\Pcal)=p} \Delta_{\Pcal !} \Delta_{\Pcal}^*\Ical^\pt \otimes 
E_\Pcal \to\cdots \to\oplus_{c(\Pcal)=0} \Delta_{\Pcal !} \Delta_{\Pcal}^*\Ical^\pt \otimes E_\Pcal\to 0
\]
is then an injective resolution of  $j_*j^*\Ical$, where $j: \inj_I(C\ssm P)\subset C^I$ is the inclusion 
(note that the last term $\oplus_{c(\Pcal)=0} \Delta_{\Pcal !} \Delta_{\Pcal}^*\Ical^\pt \otimes E_\Pcal$ 
is just a ridiculously complicated way to write $\Ical^\pt$). 
A Thom isomorphism establishes a quasi-isomorphism of $\Delta_{\Pcal !} \Delta_{\Pcal}^*\Ical^\pt $
with $\Delta_{\Pcal *}\ZZ_{C^{I_\Pcal}}(-c(\Pcal))[-2c(\Pcal)]$, where $(-c(\Pcal))$ indicates a Tate twist
(this makes it a Hodge structure of  rank one of Hodge bidegree $(c(\Pcal), c(\Pcal))$) and 
$[-2c(\Pcal)]$ tells us that we  put this in degree $2c(\Pcal)$. This gives rise to a spectral sequence of mixed Hodge structures with
\begin{equation}\label{display:basicss}
E_1^{-r,s}=\oplus_{c(\Pcal)=r} H^{s-2r}(C^{I_\Pcal})\otimes E_\Pcal(-r)
\Rightarrow H^{s-r}(\inj_I(C);\ZZ).
\end{equation}
We simplify notation a bit by putting 
\[
A_\Pcal:= E_\Pcal(-c(\Pcal))\; \text{  and   }\;  A ^r_{I,P}:=\oplus_{c(\Pcal)=r} A_\Pcal,
\]
 often writing $A^r_I$ for $A^r_{I,\emptyset}$ and 
$A ^r_{N,P}$ resp.\  $A ^r_{N}$ if $I=\{1, \dots, N\}$. The choice of a bijection $I\cong\{1, \dots, N\}$ 
identifies $E_1^{-r,s}$ with $H^{s-2r}(C^{N-r})\otimes A^r_{N,P}$  and hence $E_2^{-r,s}$ with the cohomology of a complex
\begin{equation}\label{display:exact}
H^{s-2r-2}(C^{N-r-1})\otimes A^{r+1}_{N,P}\to H^{s-2r}(C^{N-r})\otimes A^r_{N,P}\to\\ 
H^{s-2r+2}(C^{N-r+1})\otimes A^{r-1}_{N,P}
\end{equation}
of which the differentials are defined by Gysin maps. We make this more explicit below.

\subsubsection*{The Arnol'd algebra} We consider for a moment the case when $C=\CC$ (this is why we did not 
want to insist  that $C$ be complete). Since  $C$ has then the cohomology of a singleton, the spectral sequence
degenerates on the first page. This gives a useful interpretation of $A^r_{I, P}$,   identifying it with $H^r(\inj_I(\CC\ssm P))$. 
In particular, $A^\pt_{I, P}$ is in a natural manner a graded algebra. For $P=\emptyset$, this algebra was introduced 
by Arnol'd and so we shall refer to it as the \emph{Arnol'd algebra}. 
A  presentation of this as a graded-commutative algebra is known: the generators are 
the $g_{ij}=g_{ji}=\frac{d(z_i-z_j)}{z_i-z_j}$ and the $h^p_i:=\frac{dz_i}{z_i-p}$ (which are all in degree $1$ and 
hence anticommute amongst each other) and a complete set of relations is in degree $2$: for $i,j,k$ in $I$ we have 
\[
g_{ij}g_{jk}+ g_{jk}g_{ki}+g_{ki}g_{ij}=0, \quad (h^p_{i}-h^p_{j})g_{ij}=h^p_{i}h^p_{j}.
\]

A set of additive generators of $A^\pt_{I}$ is obtained as follows. Let $r$ be a positive integer. 
With a sequence $Y=(y_1, \dots , y_r)$ in $I$ of pairwise distinct elements, we associate  
\[
g(Y)=g_{y_1y_2}g_{y_2y_3}\cdots g_{y_{r-1}y_r}\in A^{r-1}_{I},
\]
assuming here $r\ge 2$, while for $r=1$, we stipulate this to be $1$. 
Such $g(Y)$ are not linearly independent, because for $r\ge 3$ they sum up to 
zero after cyclic permutation  of the  $(y_1, \dots , y_r)$, provided we tensor with the sign character 
(which can be dropped if we regard them as polydifferentials, see the discussion in Remark \ref{rem:graphical}). 
A set of additive generators of  $A^{N-\ell}_{I}$ then consists of the products $g_{Y_1}g_{Y_2}\cdots g_{Y_\ell}$, 
where $(Y_1, \dots , Y_\ell )$ is a decomposition of $I$ into totally ordered subsets $Y_i$. In particular,
the elements defined by the total orders on  $I$ span $A^{N-1}_I$.  A somewhat more intrinsic way to state this is that  
for every partition $\Pcal$ of $I$ into $\ell$ parts, we have $A_\Pcal\cong \otimes_{J\in \Pcal} A^{|J|-1}_J$ 
(a canonical isomorphism  involves a sign  character, which we omit). The choice of  a total order  on each part $J$ of 
$\Pcal$  determines a partial order  on $I$ that makes it a disjoint union of chains,  which in turn determines  a 
primitive rank one submodule   of $A_\Pcal$. The space  $A_\Pcal$ is spanned by such rank one submodules 
and $A^\pt_I=\oplus_{\Pcal} A_\Pcal$.

The  generalization to $A^\pt_{I,P}$ is straightforward:  we have a decomposition of $A^\pt_{I,P}$ indexed by 
$\Pscr(I,P)$:  $A^\pt_{I,P}=\oplus_{\Pcal\in \Pscr(I,P)}A^\pt_\Pcal$ 
with each summand  $A^\pt_\Pcal$ decomposing as a tensor product according to the members of 
$\Pcal$: $A^\pt_\Pcal\cong \otimes_{J\in \Pcal} A^\pt_J$. 
For $J\in \Pcal$, a total order on $J\cap I$ determines a rank one sublattice of $A^\pt_J$ and such 
sublattices generate $A^\pt_J$.
Concretely: if for a sequence $Y=(y_1, \dots, y_r)$ of pairwise distinct elements in $I$ and   $p\in P$,  we put
\[
h^p(Y):=h^p_{y_1}g(Y)=h^p_{y_1}g_{y_1y_2}g_{y_2y_3}\cdots g_{y_{r-1}y_r}\in A^r_{I,P}, 
\]
and convene that $g(\emptyset)=0$ and $h^p(\emptyset)=1$, then products of such elements with disjoint support 
will generate $A^\pt_{I,P}$. 

These explicit descriptions are useful in view of following link with our Lie module categories.
For this we return to Example \ref{example:surface} with $M=\CC$ and $\mathring M=\CC\ssm P$. We put
\begin{gather*}
\Acal(I):=A^{|I|-1}_{I}\otimes \sign(I)= H^{|I|-1}(\inj_I(\CC))\otimes \sign(I), \\
\Acal_P(I):=A^{|I|}_{I,P}\otimes \sign(I)= H^{|I|}(\inj_I(\CC\ssm P))\otimes \sign(I)\text{ when } P\not=\emptyset.
\end{gather*}
We saw that the usual residue maps make $\Acal$ an $\Rscr$-module and $ \Acal_P$ an $\Rscr_P$-module.
Given the presentation of these algebras, the following is straightforward:

\begin{proposition}\label{prop:}These are principal modules in the sense that 
the additive functors $\Rscr\to \Acal$ and $\Rscr_P\to \Acal_P$ are equivalences of pre-additive categories.
\end{proposition}
For $P=\emptyset$ this is merely a reformulation of the classical theorem of F.~Cohen \cite{cohenf95}.

\subsection{The projective case}\label{subsect:projcase}
From now on we assume that $C$ is projective. 

\subsubsection*{Review of the K\"unneth decomposition} What follows is (very) well-known, but since we 
must be careful with signs, we believe it is worthwhile to spell this out in our setting. We let $\Scal_2$ act on 
$C^2$ by transposition of factors. If we take its action on the cohomology into account, then the K\"unneth decomposition 
reads as follows
\[
H^2(C^2)\cong \big(H^2(C)\!\otimes\! 1 \oplus 1\!\otimes\! 
H^2(C)\big)\oplus \big(H^1(C)\otimes  H^1(C)\otimes\sign_2\big).
\]
Indeed, if $\mu\in H^2(C)$ is the natural generator, then the transposition exchanges 
$\mu\otimes 1$ and $1\otimes\mu$, but if $\alpha, \beta\in H^1(C)$, then it exchanges  
$\alpha\otimes\beta$ and $-\beta\otimes\alpha$. If $\Delta: C\to C^2$ is the diagonal, then the Gysin map 
$\Delta_!: H^\pt (C)\to H^{\pt+2}(C^2)$ is $H^\pt(C^2)$-linear.  
The class of the diagonal is $\Delta_{!}(1)\in H^2(C^2)$  is of course fixed under $\Scal_2$ and is 
under the K\"unneth decomposition equal to 
$\mu\times 1+1\times \mu +\delta$, where 
$\delta\in H^1(C)\otimes H^1(C)$ is defined by the intersection pairing on $H_1(C)$. 
To be precise,  if $\alpha_{\pm1}, \dots, \alpha_{\pm g}$ is a symplectic basis
of $H^1(C)$ (i.e.,  if $i>0$, then $\alpha_i\cdot\alpha_j$ is $1$ when $j=-i$ and is zero otherwise), then 
$\delta=\sum_{i=1}^g (-\alpha_i\times\alpha_{-i}+\alpha_{-i}\times\alpha_i)$. 
Note that $\delta$ has the property that for $\alpha\in H^1(C)$,  
\begin{gather*}
(\alpha\times 1)\cup\delta =\mu\times\alpha, \quad (1\times\alpha)\cup\delta=\alpha\times \mu \text{  and so}\\
\Delta_{!}(\alpha)=\mu\times\alpha+\alpha\times\mu.
\end{gather*} 
Both $\mu$ and $\delta$ are of Hodge type $(1,1)$. 

As we explained under the Conventions heading of the introduction, if $I$ is a nonempty set, 
then we have a K\"unneth  summand
$H^1(C)^{\otimes I}\otimes\sign (I)$ in $H^\pt (C^I)$. 
The full K\"unneth decomposition of $H^\pt (C^I)$ is then as follows:
\[
H^\pt (C^I)=\oplus_{J\cap K=\emptyset}\;  (\otimes_{j\in J}\pi_j^*\mu)\otimes H^1(C)^{\otimes K}\otimes\sign(K),
\]
where the sum is over pairs $(J,K)$ of disjoint subsets of $I$, and where we stipulate that for $K=\emptyset$, 
the tensor factor $H^1(C)^{\otimes K}\otimes\sign(K)$ this is just $\ZZ$ placed in degree zero.

The diagonal embedding $\Delta_{ij}: C^{I_{ij}}\hookrightarrow C^I$ defines an algebra  homomorphism
$\Delta_{ij}^*: H^\pt (C^I)\to H^\pt (C^{I_{ij}})$ which in terms of the K\"unneth decompositions is on the 
$(J,K)$-component given as follows.
When $i, j$  both lie in $J\cup K$ with one of them lying in $J$, then it is the zero map.  
If one of them does not lie in $J\cup K$, then $J$ and $K$ have faithful images $\bar J$ and $\bar K$ in $I_{ij}$ and 
$\Delta_{ij}^*$ then maps the $(J,K)$-component in an evident manner of the 
$(\bar J, \bar K)$-component. Finally, if both $i,j\in K$, then $\Delta_{ij}^*$ is the cup product on 
the corresponding factors:
\[
H^1(C)^{\otimes K}\otimes \sign K\to \mu_{[ij]}\otimes H^1(C)^{\otimes (K\ssm\{i,j\})}\otimes \sign (K\ssm\{i,j\}),
\]
where $\mu_{[ij]}$ is to be understood as $\mu$ put in the slot $[ij]$ of $I_{ij}$. The ordering of the pair $(i,j)$ is 
important here, but when an ordering is given, we indeed have an identification $\sign (K)\cong \sign (K\ssm\{i,j\})$ 
(a total order on $K$ must begin with $i<j$). 

The associated Gysin map $\Delta_{ij!}: H^\pt (C^{I_{ij}})\to H^{2+\pt}(C^I)$ is $H^\pt (C^I)$-linear and we have
\[
\Delta_{ij!}(1)=\mu_i +\mu_j +\pi_{ij}^*\delta,
\]
where we regard $\pi_{ij}^*\delta$ as an element of the K\"unneth summand defined by $J=\emptyset$ and  $K=\{i,j\}$. The 
restriction of $\Delta_{ij!}$ to a K\"unneth summand  of $H^\pt (C^{I_{ij}})$ indexed by  a pair $(\bar J, \bar K)$ is 
then easy to describe:
it is zero if $[ij]\in \bar J$. In case $[ij]\in \bar K$, we get 
two identical such summands in $H^\pt (C^{I})$, indexed by $(\bar J\cup \{i\}, K_j)$ and $(\bar J\cup \{j\}, K_i)$ 
(where $K_i$ maps bijectively to $\bar K$ and contains $i$ and similarly for $K_j$), and $\Delta_{ij!}$ maps 
this summand diagonally to these two.
In the remaining case, where $[ij]$ is neither in $\bar J$ nor in $\bar K$, we find that cupping with 
$\pi_{ij}^*\delta$ naturally maps
$ H^1(C)^{\otimes \bar K}\otimes\sign (\bar K)$
to $H^1(C)^{\otimes (\{i,j\}\cup\bar K) }\otimes\sign (\{i,j\}\cup\bar K)$, where we let the ordered pair $(i,j)$ 
precede $\bar K$ when identifying
$\sign (\bar K)$ with $\sign (\{i,j\})\cup\bar K$.

\subsubsection*{A differential graded algebra}
Returning to our basic spectral sequence \ref{display:basicss}, we imitate Totaro \cite{totaro} 
(who considered the case $P=\emptyset$, but allowed $C$ to be any topological manifold) by 
identifying  its first page with its differential as a differential graded algebra:
consider the  bigraded-commutative $H^\pt(C^I)$-algebra  $A^{\pt,\pt}_{I, P}(C)$ that is obtained as a 
quotient of $H^\pt(C^I)\otimes A^\pt_{I,P}$ by dividing out by the ideal $\Ical$ generated by  
$(\pi_i^*-\pi_j^*)(x)\otimes g_{ij}$ and $\pi_i^*x\otimes h^p_i$ with $x\in H^\pt(C)$
 of positive degree and  $i, j\in I$. Here $\pi_i: C^I\to C$ takes the value at $i$ (the $i$-th component).  So
$E_1^{-r,s}=A^{s-2r,r}_{I, P}$.
It is clear that the dependence of $A^{\pt,\pt}_{I, P}(C)$ on $C$ is only via the augmented graded algebra  $H^\pt(C)$.
Each summand $H^\pt(C^{I_\Pcal})\otimes A^{c(\Pcal)}_{\Pcal}$ decomposes as a tensor product whose factors are 
indexed by the members $J$ of
$\Pcal$: if $J\subset I$, then we have a tensor factor $H^\pt (C)\otimes A^{|J|-1}_J$ with $C$ here being 
identified with the main diagonal of $C^J$ (the constant maps $J\to C$) and  if $J$ contains $p\in P$, 
the tensor factor is just $A_{J\ssm\{p\}, \{p\}}^{|J|-1}$. Somewhat more concretely,  
if we enumerate the elements of $P$ by means of a bijection $\nu\in \{1, \dots, n\}\mapsto p_\nu\in P$, 
then an element of $A^{s-2r,r}_{I, P}(C)$ is a sum of elements of the following type
\[
x\otimes g(Y_{1})\cdots g(Y_{|I|-r})h^{p_1}(Y^1)\cdots h^{p_n}(Y^n)
\]
with  $x\in H^{s-2r}(C^{|I|-r})$, where $Y_1, \dots , Y_{|I|-r}, \{p_1\}\cup Y^1, \dots , \{p_n\}\cup  Y^n$ 
decompose $I\sqcup P$
 (this copy of $C^{|I|-r}$ is better thought of as the stratum closure defined by $z_i=z_j$ whenever $i,j\in Y_t$ for some $t$
and $z_i=p_\nu$ when $i\in Y^\nu$).  So this can only be nonzero when $s-2r\le 2(|I|-r)$, i.e., when $s\le |I|$.

We make $A^\pt_{I, P}(C)$ a differential graded algebra over $H^\pt(C^I)$ by putting 
\[
d(g_{ij})=\pi_{ij}^*\Delta_{!}(1)=\mu_i+\delta_{ij} +\mu_j \text{  and  } d(h^p_i)=\mu_i,
\] 
where $\mu_i=\pi_i^*(\mu)$ and $\delta_{ij}:=\pi_{ij}^*\delta$. This is well-defined, for one checks that $d$ maps both 
$g_{ij}g_{jk}+ g_{jk}g_{ki}+g_{ki}g_{ij}$ and $(h^p_{i}-h^p_{j})g_{ij}-h^p_{i}h^p_{j}$ to $\Ical$.
This is in fact a differential graded algebra in the category of mixed Hodge structures.  Note that then 
\begin{equation*}\label{eqn:diff1}
\textstyle d g(Y)=\sum_{Y=Y'Y''} (-1)^{|Y'|-1}(\mu_{[Y']} +\delta_{[Y'],[Y'']} +\mu_{[Y'']})\otimes g(Y') g(Y''), 
\end{equation*}
where the notation should be self-explanatory: 
$Y$ is written as a juxtaposition $Y'Y''$ with $Y'$ and $Y''$ nonempty and if  
$Y'=(y_1\cdots y_k)$ and $Y''=(y_{k+1}\cdots y_r)$ for some
$k\in \{1, \dots , r-1\}$, then  the defining relations justify writing $g(Y')(\mu_{[Y']} +\delta_{[Y'],[Y'']} +\mu_{[Y'']})g(Y'')$
 for $ g(Y')(\mu_{y_k} +\delta_{y_k,y_{k+1}} +\mu_{y_{k+1}}) g(Y'')$.
Similarly we have
\begin{equation*}\label{eqn:diff2}
\textstyle dh^p(Y)=\sum_{Y=Y'Y''} (-1)^{|Y'|}\mu_{[Y'']}\otimes  h^p(Y')g(Y''), 
\end{equation*}
The exact sequence \ref{display:exact} is now written more intrinsically as
\begin{equation}\label{eqn:dga}
A^{s-2r-2,r+1}_{I, P}(C)\xrightarrow{d}  A^{s-2r,r}_{I, P}(C)\xrightarrow{d} A^{s-2r+2,r-1}_{I, P}(C). 
\end{equation}
So it gives the first page of our spectral sequence with its differential, but also takes into account its multiplicative structure. 

\begin{corollary}\label{cor:configDGA}
If  $C$ projective,  then  the cohomology of 
\begin{equation*}
H^{l-r-2}(C^{|I|-r-1})\otimes A^{r+1}_{I,P}\to H^{l-r}(C^{|I|-r})\otimes A^r_{I,P}\to\\ 
H^{l-r+2}(C^{|I|-r+1})\otimes A^{r-1}_{I,P}.
\end{equation*}
is a subquotient of $H^{l}(\inj_I(C\ssm P))$ which, when tensored with $\QQ$, equals 
$\Gr^W_{l+r} H^{l}(\inj_I(C\ssm P); \QQ)$.
\end{corollary}
\begin{proof}
Under our assumption $E_1^{-r,s}$ is pure of weight $s$ and so the spectral sequence degenerates over 
$\QQ$ at the second page 
and  the Leray filtration is essentially the weight  filtration. This is equivalent to the assertion of the corollary.
\end{proof}

\begin{example}[The case of genus zero]\label{example:g=0}
Consider the special case when $C=\PP^1$ and $\infty\in P$.  We 
put $P':=P\ssm\{\infty\}$, so that $\PP^1\ssm P=\AA^1\ssm P'$.  
For each side of the identity $\inj_I(\AA^1\ssm P')=\inj_I(\PP^1\ssm P)$ we have a method
for computing the cohomology and we wish to compare them. 
We know that the cohomology algebra of $\inj_I(\AA^1\ssm P')$ is $A^\pt_{I, P'}$. 
In particular,  its cohomology in degree $r$  is pure of type $(r,r)$.
On the other hand, $A_{I,P}(\PP^1)^\pt$ is linearly spanned by the tensors of the form 
\[
\mu_{Y_{i_1}}\mu_{Y_{i_2}}\cdots \mu_{Y_{i_k}}\otimes g(Y_{1})\cdots g(Y_s)h^{p_1}(Y_{s+1})\cdots h^{p_t}(Y_{s+t})
\]
where $1\le i_1<\cdots i_k\le s$. The Hodge weight of such a tensor is twice its degree if and only if 
$k=0$, that is, if and only if this element
lies in $A_{I,P}^\pt$. Let us here note that the differential of $A_{I,P}(\PP^1)^\pt$ takes 
$g_{ij}$ to $\mu_i+\mu_j$ and $h^p_i$ to $\mu_i$.
In particular, any coboundary of $A_{I,P}(\PP^1)^\pt$ is in the submodule of $A_{I,P}(\PP^1)^\pt$ 
generated by the $\mu_i$'s and so 
$A_{I,P}^\pt$ will not contain nonzero coboundaries of  $A_{I,P}(\PP^1)^\pt$. It follows that
\[
K^\pt_{I, P}:=\ker (A_{I,P}^\pt \xrightarrow{d}A_{I,P}(\PP^1)^\pt)
\]
can be identified with $H^\pt(\inj_I(\AA^1\ssm P'))$ and hence with $A^\pt_{I, P'}$ (with its graded Hodge structure). The 
isomorphism $f: A^\pt_{I, P'}\cong K^\pt_{I,P}$ thus obtained is in terms of rational differential forms the identity,
the point  is merely that these forms are written differently in the source and target of $f$: the rational form $z/(z-p)$ on 
$\AA^1$ with $p\in P'$ 
has on $\PP^1$ polar divisor $(p)+(\infty)$ and  the rational form
$d(z_1-z_2)/(z_1-z_2)$ on $\AA^2$ on  $\PP^1\times\PP^1$ polar divisor $(z_1=z_2)+(z_1=\infty) +(z_2=\infty)$,  
where in both cases the residues at infinity are $-1$.  This translates into $f(h^p_i)=h^p_i-h^\infty_i$ and 
$f(g_{ij})=g_{ij}-h^\infty_i-h^\infty_j$.
So phrased in a manner  that does not involve the choice of $\infty$:  $K^\pt_{I,P}$ is as  a subalgebra of 
$A^\pt_{I,P}$ generated by  $\{g_{ij}-h^p_i-h^p_j\}_{i\not=j, p\in P}\cup \{h^p_i-h^q_i\}_{p\not=q, i\in I}$. 

Let us see how this works out in the top  degree $N:=|I|$.  We know that  $A^N_{I,P'}$ is as a group generated by  
the expressions $h^p(Y)$, where $p\in P'$ and $Y$ is
defined by a total order on $I$ (so for this to be nonzero, we  need $P'$ to be nonempty as well). 
If we use the total order to identify $I$ with $\{1, \dots, N\}$, so that 
$h^p(Y)=h^p_1g_{1,2}g_{2,3}\cdots g_{N-1,N}$, then the associated rational $N$-form on $\CC^N$ is 
$\frac{dz_1}{z_1}\wedge\frac{d(z_1-z_2)}{z_1-z_2}\wedge\cdots \wedge \frac{d(z_{N-1}-z_N)}{z_{N-1}-z_N}$.  
To interpret this as an element of 
$K^N_{I,P}$, we need to regard this as a rational $N$-form on $(\PP^1)^N$. Its polar divisor is then 
$(z_1=p)+ (z_1=z_2)+\cdots +(z_{N-1}=z_N)+(z_N=\infty)$ (successive terms have opposite residues). 
The associated element of $K^N_{I,P}$ is obtained from the sequence 
$(h^p_1, g_{1,2}, g_{2,3},\cdots ,g_{N-1,N}, h^\infty_N)$,  or rather the underlying  cyclic order on this set of 
$N+1$ of elements,  by taking
the sum of the $N+1$ products that can formed from $N$ successive terms. 
For example, for $N=2$, we get $h^p_1g_{1,2}+g_{1,2}h^\infty_2+h^\infty_2h^p_1\in K^2_{2, P}$.
\end{example}

\subsubsection*{The Hodge filtration} As to the Hodge filtration, the constructions and theorems of Deligne in the 
normal crossing case in \S 3 of \cite{deligne:hodge2} carry over to the present situation with little change. 
By an argument similar as in the normal crossing case, we find:

\begin{proposition}\label{prop:}
Let  $\Omega_{C^I}^\pt(\log \Delta_{I,P})$ be  the complex of holomorphic forms on $C^I$ with logarithmic poles along 
$\Delta_{I,P}$, i.e.,  characterized by the property of being the largest subcomplex of the direct image of the holomorphic 
De Rham complex of $\inj_I(C)$ consisting of forms that are meromorphic on $C^I$ and have poles of order $1$ at most. 
Then  this complex represents $Rj_*j^*\CC_{C^I}$, where $j: \inj_I(C)\subset C^I$ and the associated spectral sequence 
\[
E^{p,q}_1=H^{q}(C^I, \Omega_{C^I}^p(\log \Delta_{I,P}))\Rightarrow H^{p+q}(\inj_I(C); \CC)
\]
degenerates on the first page and establishes an isomorphism
\begin{equation}\label{display:hodgefil}
\Gr^p_F H^r(\inj_I(C\ssm P); \CC)=H^{r-p}(C^I, \Omega_{C^I}^p(\log \Delta_{I,P})). 
\end{equation}
\hfill$\square$
\end{proposition}

The complex of $\Ocal_{C^I}$-modules $\Omega_{C^I}^\pt(\log \Delta_{I,P})$ comes with a weight filtration:
the subcomplex of 
forms for which each $(q+1)$-fold residue vanishes defines $W_{p+q}\Omega_{C^I}^p(\log \Delta_{I,P})$.  Then 
\[
\Gr^W_{p+q}\Omega_{C^I}^p(\log \Delta_{I,P})=
\oplus_{c(I_\Pcal)=q}\Delta_{\Pcal *}\Omega^{p-q}_{C^{I_\Pcal}}\otimes E_\Pcal.
\] 
This is compatible with our basic spectral sequence, so that $\Gr^W_{r+q}\Gr^p_F H^r(\inj_I(C\ssm P);\CC)$ appears as the 
cohomology of 
\begin{multline*}
\oplus_{c(I_\Pcal)=q+1} H^{r-p-1}(C^{I_\Pcal}, \Omega^{p-q-1}_{C^{I_\Pcal}})\otimes E_\Pcal \to 
\oplus_{c(I_\Pcal)=q} H^{r-p}(C^{I_\Pcal}, \Omega^{p-q}_{C^{I_\Pcal}})\otimes E_\Pcal\to\\ \to 
\oplus_{c(I_\Pcal)=q-1} H^{r-p+1}(C^{I_\Pcal}, \Omega^{p-q+1}_{C^{I_\Pcal}})\otimes E_\Pcal
\end{multline*}
For $p=r=|I|$ this identifies $\Gr^W_{|I|+q} F^{|I|}H^{|I|}(\inj_I(C\ssm P); \CC)$ 
with the kernel of the  differential
\begin{equation}\label{display:hodgefil2}
d: \oplus_{c(I_\Pcal)=q} H^0(C^{I_\Pcal}, \Omega^{|I_\Pcal|}_{C^{I_\Pcal}})\otimes E_\Pcal\to
\oplus_{c(I_{\Pcal'})=q-1} H^{1}(C^{I_{\Pcal'}}, \Omega^{|I_{\Pcal'}|}_{C^{I_{\Pcal'}}})\otimes E_{\Pcal'}.
\end{equation}
The `matrix entries' of this map are Gysin maps. Since for any nonempty finite set $L$, 
\[
H^{1}(C^L, \Omega^{|L|}_{C^L})=\oplus_{j\in  L} 
H^0(C^{L\ssm \{j\}},\Omega^{|L\ssm \{j\}|}_{C^{L\ssm \{j\}}})\otimes \pi_j^*(\mu)
\]
(as $\mu$ has even degree there is no sign issue) and recalling that  for $\alpha\in H^1(C)$ we have 
$\Delta_!(\alpha)=\alpha\times \mu +\mu\times\alpha=(\alpha\times 1+1\times\alpha)\cup (\mu\times 1+1\times\mu)$, 
we see that $d$ no longer involves the classes $\delta_{ij}$. This shows that the formation of
$\Gr^W_{|I|+q} F^{|I|}H^{|I|}(\inj_I(C\ssm P);\CC)$ is functorial in $H^0(C, \Omega_C)$ in the sense that 
there is a $\CC[\Scal_{|I|-q}]$-module  $W_q(|I|,P)$  (independent of the genus of $C$ and with trivial Hodge structure) such that
\[
\Gr^W_{|I|+q} F^{|I|}H^{|I|}(\inj_I(C\ssm P);\CC)=
H^0(C, \Omega_C)^{\otimes (|I|-q)}\otimes_{\CC[\Scal_{|I|-q}]} W_q(|I|,P)(-q).
\]
This is related to the following. The residue formula implies that for every $i\in I$  the residue  sum 
\[
\textstyle R_i=\sum_{j\in I\ssm \{ i\} } \res_{\Delta_{ij}} +\sum_{p\in P} \res_{H_i^p}
\] 
takes the value zero on  $H^0(C^I,\Omega_C^{|I|}(\log \Delta_{I,P}))$. So the
residues of an element of $H^0(C^I,\Omega_C^{|I|}))(\log \Delta_{I,P})$ cannot be arbitrary. 
However, in view of the following proposition this is the only restriction.

\begin{proposition}\label{prop:residuechar}
Let $D$ be an effective divisor on $C^N$. 
Then a section  $\xi\in \Omega_{C^N}^{N}(D)/\Omega_{C^N}^{N}$ is the restriction of a section 
$\tilde\xi$ of $\Omega_{C^N}^{N}(D)$  if and only if  for every $i\in I$ the sum of the residues of 
$\xi$ along the projection $\pi^i: C^N\to  C^{N-1}$  which forgets the $i$th component is zero.

In case $D$ is $\Scal_I$-invariant and $\xi$ transforms under  the permutation group 
$\Scal_I$ of $I$ according to an  irreducible  character $\chi$ of $\Scal_N$, 
then we can arrange  for $\tilde\xi$ to have the same property.  
\end{proposition}
\begin{proof}
Consider the exact  sequence 
$0\to \Omega^N_{C^N}\to \Omega^N_{C^N}(D)\to \Omega^N_{C^N}(D)/\Omega^N_{C^N}\to 0$.  
The long exact sequence for cohomology begins as
\[
0\to H^0(C^N, \Omega^N_{C^N})\to 
H^0(C^N,\Omega^N_{C^N}(D)\to H^0(C^N,\Omega^N_{C^N}(D)/\Omega^N_{C^N})
\xrightarrow{\delta} H^1(C^N,\Omega^N_{C^N})
\]
We have a natural isomorphism $H^1(C^N,\Omega^N_{C^N})\cong \oplus_{i =1}^{N} 
\mu_i\otimes H^0(C^{N-1}, \Omega^{N-1}_{C^{N-1}})$ and via this isomorphism the $i$th component of the 
coboundary $\delta$
is given by taking residue sum  along  the $i$th coordinate.  The first assertion follows.

The second is a consequence of the fact that taking the $\chi$-isotypical part is an exact functor. 
\end{proof}

\begin{example}\label{example:}
Let us carry this out  for the case that plays  a central  role in \cite{looij:2021}: 
where $I$ a $3$-element set $\{a,b,c\}$ and $P=\emptyset$. 
So $\Pscr(I)$ then consists of $\Pcal_{\min}$ (the partition into atoms defining the open stratum of $C^I$), 
$\Pcal_{\max}$ (defining the main diagonal) and the partitions $ab|c$, $bc|a$, $ca|b$
defining the diagonal  hypersurfaces. Let us determine $\Gr^W_4\Gr^3_F H^3(\inj_I(C))$ using 
\ref{display:hodgefil2}. In this case  
both $E_\Pcal$ and $E_{\Pcal'}$ are just $\ZZ$ and hence can be omitted. 
Let us abbreviate $\Omega:=H^0(C, \Omega_C)$. 
We will deal with tensor powers of $\Omega$ with exponent a $2$-element set. 
We make sure that these $2$-element sets are ordered so that they can be identified with each other.  
Let $\alpha\mapsto \alpha^*$ denote the exchange 
map that is defined by changing the order. 
The component associated to $\Pcal=ab|c$ and $\Pcal'=\Pcal_{\min}$ then becomes
\[
\Omega^{\otimes(ab, c)}\to \Omega^{\otimes (b,c)}\times\mu_a + \Omega^{\otimes(c,a)}\times\mu_b, \quad 
\alpha\mapsto \alpha\times \mu_a+\alpha^*\times\mu_b
\]
and the components associated to the other diagonal  hypersurfaces are obtained by means of a cyclic permutation. 
It follows that the kernel of the map \ref{display:hodgefil2} is identified with the triples $(\alpha, \beta, \gamma)$ in 
$\Omega \otimes \Omega$ for which
$(\gamma+\beta^*, \alpha+\gamma^*, \beta+\alpha^*)=(0,0,0)$. 
This amounts to: $\alpha=\beta=\gamma=-\alpha^*$. Hence this kernel is identified
with $\sym^2\Omega$. 

A similar argument shows that the map 
\[
\Omega^{\otimes(a, b, c)}\otimes E_{\Pcal_{\min}}\to \Omega^{\otimes (ab)}\times\mu_c \oplus \Omega^{\otimes(bc)}\times\mu_a\oplus\Omega^{\otimes(ca)}\times\mu_b
\]
is injective, so that $\Gr^W_5\Gr^3_F H^3(\inj_I(C))=0$ (this also follows from the fact that the residue theorem 
precludes the existence of
a meromorphic $2$-form on a diagonal divisor with polar divisor minus the main diagonal).
\end{example}

This remains so if we include a sign twist, allowing us to work with polydifferentials instead of differential forms. 
The displayed identity \ref{display:hodgefil} shows that
\[
F^{|I|}H^{|I|}(\inj_I(C\ssm P);\CC)\otimes\sign(I) =H^0(C^I, \Omega_C(P)^{(I)}\la \llog\ra).
\]
Its weight filtration is characterized in terms of residues as above, but beware that these are then residues 
taken in the (anticommutative) sense of Subsection \ref{subsect:Liepointedcurve}. We record this follows.

\begin{corollary}\label{cor:recordpolydiffs}
Let $N:=|I|$. We have a natural  isomorphism 
\[
H^0(C^I, \Omega_C(P)^{(I)}\la \llog\ra)\cong F^NH^{N}(\inj_I(C\ssm P);\CC)\otimes\sign(I).
\] 
Via this identification,
$W_{N+q}F^{N}H^{N}(\inj_I(C\ssm P))\otimes\sign(I)$ corresponds to the space logarithmic forms in 
$H^0(C^I, \Omega_C(P)^{(I)}\la \llog\ra)$ for which every $(q+1)$-fold residue is zero. 
There exists a $\CC[\Scal_{N-q}]$-module $W_q(N,P)$ with trivial Hodge structure such that this induces an isomorphism of  pure Hodge structures
\[
\Gr^W_{N+q}F^{N}H^{N}(\inj_I(C\ssm P);\CC)\otimes\sign (I)\cong 
H^0(C, \Omega_C)^{\otimes (N-q)}\otimes_{\CC[\Scal_{N-q}]} W_q(N,P)(-q).
\]

Furthermore, if we are given on every stratum of $\Delta_{I,P}$ a logarithmic polydifferential of maximal degree such that 
for every codimension one incidence of such strata one is a residue of the other, then these are the residues of a 
single logarithmic polydifferential on $C^I$ if and only if 
if  for every $i\in I$,  the polydifferentials on the open strata of $\Delta_{I,P}$ that involve the index $i$ add up to zero.
\end{corollary}

\begin{remark}\label{rem:}
According to Deligne (\cite{deligne:hodge2} Lemme 1.2.8), a mixed Hodge structure admits a natural splitting over 
$\CC$ (not in general over $\RR)$ and so yields an isomorphism
\[
H^0(C^I, \Omega_C(P)^{(I)}\la \llog\ra)\cong 
\oplus_q H^0(C, \Omega_C)^{\otimes (N-q)}\otimes_{\CC[\Scal_{N-q}]} W_q(N,P)(-q).
\]
\end{remark}

\begin{remark}\label{rem:tateinsertion}
Lefschetz duality yields an isomorphism $H^{N}(\inj_I(C\ssm P);\ZZ)\cong H_{N}(C^I, \Delta_{I,P};\ZZ)$. The latter is same as the homology with closed support $H^{cl}_{N}(\inj_I(C\ssm P);\ZZ)$, also known as  \emph{Borel-Moore} homology. But in order to take the mixed Hodge structure into account, a Tate twist with $\ZZ(N)$ is mandatory:
\[
H^{cl}_{N}(\inj_I(C\ssm P);\ZZ)\cong H^{N}(\inj_I(C\ssm P);\ZZ(N)).
\]
Corollary \ref{cor:recordpolydiffs} then identifies $H^0(C^I, \Omega_C(P)^{(I)}\la \llog\ra)$ with $
F^{0}H_{N}^{cl}(\inj_I(C\ssm P);\CC)\otimes\sign(I)$ and this give rise to isomorphisms
\[
\Gr^W_{-s}F^{0}H^{cl}_{N}(\inj_I(C\ssm P);\CC)\otimes\sign (I)\cong \big(H^0(C, \Omega_C)(1)\big)^{\otimes s }\otimes_{\CC[\Scal_{s}]} W_{N-s}(N,P).
\]
One advantage of this Tate twist is that it converts the omnipresent residue maps into boundary maps that involve no Tate twist. This implies for instance that $I\mapsto H^{N}(\inj_I(C\ssm P);\ZZ(|I|))$ is a $\Rscr_P$-module in the category of mixed Hodge structures. 
There are probably more conceptual reasons for inserting it as well. In view of Beilinson's motivic description  of the fundamental group (as explained in  Deligne-Goncharov \cite{del-gon}), the Hodge structure  on $H^{cl}_{N}(\inj_I(C\ssm P);\ZZ)$ is likely to be related to the Hodge structure on a truncation of the  fundamental groupoid of $C_s$ restricted to $P_s$. In any case, for our envisaged application this seems to be the more natural group to consider and this is why we shall later insert it.
\end{remark}

\section{Canonical quasi-logarithmic polydifferentials}\label{sect:simpleproduct} 
In  this section we take $k=\CC$ and  $C$ stands for a nonsingular projective irreducible complex curve (equivalently, 
a compact connected Riemann surface) of genus $g$. We let  $P\subset C$ be a nonempty finite subset, 
so that $C\ssm P$ is a nonsingular affine complex curve.

\subsection{A canonical bidifferential on $C^2$}\label{subsect:canbidiff}
 We first note  that the biresidue defines an  equivariant trivialization
\[
\Omega_{C^2}^2(2\Delta_{12})\otimes \Ocal_{\Delta_{12}}\cong \Ocal_{\Delta_{12}}\otimes\sign_2.
\]
According to Biswas-Raina (Prop.\ 2.10 of \cite{br1}) there exists a 
$\zeta\in H^0(C^2,\Omega^2(2\Delta_{12}))$ whose restriction to
${\Delta_{12}}$ yields this trivialization. We may, of course, take this generator to be anti-invariant under 
$\sigma$: $\sigma^*\zeta=-\zeta$. 
It is then unique up an element of $H^0(C^2, \Omega^2)^{-\sigma}\cong \sym^2H^0(C,\Omega_C)$. 
Note that this space is of Hodge type $(2,0)$. It defines an anti-invariant class in $H^2(\inj_2(C); \CC)$.  
As is clear from the K\"unneth decomposition (see the discussion in Subsection \ref{subsect:projcase}), 
this space of anti-invariant classes  can be identified with the direct sum of  $\sym^2 H^1(C; \CC)$ and the span of 
$-\mu\times 1+1\times\mu$. We will identify a unique  choice of $\zeta$ of pure type $(1,1)$ and 
determine its class in terms of this decomposition.

In order to find the coefficient of $[\zeta]$ on $-\mu\times 1+1\times\mu$ we proceed as follows.
Choose $p\in C$ and consider $Z:=-C\times\{p\}+\{ p\}\times C$ as a $2$-cycle. 
It is clear that with respect to the K\"unneth decomposition  of $H^2(C)$,  
its class is annihilated by $H^1(C)\times H^1(C)$, whereas $\mu\times 1$ and $1\times\mu$ take on it 
the value $-1$ resp.\ $1$. So the coefficient in question is then computed by integrating $\zeta$ over any 
$2$-cycle on $C^2\ssm \Delta_{12}$ that is homologous to $Z$ in $H_2(C^2)$. 
We construct such a cycle   by modifying $Z$ a bit near $(p,p)$. To this end we choose a chart $(U;z)$ centered at $p$,
so  that $\zeta$  has on $U^2$ the form
\[
\zeta =(z_1-z_2)^{-2} dz_1 \wedge dz_2 +\text{a form regular at  $(p,p)$.} \tag{$\dagger$}
\]
For any  $\eps>0$  such that $U$  contains a  closed disk $D_{2\eps}$ mapping onto the closed disk of radius 
$2\eps$ in $\CC$, we embed the cylinder $[-\eps, \eps]\times S^1$ in $U^2$ by 
\[
u_\eps: [-\eps, \eps]\times [0, 2\pi]\to U^2; \quad   u^*z_1=
(\eps-s)e^{\sqrt{-1}\phi}\,  , \, u^*z_2=-(\eps +s)e^{\sqrt{-1}\phi}
\]
Considered as a $2$-chain, its  boundary is $\{p\} \times\partial D_\eps-\partial D_\eps\times \{p\}$ 
(we use the  complex orientation), so that if we add to this the $2$-chain 
$-(C\ssm D_\eps)\times\{p\} +\{p\} \times (C\ssm D_\eps)$ we obtain a $2$-cycle $Z_\eps$ in $C^2\ssm \Delta_{12}$.
It is clear that $Z_\eps$ is homologous to $Z$ in $\inj_2(C)$. 

\begin{lemma}\label{lemma:integral}
The value of the cohomology class $[\zeta]$ on $Z_\eps$ 
(and hence the coefficient of $[\zeta]$ on $-\mu\times 1+1\times\mu$) equals $2\pi\sqrt{-1}$.
\end{lemma}
\begin{proof} The form pull-back of $\zeta$ to $Z_\eps$ is nonzero only the cylindrical part of $Z_\eps$  and so 
\[
\textstyle \la [\zeta],Z_\eps \ra= \int_{Z_\eps} \zeta=\int_{[-\eps, \eps]\times [0, 2\pi]} u_\eps^*\zeta
\]
In order to compute the latter, we  note that 
\[
u_\eps^*dz_1= e^{\sqrt{-1}\phi} (-ds +\sqrt{-1}(\eps-s)d\phi) , \quad u_\eps^*dz_2= 
e^{\sqrt{-1}\phi} (-ds -\sqrt{-1}(\eps+s)d\phi)
\]
so that $u_\eps^*(dz_1\wedge dz_2)=2\sqrt{-1}\eps e^{2\sqrt{-1}\phi}ds\wedge d\phi$. 
On the other hand, $z_1-z_2= 2\eps e^{2\sqrt{-1}\phi}$ and so it follows that
\[
u_\eps^*\zeta =\big(\frac{\sqrt{-1}}{2\eps} + O(\eps)\big)ds\wedge d\phi\; \text{ and hence  }\;
 \int_{[-\eps, \eps]\times [0,2\pi]} u^*\zeta= 2\pi\sqrt{-1} +o(\eps).
\]
The left hand side of the integral  is independent of $\eps$, and so $\la [\zeta],Z_\eps\ra= 2\pi\sqrt{-1}$.
\end{proof}

Any $K=\sum \alpha_i\otimes_\CC\beta_i\in H^1(C)\otimes H^1(C)$  determines endomorphisms 
$E_K$ and $ {}_KE$ of  $H^1(C)$ defined by the formulae 
\[
\textstyle \sum_i\alpha_i \times (\beta_i\cup x)=E_K(x)\cup 
\mu,\quad \sum_i(\alpha_i\cup x) \times\beta_i=\mu\times {}_KE(x).
\]
Note that $E_{\sigma^*K}=-{}_KE$. In particular, if $\sigma^*K=-K$ 
(which means that $\sum_i \alpha_i\times\beta_i\in H^2(C^2)$ is $\sigma$-invariant), then  $E_K={}_KE$.
Our main example will be the following: let $\omega_1,\dots , \omega_g$  be an orthonormal basis of 
$H^0(C, \Omega_C)$ in the sense that $\int_C \omega_i\wedge \overline\omega_j=\delta_{ij}$, so that
\[
\textstyle W:=\omega_i\otimes \bar\omega_i+\bar\omega_i\otimes \omega_i.
\] 
represents the `inverse' hermitian form on the dual of $H^0(C, \Omega_C)$. Note that this is a real $(1,1)$-form
which satisfies $\sigma^*W=-W$ so that $E_W={}_W E$. We also find that 
\[
E_W(\omega_j)=\sum_i \omega_i(\int_C \bar\omega_i\wedge \omega_j)=-\omega_j\quad ;\quad 
E_W(\bar\omega_j)=\sum_i\bar\omega_i(\int_C \omega_i\wedge \bar\omega_j)=\bar\omega_j
\]
and so $E_W$ acts as multiplication with $-1$ resp.\  $1$ on
$H^{1,0}(C)$ resp.\ $H^{0,1}(C)$. In other words (at least  for the convention of sign used nowadays), 
$\sqrt{-1}E_W$ is the Weil operator $J_C$ acting in $H^1(C; \RR)$. Thus, $W$ describes in  fact the full 
Hodge decomposition of $H^1(C; \CC)$.  

The $2$-form $\zeta$ that appears in the proposition below was found by Colombo-Frediani-Ghigi in 
\cite{CFG} in their study of the local geometry of the period map. 
We here obtain this form in a somewhat different manner and also determine its periods.

\begin{proposition}\label{prop:uchar}
There exists a $\zeta$  whose class lies in $H^2(C^2\ssm \Delta_{12}; \RR(1))$  and is given by
\[
2\pi\sqrt{-1}(-\mu\times 1+1\times\mu +W),
\]
where  $W=\sum_{i=1}^g (\omega_i\otimes \bar\omega_i+\bar\omega_i\otimes \omega_i)$ 
(so that $\sqrt{-1}W$ represents $2\pi$ times the Weil operator). It  is the unique antisymmetric meromorphic 2-form on $C^2$ with polar divisor  $-2\Delta_{12}$ with biresidue $1$ along $\Delta_{12}$ and whose class is of type $(1,1)$.
\end{proposition}
\begin{proof}
If we are given a $p\in C$ and a meromorphic differential $\eta$ that is regular on $C\ssm\{p\}$, then it defines
an class  $[\eta]\in H^1(C\ssm\{p\}; \CC)=H^1(C; \CC)$. Its image in $H^1(C, \Ocal_C)$ is obtained as follows:
write in a punctured neighborhood of $p$ the form $\eta$ as $df$ (which is possible since the residue of 
$\eta$ at $p$ has to be zero) and then take the image of $f$ in the above quotient.
The  theorem of Stokes combined with the Cauchy residue formula implies the (well-known) identity
\begin{equation*}
\int_C [\eta]\wedge\omega =2\pi\sqrt{-1}\res_p f\omega.
\end{equation*}
We first show that  for $\zeta$ be as in the lemma,  $E_{[\zeta]}$ is the identity on $H^0(C,\Omega_C)$. 
Let $p$ and a chart $(U;z)$ centered at $p$ be as above so that $\zeta$ is on $U^2$  as above ($\dagger$).
It is also clear that $\zeta|U\times C$ is of the form  $dz_1\wedge \pi_2^*\eta$, 
where $\eta$  is a meromorphic differential relative to the projection $U\times C\to U$ with 
$\eta|U^2=(z_1-z_2)^{-2}dz_2$ and $\eta$ regular elsewhere.
 Integration of $\eta$ on $U^2$ with respect to  $z_2$ yields $-(z_1-z_2)^{-1}$ plus a holomorphic  function on $U^2$.
Let $\omega$  be an abelian differential on $C$. So $\omega|U=f(z)dz$ for some holomorphic $f$.
The residue pairing relative to the second coordinate (with $z_1$ as parameter) yields
\begin{multline*}
\int_C [\eta]\wedge \omega  =2\pi\sqrt{-1}\res_{z_2\to z_1} -(z_1-z_2)^{-1} \omega = \\
=2\pi\sqrt{-1}\res_{z_2\to z_1} (z_2-z_1)^{-1} f (z_2)dz_2= 2\pi\sqrt{-1}f(z_1)
\end{multline*}
If we apply $dz_1\wedge$ to both sides, we find that $E_{[\zeta]}$ takes $\omega$ to 
$2\pi\sqrt{-1}\omega=-2\pi \sqrt{-1}E_W(\omega)$.
We now write the component of $\zeta$ in $H^1(C;\CC)\otimes H^1(C; \CC)$ as
\[
\textstyle \sum_{i,j} \big(a_{ij}\omega_i\otimes \omega_j +b_{ij} \omega_i\otimes \bar\omega_j
+s_{ij}\bar\omega_i\otimes \omega_j+d_{ij}\bar\omega_i\otimes \bar\omega_j\big).
\]
The assumption that  $\sigma^*\zeta=-\zeta$ implies that $a_{ij}=a_{ji}$, $b_{ij}=s_{ji}$ and $d_{ij}=d_{ji}$.
From what we proved, it follows that $b_{ij}$ is $2\pi\sqrt{-1}$ times the identity matrix and that $d_{ij}=0$.
Since we have the freedom of  choosing  the $a_{ij}$'s arbitrary, the unique representative in question is then obtained
by taking each $a_{ij}=0$. If we combine this Lemma \ref{lemma:integral}, the assertion follows. 
\end{proof}

\begin{remark}\label{rem:}
We prefer to think of $\zeta$ as a polydifferential. It  is  then $\sigma$-invariant and  this  makes it more canonical, 
as unlike the $2$-form interpretation, it does not depend  on how we ordered the factors. 
Note that for  $C=\PP^1$, $\zeta$ is simply given by $(z_1-z_2)^{-2}dz_1 dz_2$ (so that its divisor is minus twice  
the diagonal). 
 
\end{remark}

\begin{remark}\label{rem:}
At each point $x\in C$ we can  find a local holomorphic coordinate $z$ such that $\zeta$ takes at $(x,x)$ the simple form 
$(z_1-z_2)^{-2}dz_1dz_2$. 
Such a coordinate is not unique, but any other coordinate with that property is necessarily a fractional linear 
(M\"obius) transform of $z$. In other words,
$\zeta$ determines a projective structure at $x$. It is clearly invariant under the automorphism group of $C$. 
This group acts transitively when $C$ has genus 
$0$ or $1$, and so this projective structure must then be the standard one.  
When the genus $g$ of $C$ is $>1$, then the universal cover of $C$ is realized by the upper half plane $\HH$   
with covering group contained in $\PSL_2(\RR)$. So this defines another  projective structure on $C$. 
But Biswas-Colombo-Frediani-Pirola \cite{bcfp} recently showed that these two projective structures differ for higher genus.
The  difference between two projective structures is a quadratic differential on $C$, and  a quadratic differential on 
$C$ defines a covector to type $(1,0)$
at the base point of its universal deformation. As this  is canonically defined, we thus find a form of type $(1,0)$ 
on the moduli stack  of compact Riemann genus $g$-surfaces. It seems worthwhile to explore its properties.
\end{remark}

The biresidue along the diagonal defines the short exact sequence
\[
0\to \Omega_C( P)^{(2)}\la\llog\ra\to \Omega_C( P)^{(2)}\la\qlog\ra\xrightarrow{\br} \Delta_{12*}\Ocal_{C}\to 0.
\]
It follows from the preceding that
the sequence of global sections
\[
0\to H^0(C^2,\Omega_C( P)^{(2)}\la\llog\ra)\to H^0(C^2,\Omega_C( P)^{(2)}\la\qlog\ra)\to \CC\to 0,
\]
is still  exact with a splitting represented by $1\in \CC\mapsto \zeta\in H^0(C^2,\Omega_C( P)^{(2)}\la\qlog\ra)$.

\begin{corollary}\label{cor:}
We have natural $\sigma$-equivariant isomorphisms 
\begin{align*}
H^0(C^2, \Omega_C( P)^{(2)}\la\llog\ra)\cong &\, F^2H^2(\inj_2(C\ssm P))\otimes \sign_2,\\
H^0(C^2, \Omega_C( P)^{(2)}\la\qlog\ra)\cong &\, \CC \zeta \oplus F^2H^2(\inj_2(C\ssm P))\otimes\sign_2,
\end{align*}
with $\zeta$ regarded as an element of $F^1H^2(\inj_2(C\ssm P))\otimes\sign$. 
The projection on the first resp.\ second summand is a global version of the local operator 
$r''_{12}$ resp.\ $r'_{12}$ introduced earlier. 
\end{corollary}

\subsection{Canonical quasi-logarithmic polydifferentials}\label{subsect:canqlog} 
We will here find a higher order generalization of the canonical bidifferential that we introduced in the previous section. 
With the help of this we can express every quasi-logarithmic polydifferential canonically in terms of logarithmic ones.

Consider for $n\ge 2$ the divisors 
\[
D_n:=\Delta_{12}+\Delta_{23}+\cdots +\Delta_{n-1,n},\quad 
E_n:=D_n+\Delta_{n,1}=\Delta_{12}+\Delta_{23}+\cdots +\Delta_{n-1,n}+\Delta_{n,1}
\]
(so that $E_2=2\Delta_{12}$). Note that $D_n$ is a normal crossing divisor, and that $E_n$ is nearly so: 
it has this property away from the main diagonal. According to Deligne \cite{deligne:hodge2}  
this implies  that  $H^q(C^n, \Omega^n_{C^n}(D_n))$ can for $q\ge 0$  be identified with a  term
of the Hodge filtration on the cohomology of $C^n\ssm D_n$, namely $F^nH^{n+q}(C^n\ssm D_n)$. 
The second assertion of the following lemma shows that this mixed Hodge structure is in fact pure  
(and therefore implies formally its first assertion).

We recall that for $i,j\in \{1,\dots, n\}$ distinct, the class $\Delta_{ij!}(1)$ of 
$\Delta_I{ij}$ in $H^\pt(C^n)=H^\pt(C)^{\otimes n} $ equals 
$\pi_{ij}^*(\mu\times 1+\delta +1\times\delta)=\pi_i^*\mu +\pi_{ij}^*\delta+\pi_j^*\mu$.

\begin{lemma}\label{lemma:string}
The inclusion $\Omega^n_{C^n}\subset\Omega^n_{C^n}(D_n)$ defines an isomorphism on global sections. 
The inclusion $C^n\ssm D_n\subset C^n$ identifies $H^\pt (C^n\ssm D_n)$ with the quotient of 
$H^\pt(C^n)$ by the ideal generated
by the classes $\Delta_{i, i+1!}(1)$ of the irreducible components of $D_n$,  ($i=1, \dots, n-1$). 
Concretely,  in the K\"unneth decomposition $H^\pt(C^n)\cong H^\pt(C)^{\otimes n}$ the subsum of the tensor products
not involving $\pi_i^*(\mu)$, for $i=2,3, \dots, n$ maps isomorphically onto $H^\pt (C_n\ssm D_n)$.
\end{lemma}
\begin{proof}
For the first statement it suffices to show that the inclusion in question defines an isomorphism when taking the direct image 
under the projection $\pi^n: C^n\to C^{n-1}$ which omits the last factor. This is clear: 
this defines a curve over the generic point of 
$C^{n-1}$ with the divisor $D_n$ defining a single section and the assertion thus becomes a consequence of the 
residue theorem.

For the second assertion, we note that for $n=2$ this is a straightforward consequence of the fact that the Gysin 
sequence for the diagonal embedding splits up in short exacts sequences.
So we assume $n>2$ and, proceeding with induction, that $H^\pt (C^{n-1})\to H^{\pt}(C^{n-1}\ssm D_{n-1})$ is 
onto with kernel generated by the classes of irreducible components of $D_{n-1}$. 
Consider the Gysin sequence associated to the closed embedding 
$C^{n-1}\ssm D_{n-1}\hookrightarrow C^n\ssm D_n$ induced by $\Delta_{n-1, n}$:
\[
\cdots \to H^{\pt-2}(C^{n-1}\ssm D_{n-1})(-1)\xrightarrow{\Delta_{n-1, n!}} 
H^{\pt}((C^{n-1}\ssm D_{n-1})\times C)\to H^\pt(C^n\ssm D_n)\to\cdots
\]
The coboundary in this sequence is the homomorphism of $H^\pt(C^{n-1}\ssm D_{n-1})$-modules 
\[
H^{\pt-2}(C^{n-1}\ssm D_{n-1})(-1)\to H^{\pt}(C^{n-1}\ssm D_{n-1})\otimes H^\pt(C)
\]
which sends $1$ to 
the image of $\Delta_{n-1, n!}(1)=\pi_{n-1}^*\mu +\pi_{n-1,n}^*\delta +\pi_n^*\mu$.
Our inductive description of $H^{\pt}(C^{n-1}\ssm D_{n-1})$ shows that this map is injective,  so that its cokernel 
gives $H^{\pt}(C^{n}\ssm D_{n})$. This proves that $H^\pt(C^n\ssm D_n)$ is as asserted.
\end{proof}

The following proposition may be regarded as a generalization of  the construction of our canonical bidifferential on 
$\inj_2(C)$.

\begin{proposition}\label{prop:cyclicclass}
For $n\ge 2$, the space  $H^0(C^n,\Omega^n_{C^n}(E_n))$ embeds in 
$H^n(C^n\ssm E_n; \CC)$ and lands in $F^{n-1}H^n(C^n\ssm E_n)$.
Its subspace $H^0(C^n,\Omega^n_{C^n})$ is of codimension one and maps isomorphically onto $F^nH^n(C^n\ssm E_n)$. 
There is a unique supplementary line of $H^0(C^n,\Omega^n_{C^n})$ in 
$H^0(C^n,\Omega^n_{C^n}(E_n))$ that maps to a line in $H^n(C^n\ssm E_n; \CC)$  of Hodge type $(n-1, n-1)$ 
and that is spanned by a real class.  The $\Scal_n$-stabilizer of $E_n$ (a dihedral group of order $2n$) acts on this line 
with the character of the expression $(z_1-z_n)(z_2-z_1)\cdots (z_n-z_{n-1})dz_1\wedge\cdots \wedge dz_n$.
For $n\ge 3$, it has a unique generator $\zeta_n$ with the property that the residue along $\Delta_{n,n-1}$ is 
equal to $\zeta_{n-1}$, where $\zeta_2=\zeta$ is the $2$-form defined earlier.
\end{proposition}
\begin{proof}
We already established this for $n=2$. We therefore assume $n>2$ and the proposition verified for $n-1$.
Taking the residue along $\Delta_{n,1}$ gives for $n\ge 3$ the exact sequence
\[
0\to \Omega^n_{C^n}(D_n)\to\Omega^n_{C^n}(E_n)\to \Delta_{n,1 *}\Omega^{n-1}_{C^{n-1}}(D_{n-1})\to 0
\]
whose associated  long exact sequence begins with 
\[
0\to H^0(\Omega^n_{C^n}(D_n))\to H^0(\Omega^n_{C^n}(E_n))\to 
H^0(\Omega^{n-1}_{C^{n-1}}(E_{n-1}))\to H^1(\Omega^n_{C^n}(D_n))
\to\dots 
\]
This sequence  is compatible with the Gysin sequence for the closed embedding 
$\Delta_{n,1}: C^{n-1}\ssm E_{n-1}\subset C^n\ssm D_n$ in the sense that we have morphism of exact sequences
\begin{small}
\begin{center}
\begin{tikzcd}[column sep=tiny]
{}\arrow[r] &H^n(C^n\ssm D_n;\CC)\arrow[r]& 
H^n(C^n\ssm E_n; \CC)\arrow [r]& H^{n-1}(C^{n-1}\ssm E_{n-1}; \CC)(-1)\arrow[r] & 
H^{n+1}(C^n\ssm D_n; \CC)\arrow[r]&{}\\
0\arrow[r] &H^0(\Omega^n_{C^n}(D_n))\arrow[r]\arrow[u] & H^0(\Omega^n_{C^n}(E_n))\arrow [r]\arrow[u] 
& H^0(\Omega^n_{C^{n-1}}(E_{n-1}))\arrow[r]\arrow[u] & H^1(\Omega^n_{C^n}(D_n))\arrow[r]\arrow[u] &{}
\end{tikzcd}
\end{center}
\end{small}
\noindent
The top sequence is one of mixed Hodge structures. Since $D_n$ is a normal crossing divisor, the first and the fourth up 
arrow are embeddings with image 
$F^nH^n(C^n\ssm D_n)$ resp.\ $F^nH^{n+1}(C^n\ssm D_n)$. According to our induction hypothesis, 
the third up arrow  embeds 
$H^0(\Omega^n_{C^{n-1}}(E_{n-1}))$ in $F^{n-2}(H^{n-1}(C^{n-1}\ssm E_{n-1}))(-1)=
F^{n-1}(H^{n-1}(C^{n-1}\ssm E_{n-1})(-1))$ and its 
image in $\Gr_F^{n-1}(H^{n-1}(C^{n-1}\ssm E_{n-1})(-1))$ is one dimensional. 
Since the (Hodge) functor $F_{n-1}$ is exact, it follows that the second up arrow is an embedding whose  
image lies  in $F^{n-1}H^n(C^n\ssm E_n)$ and  contains $F^nH^n(C^n\ssm E_n)$ as a subspace of codimension one. 
We also see that the residue map defines an isomorphism  
\[
H^0(\Omega^n_{C^n}(E_n))/H^0(\Omega^n_{C^n})\cong 
H^0(\Omega^n_{C^{n-1}}(E_{n-1}))/H^0(\Omega^{n-1}_{C^{n-1}}).
\]
Since the latter is of Tate type $(n-2,n-2)$, the former is of Tate type $(n-1, n-1)$.
The uniqueness of the line follows from the fact that it is spanned by preimage of the real cohomology 
$H^n(C^n\ssm E_n; \RR)$.

It remains to determine the character of the $\Scal_n$-stabilizer of $E_n$  on the corresponding  line of polydifferentials. 
Near the main diagonal, any generator  of that line  has in terms of a local coordinate $z$ on $C$ the form
\[
\frac{f(z_1,\dots, z_n)dz_n \wedge dz_{n-1}\cdots \wedge dz_1}{(z_1-z_n)(z_2-z_1)\cdots (z_n-z_{n-1})}.
\]
with $f$ holomorphic and invariant under cyclic permutation and  $f(z,z,\dots, z)$ constant nonzero. 
 We let $\zeta_n$ be the  generator such that this constant is $1$. Then it is clear that
the residue of $\zeta_n$ along $\Delta_{n-1, n}$ is $\zeta_{n-1}$. The fact that it transforms under the dihedral group as indicated is clear.
\end{proof}

\begin{question}\label{quest:}
It follows from Proposition \ref{prop:uchar} that the Hodge  structure on $H^1(C)$ can be read off from
the cohomology class of $\zeta$ in $H^2(C^2\ssm \Delta_{12}; \RR(1))$. So changing the complex structure
on $C$ will change this cohomology class. The same applies to the cohomology class of $\zeta_n$ ($n\ge 2$), which lies in $H^2(C^n\ssm E_n; \RR(n))$. What is the span of these classes? The answer should preferably be one in topological terms.
\end{question}

We state the preceding proposition in terms of  polydifferentials alone. We denote by $\zeta_{(n\, n-1\, \dots 1)}$ the 
polydifferential defined by $\zeta_n$: this is the polydifferential which along the main diagonal of $C^n$ is locally given by
\[
\frac{f(z_1,\dots, z_n)dz_n dz_{n-1}\cdots dz_1}{(z_1-z_n)(z_2-z_1)\cdots (z_n-z_{n-1})}.
\]
with $f$ holomorphic and invariant under cyclic permutation and  $f(z,z,\dots, z)$ constant $1$. So then we have 
\[
r_{n,n-1}\zeta_{(n\, n-1\, \dots 1)}= \zeta_{(n-1\, \dots 1)}.
\]

It is clear that $\zeta_{(n\, n-1\, \dots 1)}$ is invariant under cyclic permutation. 
A reflection $\tau$ (for instance the one given by 
$(1\,n)(2\, n-1)(3\, n-2)\cdots $)  will change the sign in every factor of the denominator and so 
$\tau^*\zeta_{(n\, n-1\, \dots 1)}=(-1)^n\zeta_{(n\, n-1\, \dots 1)}$.  
Equivalently,  if $\G$ is an polygon with vertex set $J$, then an orientation of $\G$  
determines a quasi-logarithmic polydifferential 
in $H^0(C^J, \Omega^{(J)}\la \qlog\ra)$ which gets multiplied by $(-1)^{|J|}$  under reversal of the orientation. 
In other words, we have a well-defined element 
\[
\zeta_\G\in H^0(C^J, \Omega^{(J)}\la \qlog\ra)\otimes\sign(\edge(\G)).
\]
This generalizes immediately to the case when $\G$ is disjoint union of polygons: we let $\zeta_\G$ 
simply be the (exterior) product of the (sign twisted) polydifferentials we just attached to its connected components. 
We thus find:

\begin{corollary}\label{cor:}
Let $\G$ be a disjoint union  of oriented polygons  with (finite) vertex set $I$.
If  $\Delta_\G$ stands for the sum of the diagonal divisors indexed by the edges of $\G$, then there is associated to 
$\G$ a unique polydifferential
\[
\textstyle \zeta_\G \in H^0(C^I, \Omega_C^{(I)}(\Delta_\G))\otimes\sign(\edge(\G)), 
\]
such that 
\begin{enumerate}
\item[(i)] the  line it spans in $H^{|I|}(C^I\ssm |\Delta_\G|; \CC)\otimes\sign(\edge(\G))$ is of Hodge bidegree 
$(|I|-b_0(\G),|I|-b_0(\G))$ (where $b_0(\G)$  is the number of connected components of $\G$), 
\item [(ii)] if $i,j\in I$ are distinct,  then  $s_{ij}(\zeta_\G)$ is zero unless $\{i, j\}$ supports a connected component of 
$\G$, in which case we get
$\zeta_{\G\ssm\{i,j\}}$,
 \item [(iii)] if $i,j\in I$ are distinct, then  $r_{ij}(\zeta_\G)$ is zero unless $(ij)$ is a simple edge of $\G$, 
 in which case it equals $\zeta_{\G_{ij}}$, where 
 $\G_{ij}$ stands for the graph obtained from $\G$ by contracting the  edge $(ij)$ (so with vertex set  $I_{ij}$).
\end{enumerate}

The map which assigns to the finite set $I$ the subspace of quasi-logarithmic polydifferentials on $C^I$ spanned by such 
$\zeta_\G$ (referred to as its  \emph{space of strictly quasi-logarithmic polydifferentials}) defines an $\Rscr^s$-module in 
the category of finite dimensional complex vector spaces, i.e., a functor $\qLog_C:\Rscr^s\to \Vect_\CC$.
\end{corollary}

We do not claim that the polydifferentials $\zeta_\G$ are linearly independent (indeed, they aren't in general).

We now have also a global counterpart of Corollary \ref{cor:etahat}. 
Let $\glie$ be a $k$-Lie algebra endowed with a symmetric bilinear form $s$ invariant under the adjoint action.
Recall that $s_N$ is a linear form on $\glie^{\otimes N}$ defined by  
$s_N(X_N\otimes\cdots \otimes X_1):=s([\cdots [X_N, X_{N-1}]\cdots , X_2], X_1)$. We regard
$\zeta_{C,N}:=-\sum_{\sigma\in \Scal_N} \sigma_*(s_N\zeta_{(N,N-1,\dots, 1)})$ as an element of 
$\Hom(\glie^{\otimes N}, H^0(C^N, \Omega^{(N)}\la\qlog\ra))$. This makes
also sense if we replace by a finite set $I$ and then the same argument for Corollary \ref{cor:etahat} establishes:

\begin{corollary}\label{cor:zetares}
A homomorphism of $\Rscr^s$-modules $\Lcal^{s}_{\gfrak}\to \qLog_C$ is defined by assigning to a finite set 
$I$ the homomorphism
\[
\textstyle \hat\zeta_I:=\sum_{\Pcal | I} \bigotimes_{P\in\Pcal} \zeta_{C,P}\in 
 \Hom(\glie^{\otimes I},H^0(C^N, \Omega^{(N)}\la\qlog\ra)),
\]
where the sum is over all partitions of $I$.
\end{corollary}

This also comes  with a global counterpart of Corollary  \ref{cor:etadef}:

\begin{corollary}\label{cor:zetadef}
Let $V$ be a $\glie^P$-representation and let $\xi\in \Hom_{\Rscr_P}(\Lcal_{\gfrak,P,V},\Lcal og_{C,P})$ 
(so this assigns to every finite set $I$ a linear map $\xi_I: \glie^{\otimes I}\otimes V\to 
H^0(C^I, \Omega_C(P)^{(I)}\la\llog\ra)$ subject to the usual conditions). Then a  $\Rscr_P^s$-homomorphism  
$\hat\xi: \Lcal^{s}_{\gfrak,P,V}\to \qLog_{C,P}$ is defined by assigning 
to every finite set $I$,  the  sum $\sum_{J\subset I} \xi_{I\ssm J}\otimes \hat\zeta_J:  
\glie^{\otimes I}\otimes V\to H^0(C^I, \Omega_C(P)^{(I)}\la\qlog\ra)$. The resulting map 
\[
\xi\in \hom_{\Rscr_P}(\Lcal_{\gfrak,P,V},\Lcal og_{C,P})\mapsto 
\hat\xi \in \hom_{\Rscr_P^s}(\Lcal^{s}_{\gfrak,P},\qLog_{C,P})
\]
is an isomorphism. 
\end{corollary}

\section{A loop algebra context}\label{sect:loopalg}
\subsection{Extended loop algebras}\label{subsect:affineLie}
 Here and in the remainder of the paper we fix a simple Lie algebra $\glie$ over $k$. 
 As is well-known, such a Lie algebra admits a nondegenerate symmetric bilinear form that is invariant under the 
 adjoint representation (for example, the Killing form). This is in fact a generator of the space of $\glie$-invariant 
 bilinear forms on $\glie$ and then similar statement holds of course for
the dual of $\glie$. We denote that 1-dimensional space of invariants by $\cfrak:=(\glie\otimes\glie)^{\glie}$. 
So if $\hlie\subset\glie$ is a Cartan subalgebra, then there is unique generator 
$c\in (\glie\otimes\glie)^{\glie}$ which takes  the value $2$ on the 
short roots (regarded as elements of $\hlie^*\subset \glie^*$). Since  $c$ induces an isomorphism 
$\glie^*\cong\glie$  that takes  $2c(\alpha, \alpha)^{-1}\alpha$ to the coroot $H_\alpha$, it follows that 
the `inverse form' $\check{c}$ has the property that $\check{c}(H_\alpha, H_\alpha)=2$ when $H_\alpha$ is a long root.
This generator  is independent of the choice of $\hlie$ because  the adjoint group acts transitively on the 
Cartan subalgebras of $\glie$. 

\begin{example}\label{example:}\label{ex:slin2}
For  $\gfrak=\slin (2)$ with its standard basis
$X_+=\left( \begin{smallmatrix}0 & 1\\ 0 & 0\end{smallmatrix}\right)$,
$X_-=\left(\begin{smallmatrix}0 &0\\1 &0\end{smallmatrix}\right)$, and
$H=\left(\begin{smallmatrix}1 &0\\0 &-1\end{smallmatrix}\right)$, the form $\check{c}$ is the Killing form: 
$\check{c}(X_\pm,X_\mp)=\tr_k(X_\pm X_\mp)=1$, $\check{c}(H,H)=\tr (H^2)=2$
and $\check{c}$ takes on all other basis pairs the value zero so that 
$c=X_+\otimes X_-+X_-\otimes X_++\tfrac{1}{2}H\otimes H$. 
\end{example}
 
Let $\Ocal$ be a DVR with residue field $k$ and denote its maximal ideal by $\mfrak$ and its field of fractions by 
$K$ as before. But we do not assume that $\Ocal$ is complete for the $\mfrak$-adic topology.
We put $\glie K:=\glie\otimes_k K$ and endow the direct sum of $\cfrak$ and $\glie K$ as $k$-vector spaces 
with a Lie bracket that makes it a central extension 
\[
0\to \cfrak\to \widehat{\glie K}\to \gfrak K\to 0
\]
by putting
\[
[Xf +\cfrak, Yg+\cfrak]:= [X,Y]fg + \check{c}(X,Y)\res (gd\! f) c.
\]
The $\mfrak$-adic topology makes $\widehat{\glie K}$ a topological Lie algebra and we shall regard it as such.
It comes with a Cartan subalgebra $\widehat{\hlie}:=\cfrak\times \hlie$. 
Let us fix a  system of simple roots $\alpha_1,\dots, \alpha_r$  in $\hlie^*$ and $(H_{\alpha_1}, \dots, H_{\alpha_r})$ the associated coroots in $\hlie$. 
 Let $X_{\pm \alpha_i}\in \glie$  be a nonzero $\hlie$-eigenvector of weight $\pm \alpha_i$ such that 
 $[X_{\alpha_i}, X_{-\alpha_i}]=H_{\alpha_i}$,  so that  $\glie$ has a standard (Serre)  presentation as a Lie algebra with 
 $\{X_{\alpha_i}, X_{-\alpha_i}\}_{i=1}^r$ as generators. If we choose a uniformizer $t$ of  $\Ocal$, 
 then this extends to a presentation of $\widehat{\glie K}$ with 
$\widehat\hlie$ playing the role of a Cartan subalgebra: if  
$\theta\in \hlie^*$  denotes the maximal root with respect to $\alpha_1,\dots, \alpha_r$, 
then we add to these generators the pair $\{X_{\alpha_0}:=tX_{-\theta},X_{-\alpha_0}:=t^{-1}X_\theta\}$. 
Since $[X_\theta, X_{-\theta}]$ equals the coroot $H_\theta$  associated to $\theta$ and 
$\theta$ is a long root, it follows that  $\check{c}(X_\theta,X_{-\theta})=\frac{1}{2}\check{c}(H_\theta, H_\theta)=1$ 
and so  $[X_{\alpha_0}, X_{-\alpha_0}]=-H_\theta + c$; we therefore  denote this last  element  by 
$H_{\alpha_0}$, although we have not defined a root $\alpha_0\in \widehat \hlie^*$ 
(this  makes better sense after enlarging  $\widehat{\hlie}$; see Remark \ref{rem:affinelie}).  If we write 
$H_{\theta}=\check n_1H_{\alpha_1}+\cdots +\check n_rH_{\alpha_r}$, then each coefficient 
$\check n_i$ will be a positive integer and if we put $\check n_0:=1$, then of course 
$c=\check n_0H_{\alpha_0}+\check n_1H_{\alpha_1}+\cdots +\check n_rH_{\alpha_r}$. 

 A linear form $\hat\lambda$ on $\hat\hlie$ is called an \emph{integral dominant 
weight} if  $\lambda_i:=\hat\lambda(H_{\alpha_i})$ is a nonnegative integer
for all $i=0, \dots , r$.
Let us  agree that if a $\glie$-module $V$ has a  highest weight $\mu$ (which is of course always the case when 
$V$ is finite dimensional) then to write $V^+$  for the $\glie$-submodule of $V$ generated by that weight space. This is then subspace that is killed by the root spaces 
 associated to the positive roots (the nilradical of the Borel subalgebra defined by our root data). 
 For example, is we regard $\glie$ as a $\glie$-module, then its highest weight is $\theta$ and 
 $X_\theta$ spans $\glie^+=\glie_\theta$.

The restriction of $\hat\lambda$ to $\hlie$, which we shall denote by $\lambda$ and which is given by 
$(\lambda_1,\dots, \lambda_r)$,  defines
an irreducible highest weight module $V_\lambda$ of $\gfrak$  with highest weight $\lambda$ whose highest weight space 
$V_\lambda^+=V_\lambda(\lambda)$ is a one-dimensional subspace of $V_\lambda$; 
we often denote a generator of $V_\lambda^+$ by $1_\lambda$.
In case we extend the action of $\hlie$ on $V_\lambda^+$ to one of $\widehat \hlie$ by letting it act with character 
$\hat\lambda$, then we write 
$V_{\hat\lambda}$ for $V_\lambda$. This means that $c$ then acts on this line  as multiplication with 
$\hat\lambda (c)=\lambda_0+\sum_{i=1}^r \check n_i\lambda_i$. We make $V_{\hat\lambda}$ a representation of 
$\cfrak\times\glie$ by letting $c$ act on all of $V_\lambda$ as scalar multiplication with this scalar $\hat\lambda (c)$.
We now note that the  central extension is trivial over $\glie\Ocal$, so that the reduction mod $\mfrak$ defines a 
Lie homomorphism $\widehat{\glie\Ocal}\to \cfrak\times\gfrak$. This turns $V_{\hat\lambda}$ into a 
representation of $\widehat{\glie\Ocal}$ that we continue to denote denote by $V_{\hat\lambda}$.  
We then form the induced representation
\[
\tilde \Vcal_{\hat\lambda}(K):=\U\!\widehat{\glie K}\otimes_{\U\!\widehat{\glie\Ocal}}V_{\hat\lambda}.
\]
We have fixed $\glie$, which is why we allowed ourself to suppress its presence in our notation. 
But the dependence on the discretely valued field $K$ will be important here and this explains why it appears instead.
The nonnegative integer $\hat\lambda(c)$ is called the \emph{level} of this representation; 
we may also denote it $\ell(\hat\lambda)$.  It is clear that to specify $\hat\lambda$ is equivalent to specifying a level $\ell$ 
and the dominant weight $\lambda:=\hat\lambda|\hlie$ of $\glie$ (given by $(\lambda_1, \dots, \lambda_r)$), 
the only requirement being that $\ell\ge \sum_{i=1}^r \check n_i\lambda_i$ (then $\lambda_0= 
\ell-\sum_i \check n_i\lambda_i$).

Note that $1_\lambda\in V_\lambda^+$  generates $V_\lambda$ as a $\glie$-module and 
$\tilde \Vcal_{\hat\lambda}(K)$ as a $\widehat{\glie K}$-representation.
In fact, if $\glie_-$ is the (standard) Lie  nilpotent subalgebra of $\glie$ generated by 
$X_{-\alpha_1}, \dots , X_{-\alpha_r}$, then $1_\lambda$  generates $V_\lambda$ as 
$\U\!\glie_-$-module and likewise, if $\widehat{\glie K}_-$  is the Lie subalgebra of 
$\widehat{\glie K}$ generated by  $X_{-\alpha_0}=t^{-1}X_{\theta}, X_{-\alpha_1}, \dots , X_{-\alpha_r}$, then
$\tilde\Vcal_{\hat\lambda} (K)$ is generated by $1_\lambda$ as a  $\U\!\widehat{\glie K}_-$-module. 
One may check that $\widehat{\glie K}_-=\glie [t^{-1}]t^{-1}+\glie_-$.  

When $\widehat{\glie K}$ or $\tilde\Vcal_{\hat\lambda}$ is considered as a $\glie$-module, then each finite 
subset generates a finite dimensional 
$\glie$-submodule. This implies that the Lie subalgebra $\nfrak_-\subset\glie$ (generated by the negative coroots 
$X_{-\alpha_1},\dots , X_{-\alpha_r}$) acts  in a locally nilpotent manner in the sense that for every  
$X\in  \nfrak_-$, the powers of $X$ eventually kill every element of that $\glie$-module.

\begin{remark}\label{rem:affinelie}
In order to make this an \emph{affine} Lie algebra $\widetilde{\glie K}$ of Kac-Moody type, one must enlarge 
$\widehat{\glie K}$ by taking the semi-direct product
with the $k$-span of a derivation $D_0$ of $\Ocal$. This derivation however requires the choice of a uniformizer  
(at least up to scalar) $t$ of $\Ocal$:
one then takes $D_0=t\frac{d}{dt}$. This makes $\widetilde\hlie:=\hlie\times  \cfrak\times kD_o$ a 
Cartan subalgebra of $\widetilde{\glie K}$ and then a root $\alpha_0$ is defined. 
Since this depends on the choice of $t$, this  is not canonical. But the Segal-Sugawara 
construction that we will recall below provides a natural coordinate invariant substitute 
(it will then be semi-direct product with a Virasoro algebra).
\end{remark}

\subsubsection*{The irreducible representation $\Vcal_{\hat\lambda} (K)$}
The $\widehat{\glie K}$-representation $\tilde\Vcal_{\hat\lambda} (K)$ has an irreducible quotient with highest 
$\hat\hlie$-weight $\hat\lambda$. According to \cite{kac} it is obtained by dividing out 
by the subrepresentation generated by $X_{-\alpha_0}^{1+ \lambda_0}1_\lambda=
(X_{\theta}t^{-1})^{1+ \lambda_0}1_{\lambda}$:
\[
\Vcal_{\hat\lambda} (K)= \tilde \Vcal_{\hat\lambda}(K)/
U\!\widehat{\glie K} (X_{\theta}t^{-1})^{1+ \lambda_0}1_\lambda. 
\]
(Since $\glie_{\theta}\Ocal$ already kills $1_\lambda$, this can be somewhat more invariantly 
be defined by dividing out by the left ideal generated by 
$(\glie_{\theta}\mfrak^{-1})^{1+\lambda_0}\otimes V_\lambda^+\in \tilde \Vcal_{\hat\lambda}(K)$.) 
This forces all of $\widehat{\glie K}_-$  to act locally nilpotently on $\Vcal_{\hat\lambda}(K)$. 
This has the following useful consequence.

\begin{lemma}\label{lemma:finiteAgen}
Let $A\subset K$ be a $k$-subalgebra such that $K/(A+\Ocal)$ is of finite $k$-dimension. 
Then $\Vcal_{\hat\lambda} (K)$ is a finitely generated 
$\U\!\widehat{\glie A}$-module.
\end{lemma}
\begin{proof}
Let $f_1, \dots , f_N$ be a finite collection of negative powers of $t$  whose images in   $K/(A+\Ocal)$ 
generate that space over $k$. Then $\widehat{\glie K}=\widehat{\glie A} + \glie\mfrak +\sum_{i=1}^N \glie f_i$, 
where $\widehat{\glie A}$ and $\glie\mfrak$ are subalgebras. A standard Poincar\'e-Birkhoff-Witt type of argument 
shows that
\[
\textstyle \U\!\widehat{\glie K}=\U\!\widehat{\glie A} \Big(\sum_{(m_1, \dots, m_N)\in 
\ZZ_{\ge 0}^N} (\glie f_N)^{m_N}\cdots (\glie f_1)^{m_1} \Big)\U (\glie\mfrak). 
\]
Since $\U (\glie\mfrak)$ annihilates  $1_\lambda$, the space  $\Vcal_{\hat\lambda} (K)$ is as a 
$\U\!\widehat{\glie A}$-module generated by the collection
$\{\glie f_N)^{m_N}\cdots (\glie f_1)^{m_1}V_\lambda^+\}_{\underline m}$. 
But each subspace $\glie f_i$ acts  locally nilpotently on $\Vcal_{\hat\lambda} (K)$, and so it follows that
this collection spans a finite dimensional $k$-vector space.
\end{proof}

Note  that after forming this irreducible quotient, the notion of level subsists. 
An alternative characterization of $\Vcal_{\hat\lambda} (K)$ is that it is the irreducible representation of 
$\widehat{\glie K}$ with highest weight $\hat\lambda$, which means that the $\hat\lambda$-weight space of 
$\hat\hlie$  is of dimension one  (in the above situation we can  take this to be  $V_\lambda^+$) and 
is killed by $\widehat{\glie K}_+=\glie_++\glie\Ocal$. Then every other weight of  $\hat\hlie$ lies in 
$\hat\lambda+\ZZ_{\le 0}\{\alpha_0, \alpha_1, \dots, \alpha_r\}$. In that case, the $\glie$-submodule of  
$\Vcal_{\hat\lambda} (K)$ generated by  $V_{\hat\lambda}^+$ is an irreducible  $\glie$-module  
with highest weight  $\lambda$, in other words, a copy of $V_\lambda$.

Given a topological  Lie algebra $\klie$, denote by $\PBW_\pt\klie$ the standard filtration of the universal 
enveloping algebra $\U\!\klie$ (so $\PBW_N\klie$ is the image 
of $\klie^{\otimes N}$). It is of course exhaustive. We give each term $\PBW_N\U\!\klie$ the product topology, and regard
$\U\!\klie$  as an inductive limit of topological vector spaces. The Poincar\'e-Birkhoff-Witt  theorem asserts that 
the associated graded algebra is the symmetric algebra  of $\klie$ regarded as a topological $k$-vector space. 
For $\widehat{\glie K}$ this yields:

\begin{lemma}\label{lemma:PBW}
Let $\PBW_\pt\tilde\Vcal_{\hat\lambda}(K)$ be the filtration on $\tilde\Vcal_{\hat\lambda}(K)$ be defined by 
$\PBW_N\tilde\Vcal_{\hat\lambda}(K)$ letting be the image of 
$\PBW_N\otimes V_\lambda\to  \Vcal_{\hat\lambda}(K)$. Then the natural map 
\[
\sym_\pt (\glie K/\glie\Ocal)\otimes_k V_\lambda\to \Gr^{\PBW}_\pt \tilde\Vcal_{\hat\lambda}(K)
\]
is an isomorphism of topological $k$-vector spaces.
\end{lemma}
\begin{proof} 
Since $\tilde\Vcal_{\hat\lambda}(K)$  is an induced
representation of $\widehat{\glie\Ocal}$ on $V_\lambda$ and 
$\widehat{\glie K}/\widehat{\glie\Ocal}=\glie K/\glie \Ocal$, this follows from the Poincar\'e-Birkhoff-Witt  theorem.
\end{proof}

\begin{corollary}\label{cor:}
If $K\to \bar K$ is the $\mfrak$-adic completion, then the natural maps 
$\tilde \Vcal_{\hat\lambda}(K)\to \tilde \Vcal_{\hat\lambda}(\bar K)$ and 
$\Vcal_{\hat\lambda}(K)\to \Vcal_{\hat\lambda}(\bar K)$  are  isomorphisms as  topological 
$\widehat{\glie K}$-representations (via $\widehat{\glie K}\to \widehat{\glie \bar K}$).
\end{corollary}
\begin{proof}
This indeed follows from Lemma \ref{lemma:PBW} above and the fact that $K/\Ocal\to \bar K/\bar \Ocal$ is a 
$k$-linear isomorphism.
\end{proof}

\begin{remark}\label{rem:}
The action of $c$ on $V_\lambda$ should not be confused with that of the associated Casimir element, that is, 
the image $\bar c$ of $c$  in the enveloping algebra $\U\!\glie$. This element acts on every irreducible representation of 
$\gfrak$ with highest weight $\lambda$ as a scalar 
\[
c_\lambda=c(\lambda , \lambda+2\rho),
\]
where $\rho\in \hlie^*$ is the half sum of the positive roots, also characterized by $\rho(H_{\alpha_i})=1$ for 
$i=1, \dots, r$.
(This formula is well-known for the standard Casimir element of $\glie$; it then evidently also holds for any scalar multiple 
of it, such as $\bar c$.)
In particular,  $c_\lambda$ is a positive rational number (the denominator is in fact at most 3). 
For example, the adjoint representation $\gfrak$ has highest weight $\theta$. 
Since $\theta$ is a long root, we have $c(\theta, \chi)=\chi(H_\theta)$ for $\chi\in \hlie^*$. 
So $\bar c$ then acts on $\glie$ as scalar multiplication with 
\[
c(\theta , \theta+2\rho)= 2+ 2\rho(H_{\theta})=2(\check n_0 +\check n_1+\cdots  +\check n_r)=2\check h,
\]
where $\check h:=\check n_0 +\check n_1+\cdots  +\check n_r$ is called the \emph{dual Coxeter number} 
of $\glie$ (in case  all simple roots have the same length,  this is the order of a Coxeter transformation in the 
associated Weyl group).
If  $\{ E_\kappa\}_\kappa$ is a basis of $\glie$ that is orthonormal with respect to $\check{c}$, then 
$c=\sum_\kappa E_\kappa\otimes_k E_\kappa$. Hence 
$\bar c=\sum_\kappa E_\kappa\circ E_\kappa$ and so we have $\sum_\kappa [E_\kappa , [E_\kappa, Y]]=2\check{h}Y$.   
\end{remark}

\subsection{The dual of a highest weight representation}\label{subsect:dualrep}
We assume in this subsection that $\Ocal$ is complete. We denote by $\tilde\Vcal^{\hat\lambda}(K)$ resp.\ 
$\Vcal^{\hat\lambda}(K)$  the topological dual of
$\tilde\Vcal_{\hat\lambda}(K)$ resp.\ $\Vcal_{\hat\lambda}(K)$. We also put 
\[
\tilde\Vcal_\qlog^{\hat\lambda}(K):=\hom_{\Rscr^s_*}(\Lcal^{\ell\check{c}}_{\glie, V_\lambda}, \qLog_K).
\] 
The notation will be justified once we identify this with  $\tilde\Vcal^{\hat\lambda}(K)$.
We could refer to Subsection \ref{subsect:lierep} for what  a homomorphism $\Lcal^{\ell\check{c}}_{\glie, V_\lambda}\to \qLog_K$ is, 
but let us here unpack the definition: it 
assigns to every nonempty finite set $I$ a linear map 
$\xi_I: \gfrak^{\otimes I}\otimes V_\lambda\to \omega^{(I)}\la\qlog\ra$ subject
to certain conditions. Spelled out, this amounts  to giving  a graded linear map
\[
\xi=\big(\xi_N: \gfrak^{\otimes N}\otimes V_\lambda\to\omega^{(N)}\la\qlog\ra_o)\big)_{N=0}^\infty
\] 
such that $\xi_N$ is  $\Scal_N$-equivariant and
\begin{gather*}\tag{$i$}
r_1(\xi_{N+1}(X_{N+1}\otimes\dots \otimes X_1\otimes v))=\xi_{N}(X_{N+1}\otimes\dots \otimes X_2\otimes X_1v),\\
\tag{$ii$}
r_{N+1,N+2} (X_{N+2}\otimes X_{N+1}\otimes\dots \otimes X_1\otimes v)=
\xi_{N+1}([X_{N+2},X_{N+1}]\otimes\cdots\otimes X_1\otimes v),\\
\tag{$iii$} s_{N+2,N+1} (\xi_{N+2}(X_{N+2}\otimes X_{N+1}\otimes \dots \otimes X_1\otimes v))=
\ell\check{c}(X_{N+2},X_{N+1})\xi_{N}(X_{N}\otimes\dots\otimes X_1\otimes v).
\end{gather*}

Since the pairs $X\otimes Y\in \gfrak\otimes \gfrak$ with $[X,Y]\not=0$ span 
$\gfrak\otimes \gfrak$, we see that $\xi_N$ determines $\xi_{N-1}$ and hence (with induction) all its predecessors 
$\xi_0,\dots ,\xi_{N-1}$. 
In particular, if $\xi_N$ is nonzero, then all of its successors are. On the other hand, the residues of $\xi_N$ are given by 
$\xi_1, \dots, \xi_{N-1}$, so that these determine $\xi_N$ up to a regular polydifferential, i.e., an element of 
$(F^1\omega)^{(N)}$.

Notice that property ($i$) implies that if $X_1v=0$, then $\xi_N$ has no pole at $t_1=0$. 
Because of the $\Scal_N$-equivariance, this
implies that more generally, $\xi_N$ has no pole at $t_i=0$ if $X_i(v)=0$. Of particular interest will be the value of 
$\xi$ on a tensor of the form $Z\otimes X_\theta^{\otimes s}\otimes I_\lambda$, where $Z$ is in  the tensor algebra of 
$\glie$: 
$\xi(Z\otimes X_\theta^{\otimes s}\otimes v)$ will have no pole in the generic point of $t_1=\cdots =t_s=0$ and 
hence has a well-defined restriction to that locus.

\begin{example}\label{ex:}
We work this out for $N\le 2$. Clearly, $\xi_0\in V^*$. If $t$ is a uniformizer for $\Ocal$, then
\[
\xi_1(X\otimes v )=\xi_0 (Xv ) \frac{dt}{t}+\eta_1(X\otimes v )
\]
for some $\eta_1(X\otimes v )\in F^1\omega$.
The polar part of $\xi_2(X_2\otimes X_1\otimes v )$ is 
\begin{multline*}
\ell \check c(X_2\otimes X_1)\xi_0(v )\frac{dt_1dt_2}{(t_1-t_2)^2} +
\pi_1^*\xi_1 ([X_2,X_1]\otimes v ) \frac{dt_2}{t_1-t_2}+ \\
+\pi_2^*\xi_1 (X_2\otimes X_1v )\frac{dt_1}{t_1} +\pi_1^*\xi_1 (X_1\otimes X_2v )\frac{dt_2}{t_2}\equiv\\ \equiv
\ell \check c(X_2\otimes X_1)\xi_0(v )\frac{dt_1dt_2}{(t_1-t_2)^2} +\xi_0 ([X_2,X_1]v ) \frac{dt_1dt_2}{(t_1-t_2)t_1} 
+\xi_0((X_2X_1+X_1X_2)v )\frac{dt_1 dt_2}{t_1t_2} +\\
+\pi_1^*\eta_1 ([X_2,X_1]\otimes v ) \frac{dt_2}{t_1-t_2}
+\pi_2^*\eta_1 (X_2\otimes X_1v )\frac{dt_1}{t_1} +\pi_1^*\eta_1 (X_1\otimes X_2v )\frac{dt_2}{t_2}.
\end{multline*}
We may define $\eta_2(X_2\otimes X_1\otimes v )\in F_1\omega_2$ as the difference between 
$\xi_2(X_2\otimes X_1\otimes v )$ and the last expression and continue in this manner.
This makes it  clear that to give the sequence  $(\xi_i)_{i=0}^N$ is equivalent to giving a sequence 
\[
(\eta_i\in \sym_i(\gfrak^*\otimes F^1\omega))_{i=0}^N,
\]
with  $\eta_0=\xi_0$. Since the topological dual of $F^1\omega$ is $K/\Ocal$, we may also think of these as elements of 
$k[\gfrak K/\gfrak\Ocal]$. But beware that the $\eta_i$ depend on our choice of the uniformizer $t$.
\end{example}

\begin{example}\label{ex:}
Consider the Example \ref{ex:slin2}, so when $\gfrak=\slin (2)$ with its standard basis
$X_+=\left( \begin{smallmatrix}0 & 1\\ 0 & 0\end{smallmatrix}\right)$,
$X_-=\left(\begin{smallmatrix}0 &0\\1 &0\end{smallmatrix}\right)$, and
$H=\left(\begin{smallmatrix}1 &0\\0 &-1\end{smallmatrix}\right)$.
If we take $\lambda=(\lambda_1)=0$, then along the diagonal $\Delta$, 
$\xi_2(X_\pm\otimes X_\mp)$ and $\xi_2(H\otimes H)$ have a pole of order 
$\le 2$, $\xi_2(X_\pm\otimes H)$,  $\xi_2(H\otimes X_\pm)$  and $\xi_2(H\otimes Y)$ 
have there  a pole of order $\le 1$, whereas $\alpha (X_\pm\otimes X_\pm)$ and $\alpha (H\otimes H)$ are there regular.
This corresponds to the decomposition of $\gfrak\otimes\gfrak$ as a  $\gfrak$-representation 
(via the adjoint representation) into irreducibles: 
$\gfrak\otimes\gfrak\cong \sym^2_o\gfrak\oplus \wedge^2\gfrak \oplus k c$.
Then $\xi_2$ maps  $\sym^2_o\gfrak$ to $(F^1\omega)^{(2)}$, $\wedge^2\gfrak$ to
 $(F^0\omega)^{(2)}(\Delta)$ and $c$ to $(F^1\omega)^{(2)}(2\Delta)$.
 \end{example}

The functor $\hat\eta$ that we introduced in Corollary \ref{cor:etahat} (with $s$ being given by $\ell\check{c}$) 
can be regarded as an element of $\tilde\Vcal_\qlog^{\ell}(K)$, where $\Vcal_\qlog^{\ell}$ stands for the trivial 
representation $k$ of $\glie$ on which $c$ acts as multiplication by $\ell$.  We deduce from Corollary \ref{cor:etadef}:

\begin{corollary}\label{cor:} 
The construction of Corollary \ref{cor:etadef} produces an isomorphism
\[
\hom_{\Rscr_*}(\Lcal_{\glie, V_\lambda}, \Lcal og_K)\xrightarrow{\cong} \tilde\Vcal_\qlog^{\hat\lambda}(K)
\]
\end{corollary}

We denote the left hand side by $\tilde\Vcal_\llog^{\hat\lambda}(K)$. 
Note that this only depends on $V_\lambda$ and not on $\ell$
and so the same is true for the topological vector space that underlies $\tilde\Vcal^{\hat\lambda}(K)$.

Let $\xi\in \tilde\Vcal_\qlog^{\hat\lambda}(K)$. We extend $\xi_N$ to a $k$-linear map:
\[
\xi_N: (\glie K)^{\otimes_k N}\otimes V_\lambda\to \omega^{(N)}(\infty\Delta)
\]
by 
$\xi_N (X_Nf_N\otimes\cdots \otimes X_1f_1\otimes v):=
\pi_1^*f_1\cdots \pi_N^*f_N\xi_N(X_N\otimes\cdots \otimes X_1\otimes v)$.
The following proposition has its origin in a theorem of  Beilinson-Drinfeld \cite{bd}. 
Since we assumed $\Ocal$ to be complete, the 
residue pairing $(\alpha, f)\in \omega\times K\mapsto \res (f\alpha)\in k$ is topologically perfect. 
It identifies  $F^1\omega$ with the
$k$-dual of $K/\Ocal$ and hence $(F^1\omega)^{(N)}$, 
which we recall, was defined to be the completion of $(F_1\omega)^{\otimes_k N}$, 
with the topological $k$-dual of $(K/\Ocal)^{\otimes_kN}$. It of course does so $\Scal_N$-equivariantly.

\begin{proposition}\label{prop:localduality1}
The pairing $\tilde\Vcal_\qlog^{\hat\lambda}(K)\times (\oplus_{N\ge 0}(\glie K)^{\otimes_k N}\otimes V_\lambda)\to 
k$ given by 
\[
\big\la \xi\vert X_Nf_N\otimes\cdots \otimes X_1f_1\otimes v\big\ra := 
r_N\cdots r_1(\pi_1^*f_1\cdots \pi_N^*f_N\xi_N(X_N\otimes\cdots \otimes X_1\otimes v)),
\]
drops to a pairing 
\[
 \tilde\Vcal_\qlog^{\hat\lambda}(K)\times \tilde\Vcal_{\hat\lambda} (K)\xrightarrow{\la\;|\;\ra }  k
\] 
which identifies $\tilde\Vcal_\qlog^{\hat\lambda}(K)$ with the topological dual of the representation 
$\tilde\Vcal_{\hat\lambda} (K)$.
\end{proposition}
\begin{proof}
We first verify that the pairing is well-defined. Let $\xi\in \tilde\Vcal_\qlog^{\hat\lambda}(K)$.
The PBW theorem and the  symmetry properties of  $\xi$ tell us that it suffices to check that for 
$v\in V$, $X, Y\in\gfrak$,  and $W\in  (\glie K)^{\otimes_k N}$, then 
\[
r_{N+1}\cdots r_{1}\xi_{N+1}(W\otimes Xf\otimes v) =f(o)r_1r_{2}\cdots r_{N+1}(W\otimes Xv), \text{ when $f\in \Ocal$},
\]
and 
\begin{multline*}
r_{N+2}\cdots r_{1}\xi_{N+2}((Xf\otimes Yg-Yg\otimes Xf)\otimes W\otimes v)=\\
=r_{N+1}\cdots r_{1}\xi_{N+1}([X,Y]fg\otimes W \otimes v)
+\ell\check{c}(X,Y)\res (gdf).r_{N}\cdots r_1\xi_N(W\otimes v)
\end{multline*}
when $f,g\in K$. The first assertion is immediate from the ordinary residue theorem. 

As to the second identity, we note that  since $\xi_{N+2}\in \omega^{(N+2)}\la\qlog\ra_o$, 
we can write it as a sum $\xi'_{N+2}+\xi''_{N+1}$, such that  $\xi''_{N+1}$ takes its values in  
$(t_{N+2}-t_{N+1})^{-2}dt_{N+2}dt_{N+1}\omega^{(N)}\la\qlog\ra$ and $\xi'_{N+2}$ takes its values 
in the quasi-logarithmic polydifferentials which have along $\Delta_{N+1, N+2}$ a pole of order one at most.
So $\alpha:=r_{N}\cdots r_{1}\xi_{N+2}((X\otimes Y-Y\otimes X)\otimes W\otimes v)$ 
can be written accordingly as $\alpha'+\alpha''$ with $\alpha''$ a constant times $(t_{2}-t_{1})^{-2}dt_{2}dt_{1}$
and with $\alpha'$ having a pole along the diagonal of order one at most. 
In other words,  $\alpha$ satisfies the assumptions of Lemma \ref{lemma:resformula}.  We then find
\begin{multline*}
r_{N+2}\cdots r_{1}\xi_{N+2}((Xf\otimes Yg-Yg\otimes Xf)\otimes W\otimes v)
=r_2r_1 \big (\pi_2^*f \pi_1^*g.\alpha-\sigma^*\pi_2^*f \pi_1^*g.\alpha\big)=\\
=[r_2,r_1] \pi_2^*f \pi_1^*g.\alpha=r_{[21]}\pi_{[21]}^*(fg)r_{2,1} \alpha)+
\res (gdf) s_{2,1}\alpha, 
\end{multline*}
where in the first line $\sigma$ denotes the transposition of the factors indexed by $1$ and $2$ and where we used 
the permutation equivariance of $\xi_{N+2}$.  But the last expression of that display is by the defining properties 
(ii) and (iii) of $\tilde\Vcal_\qlog^{\hat\lambda}(K)$  equal to 
$\xi_{N+1}([X,Y]fg\otimes W\otimes v)+\ell \check{c}(X,Y)\res gdf . \xi_N(W\otimes v)$.

Next we show that the pairing is topologically perfect. Recall from Lemma \ref{lemma:PBW} that we have an 
increasing exhaustive filtration 
$\PBW_\pt\tilde\Vcal_{\hat\lambda} (K)$ on $\tilde\Vcal_{\hat\lambda} (K)$ for which 
\[
\PBW_N\tilde\Vcal_{\hat\lambda} (K)/\PBW_{N-1}\tilde\Vcal_{\hat\lambda} (K)\cong 
\sym_N{(\glie\otimes (K/\Ocal))}\otimes V_\lambda.
\]
Let  $W_{-N}\tilde\Vcal_\qlog^{\hat\lambda}(K)$ be the space of  $\xi$ with $\xi_{N-1}=0$, 
or equivalently, with $\xi_N$ taking its values in $(F^1\omega)^{(N)}$. Since $\xi_N$ determines 
its predecessors as residues, we will then have $\xi_r=0$ for all $r<N$, so that this defines an increasing filtration 
$W_\pt\tilde\Vcal_\qlog^{\hat\lambda}(K)$ on $\tilde\Vcal_\qlog^{\hat\lambda}(K)$.  
If $\xi\in W_{-N}\tilde\Vcal_\qlog^{\hat\lambda}(K)$, then the fact that $\xi_N$ takes its values in polydifferentials 
without diagonal residues implies that $\xi_N$  factors through $\sym ^N\glie\otimes V_\lambda$ and takes its 
values in the $\Scal_N$-invariant part 
of $(F^1\omega)^{(N)}$.

It remains to see that every symmetric $k$-linear map $\glie^{\otimes N}\otimes V_\lambda\to (F^1\omega)^{(N)}$ 
with this property so arises. 
Or rather, that if $\xi_0, \dots , \xi_N$ are such that they obey the defining properties of a 
$\Rscr^s_*$-homomorphism $\Lcal^{\ell\check{c}}_{\glie, V_\lambda}\to \qLog_K$ insofar 
they involve these elements, we can find
a $\xi_{N+1}$ such that the same is true for $\xi_0, \dots , \xi_{N+1}$. This amounts to saying that for any 
$X_{N+1}\otimes \cdots \otimes X_1\otimes v\in \glie^{\otimes(N+1}\otimes V$ 
only the polar part of $\xi_{N+1}(X_{N+1}\otimes \cdots \otimes X_1\otimes v)$ must be specified. 
There is then no issue if we let each $X_i$  run over a basis of $\glie$ and $v$ over a basis of $V$; 
we only must make sure that the resulting lift is $\Scal_{N+1}$-invariant (which we can always do by 
averaging over $\Scal_{N+1}$).
As we noted, the residue pairing identifies $(F^1\omega)^{(N)}$ with the topological $k$-dual of $K/\Ocal$
in a $\Scal_N$-equivariant manner. This establishes an isomorphism 
\[
W_{-N}\tilde\Vcal_{\hat\lambda} (K)/W_{-N-1}\tilde\Vcal_\qlog^{\hat\lambda} (K)\cong 
\Hom^{\ct}_k(\sym_N(\glie\otimes_k K/\Ocal)\otimes V_\lambda,k).
\]
This proves that the residue pairing induces a topologically perfect pairing between the graded pieces
between $\PBW_N\tilde\Vcal_{\hat\lambda} (K)/\PBW_{N-1}\tilde\Vcal_{\hat\lambda} (K)$ and 
$W_{-N}/W_{-N-1}$. So  the pairing between  
$\tilde\Vcal_\qlog^{\hat\lambda}(K)$ and $\tilde\Vcal_{\hat\lambda} (K)$ is topologically perfect as well 
(with $W_{-N-1}$ and $\PBW_{N}$ being each others annihilator).
\end{proof}

Before we discuss the  topological dual of the irreducible representation 
$\Vcal_{\hat\lambda} (K)$, let us recall that there is also a notion of a \emph{level}
for an irreducible $\glie$-module: 

\begin{definition}\label{def:highestweightsub}
The \emph{level} of a dominant integral weight $\lambda$ is $\lambda(H_\theta)$ and the  
\emph{level} of a finite dimensional  irreducible $\glie$-module
is the level of its highest weight. 

Given a nonnegative integer $\ell$ and  a finite dimensional $\glie$-module $V$, then the  \emph{level  $\ell$ subspace} 
$V^{\ell}$ of $V$ is the sum of the irreducible subrepresentations  of $V$ with a highest weight of level $\ell$; 
the notations $V^{>\ell}, V^{\le \ell}, \dots$ have a similar meaning.
\end{definition}

So  $V^{>\ell}$ is the $\glie$-submodule of $V$ generated by
$X_\theta^{\ell +1}V$ so that $V^{\le \ell}$ can be regarded as the largest $\glie$-submodule 
(or quotient $\glie$-module) of $V$ killed by $X_\theta^{\ell +1}$.

Note that the $\slin_2$-copy $\slin_2(\theta)\subset\glie $ spanned  by
$X_\theta,H_{\theta}, X_{-\theta}$ has the property that the $\slin_2(\theta)$-submodule generated in 
$V_\lambda$ by $1_\lambda$ is 
irreducible of dimension $\lambda(H_{\theta})+1$ and that this is supplemented in $V_\lambda$  by 
irreducible $\slin_2$-submodules of smaller dimension.
So $X_{\pm \theta}^{\ell(\lambda)}V_\lambda$ is of dimension one, but 
$X_{\pm \theta}^{\ell(\lambda)+1}V_\lambda=0$. The level of an  finite dimensional  irreducible 
$\glie$-module is evidently canonically defined  (i.e., is independent of the choice of a 
Borel subalgebra and a Cartan subalgebra contained in it); indeed the $X_\theta$   
that so arise make up a single orbit  under the action of the adjoint group of $\glie$. (In fact, we can define
this  more directly as the maximal $\ell$ for which there exists a Lie embedding $\slin (2)\hookrightarrow \glie$ and
an irreducible  $\slin (2)$-subrepresentation of $V$ equivalent to the $\ell$th symmetric power of the 
tautological representation.)
Since $\glie$ has highest weight $\theta$, its level is $\theta(H_\theta)=2$.

\begin{corollary}\label{cor:topduallocal}
The  topological dual of $\Vcal_{\hat\lambda} (K)$ is the subspace 
$\Vcal_\qlog^{\hat\lambda} (K)$ of  $\tilde\Vcal_\qlog^{\hat\lambda} (K)$ consisting of the $\xi$ for which
$\glie^{\otimes N}\otimes X_\theta^{(\lambda_0+1)}\otimes  1_\lambda$ (or equivalently, but in more invariant terms, 
$\glie^{\otimes N}\otimes (\glie^{(\lambda_0+1)}\otimes V_\lambda)^+$)  vanishes in the generic point of the locus 
$t_1=\cdots =t_{\lambda_0+1}=0$.

The construction of Corollary \ref{cor:etadef} identifies this with the subspace 
$\Vcal_\llog^{\hat\lambda} (K)\subset \tilde\Vcal_\llog^{\hat\lambda} (K)$  defined by the same property.
\end{corollary}
\begin{proof}
The  topological dual of $\Vcal_{\hat\lambda} (K)$  consists of  the $\xi\in \tilde\Vcal_\qlog^{\hat\lambda} (K)$ for 
which for all  $N\ge 0$, $X_i\in \glie$ and $f_i\in K$  we have
$\la \xi | X_Nf_N\circ \cdots \circ X_1f_1\circ (X_\theta t^{-1})^{\circ (\lambda_0+1)}\circ 1_\lambda\ra=0$. 
This means that for every 
$Z\in \glie^{\otimes N}$, and $f_1, \dots , f_N\in K$, 
\[
r_{N+\lambda_0+1} \cdots r_1 \big(f_N(t_{N+\lambda_0+1})\cdots 
f_1(t_{\lambda_0+1})t_{\lambda_0+1}^{-1}\cdots t_1^{-1} 
\xi(Z\otimes X_\theta^{\otimes(\lambda_0+1)}\otimes v)\big)
\]
vanishes.
Now $\xi(Z\otimes X_\theta^{\otimes (\lambda_0+1)}\otimes v)$ is 
regular in the generic point of  $t_1=\cdots =t_{\lambda_0+1}=0$. It follows that 
$r_{\lambda_0+1}\cdots r_1\big(t_{\lambda_0+1}^{-1}\cdots t_1^{-1} 
\xi(Z\otimes X_\theta^{\otimes (\lambda_0+1)}\otimes v)\big)$ 
is the contraction of  $\xi(Z\otimes X_\theta^{\otimes (\lambda_0+1)}\otimes v)$ 
with the tensor $\p /\p t_1\otimes\cdots \otimes  \p /\p t_{\lambda_0+1}$ followed by restriction to 
 $t_1=\cdots =t_{\lambda_0+1}=0$. The vanishing of latter  is equivalent to the  vanishing of 
 $\la \xi | X_Nf_N\circ \cdots X_1f_1\circ (X_\theta t^{-1})^{\circ (\lambda_0+1)}\circ 1_\lambda\ra$ for all 
 $(f_1, \dots, f_N)\in K^N$.
This characterization  must of course be invariant under the automorphism group of $\glie$. Since 
$\glie^{(\lambda_0+1)}\otimes V_\lambda$, when considered as a  $\glie$-module, has  highest weight  
$(\lambda_0+1)\theta +\lambda$ with highest weight space  spanned by 
$(X_\theta)^{(\lambda_0+1)}\otimes 1_\lambda$, this amounts to the corresponding vanishing property of 
$\xi|\oplus_N \glie^{\otimes N}\otimes (\glie^{(\lambda_0+1)}\otimes V_\lambda)^+$.

The last assertion is clear.
\end{proof}

\begin{remark}\label{rem:}
Since any $\xi$ in $\Vcal_\qlog^{\hat\lambda} (K)$ or $\Vcal_\llog^{\hat\lambda} (K)$ is permutation invariant, 
it follows that if we substitute $1_\lambda$ in the $V_\lambda$-slot  and $X_\theta$ in  $\lambda_0+1$ of the 
$N$ $\glie$-slots of  $\xi_N$, then $\xi_N$ vanishes on the locus defined by putting the corresponding 
$\lambda_0+1$ coordinates equal to zero.
\end{remark}

\subsection{The case of a multirepresentation}\label{subsect:multo}
Let $C$ be a nonsingular irreducible  projective curve of genus $g$ over $k$ and  
$P\subset C$ a \emph{nonempty} finite subset of closed points. For every $p\in P$, we write $\Ocal_p$ for the 
completed local ring $\varinjlim_{r}\Ocal_{C,p}/\mfrak_{C,p}^r$.  We put   $\Ocal_P:=\prod_{p\in P} \Ocal_p$ 
and $K_P:=\prod_{p\in P} K_p$, where $K_p$ is the field of fractions of $\Ocal_p$. 
This brings us in the setting of Proposition \ref{prop:F_KP}.  If we write $\hat C_P$ for the formal completion of $C$ at 
$P$, then it makes sense to think of an element of $\omega _{K_P}^{(I)}\la\qlog\ra$ as a section of 
$\Omega_C(P)^{(I)}\la\qlog\ra$ over $\hat C_P$ so that we 
sometimes write   $\omega_{\hat C_P}(P)^{(I)}\la\qlog\ra$ for $\omega _{K_P}^{(I)}\la\qlog\ra$.

A central extension $\widehat{\glie K_P}$ of  $\glie K_P:=\glie\otimes K_P=\prod_{p\in P} \glie K_p$  by 
$\cfrak$  is defined by 
letting this  on each summand $\glie K_p$ to the central extension $\widehat{\glie K_p}$ introduced before. 
Suppose we are now also given a nonnegative integer $\ell$ and for each $p\in P$ an integral dominant weight 
$\lambda(p)\in\hlie^*$ of level $\ell$.  We put 
$V_{\lambda}:=\otimes_{p\in P} V_{\lambda (p)}$ and  regard this first as a representation of  
$\glie^P\times \cfrak$, where $c\in \cfrak$ acts as multiplication by $\ell$.
We subsequently  regard this as one of the Lie subalgebra $\widehat{\glie \Ocal_P}:=
\glie\Ocal_P\times \cfrak$  of $\widehat{\glie K_P}$ via its reduction modulo $\mfrak_P$ and then denote it 
by $V_{\hat\lambda}$. Let $\tilde\Vcal_{\hat\lambda}(K_P)$ be obtained by inducing it up to $\widehat{\glie K_P}$.  
This is easily seen to be naturally isomorphic to $\otimes_{p\in P} \tilde\Vcal_{\hat\lambda(p)}(K_p)$.
Its  irreducible quotient $\Vcal_{\hat\lambda}(K_P)$ is therefore identified with 
$\otimes_{p\in P}\Vcal_{\hat\lambda (p)}(K_p)$. In other words, we need to divide out 
$\tilde\Vcal_{\hat\lambda}(K_P)$ by the submodule generated by 
$\sum_{p\in P} (X_\theta \mfrak_p^{-1})^{\otimes (\lambda_0(p)+1)}\otimes V_{\hat\lambda}^{(p)+}$.
A minor modification of the proof of Lemma \ref{lemma:finiteAgen} yields:

\begin{lemma}\label{lemma:finiteAmultgen}
Let $A\subset K_P$ be a $k$-subalgebra such that $K_P/(A+\Ocal_P)$ is of finite $k$-dimension. 
Then $\Vcal_{\hat\lambda} (K_P)$ is a finitely generated $\U\!\widehat{\glie A}$-module. $\square$
\end{lemma}

We extend our notation for the topological duals to the present situation, denoting these by   
$\tilde\Vcal^{\hat\lambda}(K_P)$ and $\Vcal^{\hat\lambda}(K_P)$.
 Proposition \ref{prop:localduality1} also generalizes in an evident manner: we put
\[
\tilde\Vcal^{\hat\lambda}_\qlog(K_P):=\hom_{\Rscr^s_P} (\Lcal^{\ell\check{c}}_{\glie,P,{V_\lambda}}, \qLog_{K_P}).
\]
but think of an element $\xi$ the right hand side as one which assigns to
every finite set $I$ a linear map: $\xi_I : \glie^{\otimes I}\otimes V_\lambda\to \omega_{\hat C_P}(P)^{(I)}\la\qlog\ra$ 
subject to the familiar conditions.  

Given $p\in P$, then we write  $V_\lambda^{(p)}$  when  $V_\lambda$ is considered as a 
$\glie$-module acting through the $p$-th tensor factor  only and denote this action by 
$(X, v)\in \glie\times V\mapsto X^{(p)}v$.  So the kernel of the action of $X_\theta^{(p)}$ on $V$ is the highest weight 
space  $V_\lambda^{(p)+}$ of $V^{(p)}$.
For a restriction as in  Corollary \ref{cor:topduallocal}, we  also need to specify $p$, but are allowed to 
replace $1_\lambda\in V_\lambda$ by any $v\in V^{(p)+}$. 
We then find in the same way:

\begin{proposition}\label{prop:localduality2}
A well-defined pairing  
$\tilde\Vcal_\qlog^{\hat\lambda}(K_P)\times \tilde\Vcal_{\hat\lambda} (K_P)\xrightarrow{\la\;|\;\ra }  k$ is defined by
\[
\textstyle \big\la \xi \vert X_Nf_N\otimes\cdots \otimes X_1f_1\otimes v\big\ra := r_N r_{N-1}\cdots r_1 
\big(\pi_1^*f_1\cdots \pi_N^*f_N \xi_N(X_N\otimes \cdots \otimes X_1\otimes v)\big),
\]
where $r_i$ stands for the sum of the residues (at the points of $P^I$) along the $i$-th factor. 
This pairing is nondegenerate in the sense that this
identifies $\tilde\Vcal_\qlog^{\hat\lambda}(K_P)$ with the topological dual of the representation 
$\tilde\Vcal_{\hat\lambda} (K_P)$.

Via this duality the topological dual of $\Vcal_{\hat\lambda}(K_P)$ corresponds 
with  the subspace   $\Vcal^{\hat\lambda}_\qlog (K_P)$ resp.\  $\Vcal^{\hat\lambda}_\llog (K_P)$ 
with the property that for every $p\in P$ and every integer $N\ge 0$ and every 
$Z$ in the tensor algebra of $\glie$ and $v\in V_\lambda^{(p)+}$,  
$\xi(Z\otimes X_\theta^{(\lambda_0(p)+1)}\otimes  v)$  vanishes  in the generic point of the locus $t_1=\cdots =t_{\lambda_0(p)+1}=p$.
\end{proposition}

\begin{remark}\label{rem:}
The construction of Corollary \ref{cor:etadef} identifies $\tilde\Vcal_\qlog^{\hat\lambda}(K_P)$ with 
\[
\tilde\Vcal_\llog^{\hat\lambda}(K_P):= \hom_{\Rscr_P}(\Lcal_{\glie, P, V_\lambda}, \Lcal og_K)
\]
 and $\Vcal_\qlog^{\hat\lambda}(K_P)$ with the corresponding subspace   $\Vcal_\llog^{\hat\lambda}(K_P)$.
\end{remark}

\subsection{The propagation principle}\label{subsect:propagation}
Let now $P'$ be a  finite subset of $C$ which contains $P$ and   regard $\widehat{\glie_P}$ as a 
subalgebra of $\widehat{\glie_{P'}}$.
If we assign to each element of $P'\ssm P$ the trivial representation with a specified generator, 
then we have also a corresponding inclusion $\tilde\Vcal_{\hat\lambda}(K_P)\subset\tilde\Vcal_{\hat\lambda'}(K_{P'})$. 
This induces a surjection of their topological duals, which in their incarnation of Proposition
\ref{prop:localduality2} above, amounts to restricting $\xi$ to $P$.

Since we assumed $P$ nonempty, $C\ssm P$ is affine. We then regard its algebra  $k[C\ssm P]$ of regular functions as 
being  embedded in $K_P$.  It is a classical fact that the polar part of a meromorphic differential on an irreducible 
smooth projective curve can be prescribed arbitrarily, provided that the sum of the residues is zero. 
This means that the annihilator of $k[C\ssm P]$ in $\omega_P$ is the $k[C\ssm P]$-module of meromorphic 
differentials on $C$ that are regular  on $C\ssm P$. We shall need the following more precise version, which is a 
classical formulation of Serre duality for curves:

\begin{lemma}\label{lemma:serreduality}
Let $D=\sum_{z\in P} n_z (z)$ be a divisor on $C$.  Then for every nonempty subset $P$ of $C$ containing the 
support of $D$, the cover of $C$ by $C\ssm P$ and the formal neighborhood of $P$ defines an isomorphism
of $A_P(D):=k[C\ssm P]\bs \big(\prod_{p\in P} K_p/\mfrak^{n_p}\big)$ with $H^1(C, \Ocal(-D))$. 
Via this isomorphism, the Serre duality pairing is given  by  
the residue pairing
\[
\textstyle H^0(C, \Omega_C(D))\times A(D) \to k, 
\quad \big(\alpha, (f_p\in K_p/\mfrak^{n_p})_{p\in P}\big)\mapsto \sum_{p\in P} \res_{p\in P} f_p\alpha.
\]
In  particular, the latter is topologically perfect. $\square$
\end{lemma}

The  residue theorem implies  that $\glie [C\ssm P]:=\glie\otimes k[C\ssm P]$ is a 
Lie subalgebra of $\widehat{\glie K_P}$ (the central extension being trivial over it). 
The image of $k[C\ssm P]$ in $K_P$ satisfies the hypothesis of Lemma \ref{lemma:finiteAmultgen} and so 
$\Vcal_{\hat\lambda}(K_P)$ is a finitely generated 
$\glie [C\ssm P]$-module. In particular, the spaces of  $\glie [C\ssm P]$-covariants (also called the 
\emph{covacuum space} defined  by these data)
\[
\Vcal_{\hat\lambda}(C,P):=\Vcal_{\hat\lambda}(K_P)_{\glie [C\ssm P]}.
\]
is of  finite $k$-dimension. Its (continuous) dual
\[
\Vcal^{\hat\lambda}(C,P):=\Hom^{\ct}_{\glie [C\ssm P]}(\Vcal_{\hat\lambda}(K_P),k),
\]
 is  called  the \emph{vacuum} space (here $k$ is regarded as the trivial $\glie [C\ssm P]$-module). 
 We shall regard this as a subspace of the space $\Vcal^{\hat\lambda}(K_P)$ defined
in Proposition \ref{prop:localduality2}. Since $\glie[C\ssm P]$ contains $\glie$,  the corresponding 
$\xi_N$ that appear there must be invariant under the diagonal $\glie$-action.

If we assign to each point of $Q\ssm P$ a trivial representation with specified generator, then 
we obtain an embedding of $\Vcal_{\hat\lambda}(K_P)$ in $\Vcal_{\hat\lambda'}(K_{Q})$. 
This is equivariant with respect to the action of $\glie[C\ssm P]$ via the inclusion  
$\glie[C\ssm P]\subset \glie[C\ssm Q]$ so that we get a natural map of covariants
\[
\Vcal_{\hat\lambda}(K_P)_{\glie [K\ssm P]}\to \Vcal_{\hat\lambda'}(K_Q)_{\glie [K\ssm Q]}.
\]
The propagation principle asserts that this is in fact an isomorphism.
This gives the formation of these co-invariants an adelic flavor, as we can now take a direct limit over all finite subsets of $C$. 

The proper formulation in this spirit relies on a  somewhat stronger property that was exhibited by 
Beauville \cite{beauville}: we can also take $Q\subset P $  as long it is nonempty while retaining $V_{\hat\lambda}$.  
To be precise, we write $V_{\hat\lambda}=V_{\hat\lambda|Q} \otimes V_{\hat\lambda|P\ssm Q}$. Then 
$\Vcal_{\hat\lambda |Q}(K_Q)\otimes V_{\hat\lambda(p)|P\ssm Q}$ is a tensor product of 
$\glie [C\ssm Q]$-modules, where we let this Lie algebra act on the tensor factors
$ V_{\hat\lambda(p)}$ (with $p\in P\ssm Q$) by evaluation. The assertion is that for this action, the natural map
\[
\big(\Vcal_{\hat\lambda |Q}(K_Q)\otimes V_{\hat\lambda(p)|P\ssm Q} \big)_{\glie [C\ssm Q]}\to 
\Vcal_{\hat\lambda}(K_P)_{\glie [C\ssm P]}
\]
is also an isomorphism of finite dimensional $k$-vector spaces. 

The adelic formulation can then be stated as follows. Let $\Ocal_{\hat C}:=\prod_{p\in C} \Ocal_p$ 
be the `formal atomization' of $C$ and let 
$K_{\hat C}$ be the restricted product $\prod'_{p\in C} K_p$ (so $(f_p\in K_p)_{p\in C}$ lies in 
$K_{\hat C}$ if and only if $f_p\in \Ocal_p$ for  all but finitely many $p\in C$). 
This contains both $\Ocal_{\hat C}$ and the algebra of rational functions $k(C)$ as a 
subalgebra and is in fact generated by these two.
If we define 
$\widehat{\glie K_{\hat C}}$ in the usual manner as an extension of $\glie K_{\hat C}$ by $\cfrak$, then it 
contains $\widehat{\glie\Ocal_{\hat C}}$ and $\glie k(C)$ as Lie subalgebras.
We now think of  $\hat\lambda$ as a map which assigns to every  $p\in C$ a 
dominant weight $\hat\lambda(p)$ of level $\ell$ with the property that $\hat\lambda(p)=(\ell, 0,\dots, 0)$ 
for all but finitely many $p$. Let $V_{\hat\lambda}$ stand for the finite dimensional 
$\widehat{\glie\Ocal_{\hat C}}$-module $\otimes_{p\in C} V_{\lambda(p)}$ on which 
$c\in\cfrak$ acts as multiplication with $\ell$ and each factor $\glie\Ocal_p$ acts via its reduction to $\glie$ on the 
$p$th factor, so as $(Xf, v)\in \glie\Ocal_p\times V\mapsto f(p)X^{(p)}v$ (recall that for all but finitely many $p\in C$, 
$V_{\hat\lambda}^{(p)}$ is the trivial representation of $\glie$). 
We let  $\tilde\Vcal_{\hat\lambda}(\hat C)$ be the $\widehat{\glie K_{\hat C}}$-module obtained by induction via the 
inclusion $\widehat{\glie\Ocal_{\hat C}}\subset \widehat{\glie K_{\hat C}}$ and denote by 
$\Vcal_{\hat\lambda}(\hat C)$ its irreducible quotient. We identify $\ell$ with
$(\ell, 0,\dots, 0)$, so that the case in which all the representations are trivial these 
$\widehat{\glie\Ocal_{\hat C}}$-modules can be denoted  $\tilde\Vcal_\ell(\hat C)$  resp.\ $\Vcal_\ell(\hat C)$.
Then  the strong version of the propagation principle amounts to the statement that the space of its covariants with respect to 
$\glie k(C)$,  $\Vcal_{\hat\lambda}(\hat C)_{\glie k(C)}$, where $k(C)$ is the function field of $C$,
is naturally identified with $\Vcal_{\hat\lambda}(C,P)$ as defined above.

\subsection{Polydifferential incarnation of conformal blocks}\label{subsect:global}
Recall that if  $Q$ is a set and $V$ is a finite dimensional representation of  $\glie^Q$, 
then $V^{(q)}$ ($q\in Q$) stands for the $\glie$-representation defined by the $q$-th component of $\glie$,
and  in case $V^{(q)}$ has a single highest weight,  $V^{(q)+}\subset V^{(q)}$  denotes  
the subrepresentation generated by the highest weight space of $V^{(q)}$.
Here is our first main result regarding conformal blocks.

\begin{theorem}\label{thm:globaldual1}
Let $\hat\lambda$ assign to every $p\in C$ a highest weight 
$\hat\lambda (p)=(\lambda_0(p),\dots , \lambda_r(p))\in \ZZ_{\ge 0}^{r+1}$ of level $\ell$ (so with 
$\sum_i \check{n_i}\lambda_i(p)=\ell$)  such that  $\lambda(p):=(\lambda_1(p),\dots , \lambda_r(p))$ is  
zero for all but finitely many $p\in C$. Let $P$ be the support of $\lambda$  and put 
$V_\lambda:=\otimes_{p\in C} V_{\lambda(p)}=\otimes_{p\in P} V_{\lambda(p)}$ 
(a finite dimensional) representation of $\glie^C$) and denote by 
$1_\lambda:=\otimes_{p\in C} 1_{\lambda(p)}$ its highest weight vector. 
Then the vacuum space $\Vcal^{\hat\lambda}(C)$ is identified  with the subspace 
$\Vcal_\qlog^{\hat\lambda}(C)$ of the space of $\Rscr^s_P$-module homomorphisms 
$\Lcal_{\glie,P, V}^s\to \qLog_{C,P}$, i.e., with graded maps 
\[
\textstyle \xi=( \xi_N)_{N=0}^\infty\in \prod_{N=0}^\infty\Hom_k\big(\glie^{\otimes N}\otimes V_{\lambda}, H^0(C^N, \Omega_C(P)^{(N)}\la\qlog\ra)\big), 
\]
with $\xi_N$ $\Scal_N$-equivariant such that   if $(X_i\in \glie)_{i\ge 1}$, $v\in V_{\lambda}$, $p\in C$, then
\begin{enumerate}
\item[(i)]
$r_{1}^p\big(\xi_{N+1}(X_{N+1}\otimes\dots \otimes X_1\otimes v)\big)=
\xi_{N}(X_{N+1}\otimes\dots \otimes X_2\otimes X^{(p)}_{1}v)$, where 
$ X^{(p)}$ stands for the $\glie$-action of $X$ on $V_{\lambda}$ indexed by $p$
(so both sides are zero if $p\notin P$),
\item[(ii)] $r_{N+2,N+1} (\xi_{N+2}(X_{N+2}\otimes X_{N+1}\otimes\dots \otimes X_1\otimes v))=
\xi_{N+1}([X_{N+2},X_{N+1}]\otimes\cdots\otimes X_1\otimes v)$,
\item[(iii)] $s_{N+2,N+1} (\xi_{N+2}(X_{N+2}\otimes X_{N+1}\otimes \dots \otimes X_1\otimes v))=
\ell\check{c}(X_{N+2}, X_{N+1})\xi_{N}(X_{N}\otimes\dots\otimes X_1\otimes v)$,
\end{enumerate}
that satisfy in addition
\begin{enumerate}
\item[(iv)] $\xi_N$ takes $\glie^{\otimes N}\otimes  X_\theta^{\otimes(\ell +1)}\otimes  1_\lambda$  to  
rational polydifferentials that vanish on the generic point of the codimension $\ell$ diagonal $z_1=\cdots =z_{\ell+1}$.
\end{enumerate}
The duality pairing between $\Vcal_{\hat\lambda}(C)$ and  $\Vcal_\qlog^{\hat\lambda}(C)$ is  given by
\[
\la \xi | X_Nf_N\otimes \cdots \otimes X_1f_1\otimes v\ra =
\textstyle r_N\cdots  r_1 \big( \pi_1^*f_1\cdots \pi_N^*f_N\\xi_N(X_N\otimes\cdots \otimes X_1\otimes v)\big),
\]
where $f_i\in K_{\hat C}$ and $r_i=\sum_{p\in C} r^p_i$ 
stands for the total residue along the $i$th component.

Such $\xi\in \Vcal_\qlog^{\hat\lambda}(C)$ are necessarily $\glie$-invariant in the sense that each $\xi_N$ 
factors through $(\glie^{\otimes N}\otimes V_{\lambda})_\glie$ (in other words,  $\xi$ factors through 
$(\Lcal_{\glie,P, V}^s)_\glie$). 

There exists an integer $N\ge 0$ such that $\xi\in\Vcal_\qlog^{\hat\lambda}(C)\mapsto \xi_N$  is injective.
\end{theorem}
\begin{proof}
Lemma \ref{lemma:serreduality} shows that  a $\xi$ as above  annihilates 
$\glie[C\ssm P] \Vcal_{\hat\lambda}(K_P)$ and hence defines an element of $\Vcal^{\hat\lambda}(C)$.

To prove the converse,  we first take $q\in C\ssm P$ and regard $\xi$ as a linear form on 
$\Vcal_\ell(K_q)\otimes V_{\hat\lambda}$
that is $\glie[C\ssm \{q\}]$-invariant.  We proceed with induction and assume that we have already established that for 
$i\le N$, 
$\xi_i$ is defined by a polydifferential on $C^i$ with the prescribed residue properties 
(the induction starts with $N=0$, for which there is no issue).

Let $X_1, \dots, X_{N+1}\in \glie$ and $v\in V_{\hat\lambda}$ and put
\begin{gather*}
\eta_q:=\xi_{N+1}(X_{N+1}\otimes \cdots \otimes X_1\otimes v)\in 
\omega^{(N+1)}_q\la \qlog \ra,\\
\eta^{(p)}:=\xi_N(X_{N}\otimes\dots \otimes X_1\otimes X^{(p)}_{N+1}v)\in 
H^0(C^N, \Omega^{(N)}_q\la \qlog \ra), \text{  where $p\in P$.}
\end{gather*}
These polydifferentials  determine the $k$-linear maps 
$\tilde \eta_q: K_q^{N}\to \omega_q$ resp.\ $\tilde \eta^{(p)}: K_q^{N}\to k$ by the rule
\begin{gather*}
\tilde\eta_q(f_1, \dots , f_{N}):=r^q_{N}\cdots r^q_1 \pi_{N}^*(f_{N})\cdots\pi_1^*(f_1)\eta_q\in \omega_q,\\
\tilde\eta^{(p)}(f_1, \dots , f_{N}):=r^q_{N}\cdots r^q_1 \pi_{N}^*(f_{N})\cdots\pi_1^*(f_1)\eta^{(p)}\in k.
\end{gather*}
We are given that  when $f\in k[C\ssm\{q\}]$, 
$\res_q (f\tilde\eta_q)+\sum_{p\in P} f(p)\tilde\eta^{(p)}=0$ as an identity of linear forms on 
$K_q^{N}$. By Lemma \ref{lemma:serreduality} this means
that $\tilde\eta_q$ takes its values in $H^0(C, \Omega_C(P))$  and is such that its residue  in $p\in P$ is $\tilde\eta^{(p)}$
(as  a linear form on $K_q^{N}$).
In terms of $\eta_q$, this is merely saying  that $\eta_q$ extends to a section of $\Omega_C(P)^{(N+1)}\la \qlog \ra$ over 
$C$ times the formal germ of $C^N$ at $q^N$ whose residue at  
$p\in P$  is $\eta^{(p)}$. 

In this argument we could have replaced $q$ by any nonempty finite subset $Q$ of $C\ssm P$: 
the propagation principle assures us that  $\xi_{N+1}$ can also be understood a continuous  linear form on 
$(\Vcal_\ell (K_{Q})\otimes V_{\hat\lambda})_{\glie [C\ssm Q]}$
so that via this identification $\eta$ is a formal polydifferential at $Q^{N+1}$, with the argument just given showing
that this extends to a formal section over  $C\times Q^N$. In particular, if 
$(q_1, \dots, q_{N})\in (C\ssm P)^{N}$ is arbitrary, then
$\eta$ is defined on a formal neighborhood of $(q_1, \dots, q_N)\times C$. 
This extension is then independent of $q_{N+1}$ (by Lemma \ref{lemma:serreduality}). 
We thus get a  section of $\Omega_C(P)^{(N+1)}\la \qlog\ra$ over the atomization of  $C^{N+1}$ 
that is regular in every variable.
But such a section is the same thing as an ordinary section. 
This proves that $\xi_N(X_N\otimes\cdots \otimes X_1\otimes v)\in H^0(C^N, \Omega_C(P)^{(N)}\la\qlog\ra)^{\Scal_N}$.

The $\glie[C\ssm P]$-invariance implies  $\glie$-invariance: each $\xi_{N}$ must factor through 
$(\glie^{\otimes N}\otimes V_{\hat\lambda})_\glie$. 
This  amounts to the property that $\sum_{p\in P} r^p_{N}+\sum_{i=1}^{N-1}r_{N, i}=0$, so  is 
actually an incarnation  of the residue theorem (as encountered in Proposition \ref{prop:residuechar}). 

The propagation principle and Corollary \ref{cor:topduallocal} require the vanishing of 
$\xi (Z\otimes X_\theta^{\otimes(\ell +1)}\otimes v)$
(with $Z\in \glie^{\otimes N}$) on the  diagonal locus $z_1=\cdots =z_{\ell  +1}$ for all 
$v\in V_{\hat\lambda}$. 
To see that this suffices, we show that this property  implies the vanishing of  
$\xi (Z\otimes X_\theta^{\otimes (\lambda_0(p) +1)}\otimes v')$ at  $p^{\lambda_0(p) +1}$ for all 
$p\in P$ and $v'\in V_{\hat\lambda}^{(p)+}$: indeed, 
the identity
\[
r_{\lambda(p)}^p\cdots r_{1}^p\xi (Z\otimes X_\theta^{\otimes (\ell +1)}\otimes v)=
\xi(Z\otimes X_\theta^{\otimes (\ell -\lambda(p)+1)}\otimes (X_\theta^{(p)})^{\lambda(p)}v)=
\xi(Z\otimes X_\theta^{\otimes (\lambda_0(p)+1)}\otimes (X_\theta^{(p)})^{\lambda(p)}v)
\]
shows that the right hand side vanishes at  $p^{\lambda_0(p)}$. 
Since the highest weight space 
$V^+_{\lambda(p)}$ of   $V_{\lambda(p)}$ is the image of its lowest weight space 
$V_{\lambda(p)}(-\lambda(p))$ under $X_\theta^{\lambda(p)}$,  any $v'\in V_{\hat\lambda}^{(p)+}$ 
is of the form $(X_\theta^{(p)})^{\lambda(p)}v$. 
So  $\xi_{\lambda_0(p)+1}(Z\otimes X_\theta^{\otimes (\lambda_0(p)+1)}\otimes v')$ vanishes at 
$z_1=\cdots =z_{\lambda_0(p)+1}=p$.
Apart from the assertions in the last paragraph, the rest  now follows from Proposition \ref{prop:localduality2} 
and  the propagation principle.

Since $\Vcal_{\hat\lambda}(C)$ if a finite dimensional $k$-vector space, there  exists an integer  $N\ge 0$ such that 
$\PBW_N(\Vcal_{\hat\lambda}(C))\to \Vcal_{\hat\lambda}(C)$ is onto. Then clearly $\xi$ must be determined by 
$\xi_N, \dots, \xi_0$. But $\xi_N$ determines  $\xi_0,\dots, \xi_{N-1}$, as these are obtained as (iterated) 
residues of $\xi_N$.
\end{proof}

We state the version for the logarithmic polydifferentials separately. The following theorem  is an immediate 
consequence of  the preceding Theorem \ref{thm:globaldual1} and Corollary  \ref{cor:zetadef}.

\begin{theorem}\label{thm:globaldual2}
The vacuum space $\Vcal^{\hat\lambda}(C)$ is identified  with the subspace $\Vcal_\llog^{\hat\lambda}(C)$ of the 
space of $\Rscr_P$-module homomorphisms $\Lcal_{\glie,P, V}\to \Lcal og_{C,P}$, i.e., with graded maps 
$\xi=( \xi_N)_{N=0}^\infty$ in
\[
\textstyle \prod_{N=0}^\infty\Hom_k\big(\glie^{\otimes N}\otimes V_{\lambda}, H^0(C^N, \Omega_C(P)^{(N)}\la\llog\ra)\big), 
\]
with $\xi_N$ $\Scal_N$-equivariant and possessing properties (i), (ii) of Theorem \ref{thm:globaldual1}, that in 
addition satisfy property (iv) of that theorem.
\end{theorem}

\begin{example}[The genus zero case]\label{example:}
We do the well known case when $C$ has genus zero (compare Prop.\ 4.1 in Beauville \cite{beauville}). 
Since $C$ has no nonzero regular differentials, a polydifferential on  a power of $C$ is completely given by its residues. 
This means that any $\xi \in \Vcal_\llog^{\hat\lambda}(C)$ is completely determined by its zeroth term 
$\xi_0\in \Hom(V_{\lambda}, k)^\glie$. We show how this is done.

If $P$ is the support of $\lambda$, then the $\glie$-invariance of  $\xi_0$ implies that  $\sum_{p\in P} \xi_0(X^{(p)}v)=0$.
We then find that if $z$ is an affine coordinate on $C$ whose domain contains $P$, then 
\[
\xi_1 (X\otimes v)=\sum_{p\in P} \frac{\xi_0(X^{(p)}v)dz}{z-p}.
\] 
The $\glie$-invariance of $\xi_0$ amounts to the property that the residue sum at $P$ is zero, so that 
(by the residue formula) this form is indeed regular at $z=\infty$. Hence, in  terms of the coordinate 
$w:=1/z$, $\xi_1$ is at $\infty$ given by
\begin{multline*}
\sum_{p\in P} \xi_0(X^{(p)}v)\frac{w^{-2}dw}{w^{-1}-p}=\sum_{p\in P} \xi_0(X^{(p)}v) \frac{dw}{w(1-pw)}=\\
=\sum_{p\in P} \xi_0(X^{(p)}v)(1+pw +p^2w^2+\cdots) \frac{dw}{w}  =\sum_{p\in P} \xi_0(X^{(p)}v) (p +p^2w+\cdots )dw.
\end{multline*}
In particular, its value in $\infty$ is  the covector $\xi_0(T(X)v)dw $, where  $T(X):=\sum_{p\in P} pX^{(p)}$.

In degree two we find that
\begin{multline*}
\xi_2 (X_2\otimes X_1\otimes v)=\sum_{p_1, p_2\in P} 
\frac{\xi_0(X_2^{(p_2)} X_1^{(p_1)}v)}{(z_1-p_1)(z_2-p_2)}dz_1 dz_2 - \sum_{p\in P} 
\frac{\xi_0([X_2,X_1]^{(p)}v)}{(z_1-z_2)(z_2-p)}dz_1 dz_2=
\\= \sum_{p_1, p_2\in P; p_1\not=p_2} \frac{\xi_0(X_2^{(p_2)} X_1^{(p_1)}v)}{(z_1-p_1)(z_2-p_2)}dz_1 dz_2
+ \sum_{p\in P}\Big(\frac{\xi_0(X_2^{(p)} X_1^{(p)} v)}{(z_2-z_1)(z_1-p)}  +
\frac{\xi_0(X_1^{(p)} X_2^{(p)} v)}{(z_1-z_2)(z_2-p)} \Big) dz_1 dz_2.
\end{multline*}
Since $X_2^{(p_2)}$ and $X_1^{(p_1)}$ commute for $p_1\not=p_2$, the last identity makes the symmetry property of 
$\xi_2$ manifest. One checks that indeed the residue of $\xi_2 (X_2\otimes X_1\otimes v)$ along $z_1-p$ resp.\ 
$z_2-p$  is $\xi_1(X_2\otimes X_1^{(p)}v)$ resp.\  $\xi_1(X_1\otimes X_2^{(p)}v)$
and that the residue taken when $z_1\to z_2$ is $-\xi_1( [X_2,X_1]\otimes v)$. Note that in particular, 
\[
\xi_2 (X\otimes X\otimes v)=\sum_{p_1, p_2\in P} \frac{\xi_0(X^{(p_2)} X^{(p_1)}v)}{(z_1-p_1)(z_2-p_2)}dz_1 dz_2.
\]
We record the special case
\[
\xi_2 (c\otimes v)=\sum_{p_1, p_2\in P} \frac{\xi_0(c^{(p_1, p_2)}v)}{(z_1-p_1)(z_2-p_2)}dz_1 dz_2,
\]
where $c^{(p_1, p_2)}$ means that $c\in\glie\otimes\glie$ acts on $V_\lambda$ via the tensor factor 
$V_{\lambda(p_1)}\otimes V_{\lambda(p_2)}$.

The description  of $\xi_N$ is for general $N$ is more conveniently given as a $\Rscr_P$-homomorphism 
$\Lcal_{\glie,P,V}\to \Lcal og_{C,P}$. We then replace  $N$ by a finite (unordered) set $I$ so that we can use a 
graphical representation  of polydifferentials on $C^I$ as in Section \ref{sect:polydiff}. 
To this end, we consider the set $\Cscr_P(I)$ of  decompositions $\G$ of $I\sqcup P$ into a chains such that each 
chain meets $P$ in a unique point with that unique member being an end point. 
We think of $\G$ as a partial order on $I\sqcup P$ that has $P$ as its set of minimal elements. 
So each $i\in I$ has a unique immediate predecessor $i^-_\G\in I\sqcup P$. We first associate to such a 
$\G \in \Cscr_P(I)$ the logarithmic polydifferential on $C^I$ defined by
\[
\alpha_\G:= \prod_{i\in I} \frac{dz_i}{(z_i-z_{i^-_\G})},
\]
where in case $i^-_\G=p\in P$, one should read $p$ for  $z_p$. In other words, if the chain $\G_p$ that contains 
$p$ is $p<i_{p,1}<i_{p,2}<\cdots$, then 
$\alpha_\G$ is simply the  product over $p\in P$ of the polydifferentials 
$(z_{i_{p,1}}-p)^{-1}dz_{i_{p,1}} (z_{i_{p,2}}-z_{i_{p,1}})^{-1}dz_{i_{p,2}}\cdots $
If we are also given a map $i\in I\mapsto X_i\in \glie $, then we can form 
$X^{(p)}_{\G_p}:=\cdots \circ X^{(p)}_{i_2}\circ X^{(p)}_{i_1}$ as an element of the universal enveloping algebra 
$\U\!\glie$ of $\glie$, or rather, as an element of the $p$-th component  of the universal enveloping algebra 
$\U (\glie^P)$ of $\glie^P$. Since the $X^{(p)}_{\G_p}$ pairwise commute, their 
product is well-defined as an element of $\U (\glie^P)$. This construction is multi-linear in the sense that 
$X_\G$ depends linearly on
$\otimes _{i\in I} X_i$. We thus obtain a linear map $u_\G: \glie^{\otimes I}\to \U (\glie^P)$.
We then define $\xi_I\in \Hom(\glie^{\otimes I}\otimes V, H^0(C^I,\Omega_C(P)^{(I)}\la\log\ra)$ by
\[
\xi_I(Z\otimes v): =\sum_{\G\in \Cscr_P(I)} \xi_0(u_\G(Z)v)\alpha_\G.
\]
It is straightforward to check that this indeed defines  a $\Rscr_P$-homomorphism 
$\Lcal_{\glie,P,V}\to \Lcal og_{C,P}$. 
It is the unique such homomorphism with prescribed $\xi_0$ (this generalizes the second expression we found for $\xi_2$). The expression we found
for $\xi_2(X\otimes X\otimes v)$ generalizes as 
\[
\xi_N (X^{\otimes N}\otimes v)=\sum_{(p_1,\dots, p_N)\in P^N} \frac{\xi_0(X^{(p_N)}\cdots X^{(p_1)}v)}{(z_1-p_1)
\cdots (z_N-p_N)}dz_1\cdots  dz_N.
\]
Its value in $\infty^N$ is $\xi_0(T^{N}(X)v)dw_1 dw_2\cdots dw_N$. 
For $X=X_\theta$ and $N=\ell +1$, this must be zero and so 
we demand that $\xi_0(T^{\ell +1}(X_\theta) v)=0$. The $\glie$-invariance of $\xi_0\in V_\lambda^*$ 
implies that this condition is independent of
$\infty$ and guarantees the vanishing of $\xi_0(T^{\ell +1}(X_\theta) v)$ on the main diagonal.
As this vanishing property is also sufficient and $\xi$ is determined by $\xi_0$,  
we thus recover the well known fact that $\Vcal^{\hat\lambda}(\PP^1)$ can be identified with the linear forms on
$V_{\lambda}$ that are invariant under  $\glie$ and killed by $T(X_\theta)^{\ell +1}$. But this construction associates 
to such an element a series of polydifferentials.
\end{example}

\section{WZW connection versus Gau\ss-Manin connection}\label{sect:confblocks}

\subsection{Conformal blocks embedded in a De Rham bundle}\label{subsect:confblocks}
In  this section  $k=\CC$ and we are given  a smooth  proper  morphism $F:\Ccal\to S$ with $S$ a smooth complex variety,  
whose geometric fibers are connected curves of genus $g$ and a finite nonempty collection $P$ of pairwise disjoint sections  $p: S\to  \Ccal$ of $F$. Although this is irrelevant at first,  we shall in addition  assume that the normal bundle 
of each of these sections has been trivialized by giving  for every $p\in P$  a generator 
(denoted $\vec p$) of $p^*\theta_{\Ccal/S}$. We then write $\vec P$ for $\{\vec p\}_{p\in P}$. 

Let  $j: {\mathring\Ccal}\subset \Ccal$ stand for the complement of the union of these sections and  put  $\mathring F=Fj$.  We  extend our earlier notation involving $P$ in an obvious manner: for every $p\in P$ we put $\Ocal_{p}:=\varprojlim_n 
\Ocal_\Ccal/\Ocal_\Ccal(-n p)$, $\Kcal_{p}:=\varinjlim_n\Ocal_p(np)$ and $\Kcal_P:=\oplus_{p\in P}\Kcal_{p}$. 
Note that $\Ocal_{p}$ contains the ideal 
$\mfrak_p:=\varprojlim_n \Ocal_S(-p)/\Ocal_S(-n p)$ for which $F^*$ induces an isomorphism 
$\Ocal_p/\mfrak_p\cong \Ocal_S$.   

Since the issues we discuss here are local on $S$, we shall assume that $S$ is affine and that there exists for each $p\in P$ an  $\Ocal_S$-algebra isomorphism  
$\phi_p: \Ocal_{p}\cong \Ocal_S[[t]]$ which lifts the given isomorphism $\Ocal_p/\mfrak_p^2\cong \Ocal_S[[t]]/(t^2)$ 
(equivalently, which takes $\vec p$ to $d/dt$). So any two such isomorphisms  are  equal modulo $t^2\Ocal_S[[t]]$. 

There is an evident relative version of the central extension of the loop algebra $ \glie\Kcal_P:=\glie\otimes \Kcal_P$:
\[
0\to \Ocal_S\to \widehat{\glie\Kcal_P}\to \glie\Kcal_P\to 0.
\]
for which we can  form the $\Ocal_S$-modules
\[
\tilde\Vcal_{\hat\lambda}(\Ccal/S,P)=\tilde\Vcal_{\hat\lambda}(\Kcal_P)_{\mathring F_*\glie\Ocal_{\mathring\Ccal}} 
\text{ and  } \Vcal_{\hat\lambda}(\Ccal/S,P)=\Vcal_{\hat\lambda}(\Kcal_P)_{\mathring F_*\glie\Ocal_{\mathring\Ccal}}.
\]
The $\Ocal_S$-module $\Vcal_{\hat\lambda}(\Ccal/S,P)$  is called the \emph{bundle of conformal blocks} defined by 
these data, and  is known to be  a locally free $\Ocal_S$-module of finite rank (\cite{TUY}, see also \cite{looij2013} (\footnote{Presumably  a proof of this fact can also be based on Theorem \ref{thm:globaldual1}.})). 
As before, we often omit $P$ from the notation by regarding $\hat\lambda$ as a weight valued function on the 
total space. But even if all weights are equal to $(\ell,0,\dots, 0)$, we still insist  $P$ that  be nonempty.
The relative versions of  Theorems \ref{thm:globaldual1}  and  \ref{thm:globaldual2} identify this with the  
$\Ocal_S$-module 
$\Vcal^{\hat\lambda}_\qlog (\Ccal/S)$  resp.\   $\Vcal^{\hat\lambda}_\llog (\Ccal/S)$, 
whose definitions should be obvious (polydifferentials are replaced by relative polydifferentials).  
The WZW-connection is a (projectively) flat connection $\nabla^{WZW}$ on the  bundle 
$\Vcal_{\hat\lambda}(\Ccal/S)$ that we shall later explicate. It determines a connection  on  its dual that is characterized by the property that the duality pairing is flat. So this gives  also a connection on  
$\Vcal_{\hat\lambda}(\Ccal/S)$ and $\Vcal_{\hat\lambda}(\Ccal/S)$; we  shall still denote that connection by $\nabla^{WZW}$. 

As we have seen  $\Vcal^{\hat\lambda}_\qlog (\Ccal/S)$  and   $\Vcal^{\hat\lambda}_\llog (\Ccal/S)$ can, at least fiberwise,  be expressed in terms of the cohomology bundles of configuration spaces. In order to be able to be more precise,  let us agree on a bit of notation first. For any integer $N\ge 0$, let
\[
F^{(N)}: (\Ccal/S)^{(N)}\to S \text{  resp.\  } \inj_N\! F: \inj_N(\Ccal/S)\to S
\] 
stand for the $N$-fold fiber product  $\Ccal\times _S\Ccal\times_S\cdots\times_S\Ccal$ over $S$ resp.\ 
for its intersection with $\inj_N(\Ccal)$ and similarly with $\Ccal$ replaced by $\mathring\Ccal$. 
The bundle $\mathring F: \mathring\Ccal/S \to S$ is locally trivial and hence so is 
$\inj_N\!\mathring F)$. It follows that $R^N\inj_N\!\mathring F_*\underline\QQ(N)$ is a  local system that supports a variation of mixed Hodge structure whose underlying vector bundle
\begin{equation}\label{eqn:gm}
\Ocal_S\otimes_{\underline \QQ} R^N\inj_N\!\mathring F_*\underline\QQ(N)\cong  R^N\inj_I\!\mathring F_*(\inj_I\!\mathring F)^{ -1}\Ocal_S 
\end{equation}
has  after a twist with sign character $\sign (N)$, the bundle  $F^{(N)}_*\Omega_{\Ccal/S}(P)^{(N)}\la \llog\ra$ as its  $F^0$-Hodge subbundle 
and for which $F^{(N)}_*\Omega_{\Ccal/S}(P)^{(N)}\la \qlog\ra$ is contained in the $F^{-1}$-Hodge subbundle.
We here inserted a Tate twist on the left because it  is natural here. As mentioned in  Remark \ref{rem:tateinsertion}, this ensures that residues preserve the Hodge filtrations, so that this  gives rise $\Rscr_P$-module in the category of mixed Hodge structures.
Mixed Hodge theory then tells us  that  the graded pieces $\Gr^\pt_W$ and $\Gr^F_\pt$ are zero  outside the range  $\{-N, \dots ,0\}$.  

The left hand side of the identity  \ref{eqn:gm} comes with an evident  flat connection $\nabla^{GM}$ (the \emph{Gau\ss-Manin} connection): if $D\in\theta_S$, then $\nabla^{GM}_D=D\otimes_{\underline \CC} 1$, where $D$ acts on $\Ocal_S$ in the usual manner. Its flat sections return the
complexified local system $R^N\inj_I\!\mathring F_*\underline\CC(N)$. 

A version of Theorem \ref{thm:globaldual1} with parameters
yields a $\Ocal_S$-homomorphism
\begin{multline*}
\gamma_N: \Vcal^{\hat\lambda}_\qlog (\Ccal/S)\to \Hom_\CC\big(\glie^{\otimes N}\otimes V_\lambda, F^{(N)}_*\Omega_{\Ccal/S}(P)^{(N)}\la \llog\ra\big)\subset \\
\subset \Hom_\CC\big(\glie^{\otimes N}\otimes V_\lambda, F^{-1}  
\big(\Ocal_S\otimes_{\underline \QQ} R^N\inj_N\!\mathring F_*\underline\QQ(N))\big)\subset
\Ocal_S\otimes_{\underline \CC}\Hom_\CC\big(\glie^{\otimes N}\otimes V_\lambda, R^N\inj_N\!\mathring F_*\underline\CC(N)).
\end{multline*}
Similarly, Theorem \ref{thm:globaldual2}  yields a $\Ocal_S$-homomorphism  
\begin{multline*}
\gamma'_N: \Vcal^{\hat\lambda}_\llog (\Ccal/S)\to \Hom_\CC\big(\glie^{\otimes N}\otimes V_\lambda, F^{(N)}_*\Omega_{\Ccal/S}(P)^{(N)}\la \llog\ra\big)\subset \\
\subset \Hom_\CC\big(\glie^{\otimes N}\otimes V_\lambda, F^0  
\big(\Ocal_S\otimes_{\underline \QQ(N)} R^N\inj_N\!\mathring F_*\underline\QQ)\big)\subset
\Ocal_S\otimes_{\underline \CC}\Hom_\CC\big(\glie^{\otimes N}\otimes V_\lambda, R^N\inj_N\!\mathring F_*\underline\CC(N)).
\end{multline*}
In either case, the  first arrow is an embedding for $N$ large enough. The interest of these maps is that 
both are $\Ocal_S$-homomorphisms whose ultimate target is the vector bundle underlying a local system. 
Our goal is to explicate the WZW-connection  by comparing it with thie Gau\ss-Manin connection in the following sense.
Letting $D\in \theta_S$ act on $\Ocal_S$ be derivation, then 
\[
\gamma_N\nabla^{WZW}_D- (D\otimes 1)\gamma_N: \Vcal^{\hat\lambda}_\qlog (\Ccal/S)\to \Ocal_S\otimes_{\underline \CC}\Hom_\CC\big(\glie^{\otimes N}\otimes V_\lambda, R^N\inj_N\mathring F_*\underline\CC(N)).
\]
is the difference  between two derivations from $\Vcal^{\hat\lambda}_\qlog (\Ccal/S)$ to the $\Ocal_S$-module on the right and hence is
$\Ocal_S$-linear. We shall exhibit this $\Ocal_S$-homomorphism. This will include a description of how this factors through a
Kodaira-Spencer homomorphism and this is why we discuss this notion first.

\subsection{The Kodaira-Spencer homomorphism}\label{subsect:ks}
It is known that ${\mathring\Ccal}$ and the  $\{\Ocal_p\}_{p\in P}$ define an acyclic covering of 
$\Ccal$ for the direct image functor $F_*$ applied to any  coherent $\Ocal_\Ccal$-module
$\Fcal$ so that
\[
R^1F_*\Fcal= \coker\big(F_*\Fcal (\infty P)\to \oplus_{p\in P} \Fcal\otimes_{\Ocal_p}(\Kcal_p/\Ocal_p)\big),
\]
where $\Fcal (\infty P):=\sum_{n\ge 0} \Fcal (n P)$.  If $\Fcal$ is a locally free of finite rank, then the pairing  
that is perfect by relative Serre duality
$R^1F_*\Fcal \otimes F_*\Hcal om_{\Ocal_{\Ccal}} (\Fcal,\Omega_{\Ccal/S})\to \Ocal_S$
is given by the  residue pairing
\[
\textstyle [(s_p\otimes f_p)_{p\in P}]\otimes \sigma\mapsto \sum_{p\in P}\res_p f_p\sigma (s_p),
\]
where $s_p\in\Fcal_p$, $f_p\in \Kcal_p$ and  the brackets indicate the image in $R^1F_*\Fcal$. 

We apply this to the case when $\Fcal=\theta_{\Ccal/S}(-2P)$. 
Let $D$ and  $\tilde D$  be as above.
The isomorphism of  $\Ocal_S$-algebras $\phi_p: \Ocal_p\cong \Ocal_S[[t]]$ yields a lift $\tilde D_p$ of $D$ at $p$.  
Then $D_p:=\tilde D-\tilde D_p\in \theta_{\Ccal/S}\otimes\Kcal_p$ and the preceding tells us that 
$D_P= (D_p)_{p\in P}$ defines
a section  of $R^1F_*\theta_{\Ccal/S}(-2P)$. It is easily verified that this section only depends on $D$ and that  
this dependence is $\Ocal_S$-linear.
We thus have defined an $\Ocal_S$-linear homomorphism
\[
\textstyle \KS_{\Ccal/S,\vec P}: \theta_S\to R^1F_*\theta_{\Ccal/S}(-2P),
\]
referred to as the \emph{Kodaira-Spencer homomorphism}. 
This in fact the coboundary map of the long exact sequence that is associated to the functor $F_*$, 
when applied  the exact sequence 
\[
0\to \theta_{\Ccal/S}(-2P)\to \theta_{\Ccal}(\log^2 P)\to F^*\theta_S\to 0,
\]
where $\theta_{\Ccal}(\log^2 P)$ stands for the sheaf of vector fields that are tangent 
to the given sections up order two.
We observe that the Serre dual of $R^1F_*\theta_{\Ccal/S}(-2P)$ is 
$F_*\Omega^{\otimes 2}_{\Ccal/S}(2P)$, the duality being given by a residue pairing: if a local section of 
$R^1F_*\theta_{\Ccal/S}(-2P)$  resp.\  $F_*\Omega^{\otimes 2}_{\Ccal/S}(2P)$ is represented by  
$(s_p\in \Kcal_p\otimes \theta_{\Ccal/S})_{p\in P}$ resp.\ $\eta$, then the pairing is given by 
$\sum_{p\in P}\res_p\la s_p|\eta\ra\in \Ocal_S$.  In particular, the adjoint of  $\KS_{\Ccal/S,\vec P}$ 
defines a $\Ocal_S$-homomorphism $F_*\Omega^{\otimes 2}_{\Ccal/S}(2P)\to 
\Omega_S$. 
So Serre duality over $S$ enables us to regard the latter as a  section of the $\Ocal_S$-dual of 
$F_* \Omega_{\Ccal/S}^{\otimes 2}(2P))$.

\subsection{Segal-Sugawara construction}\label{subsect:ss}
In this subsection we return for a moment  to the situation of Section \ref{sect:loopalg} of a single DVR $(\Ocal, \mfrak)$
with residue field $K$. Recall that  $d: K\to \omega$ denotes the universal continuous derivation and that 
$\theta$ stands for the  $K$-dual of $\omega$
(the module of continuous $k$-derivations $K\to K$). The latter is at the same time an infinite dimensional $k$-Lie algebra. Given a uniformizer $t\in\mfrak$, then
a topological $k$-basis of $\theta$  is   $\{D_n:=t^{n+1}\frac{d}{dt}\}_{n\in \ZZ}$ on which the Lie bracket takes a 
simple form: $[D_i, D_j]=(j-i)D_{i+j}$.  
We have on the direct sum $k\hbar\oplus \theta$ (where $\hbar$ is just a name of a nonzero element of a 
one-dimensional $k$-vector space) a Lie bracket which makes it a nontrivial central extension, 
known as the \emph{Virasoro algebra}:
\[
0\to k \hbar \to \hat\theta \to \theta \to 0,
\]
whose Lie bracket is given by
\[
[\hat D_i +k\hbar ,\hat D_j+k\hbar ]= (j-i)\hat D_{i+j}+\frac{i^3-i}{12}\delta_{i+j,0}\hbar. 
\]
Let $\{ E_\kappa\}_\kappa$ be an orthonormal basis of $\glie^*$ relative to $c$, so that 
$c=\sum_{\kappa} E_\kappa\otimes E_\kappa$. Consider for $n\in \ZZ$ the formal sum
\[
\textstyle L(D_n):= \frac{-1}{2(\ell+\check{h})}\sum_ {i+j=n}\sum_\kappa : E_\kappa t^{i}\circ 
E_\kappa t^j: 
\]
Here we use the normal ordering notation,  which prescribes that the factor with the 
highest index comes last and hence acts first (the exponent of $t$ serves as index). 
In other words,
\[
L(D_n):= \frac{-1}{\ell+\check{h}}\sum_\kappa\big(\frac{1}{2}E_\kappa t^{n/2}\circ 
E_\kappa t^{n/2} +\sum_ {i>n/2} E_\kappa t^{n-i}\circ E_\kappa t^i\big), 
\]
where the term $\half t_2^{n/2}t_1^{n/2}$  must be suppressed when $n$ is odd.  An arbitrary 
$D\in\theta$ is of the form $\sum_{n>-\infty} c_n D_n$ and we then put $L(D):=\sum_n c_nL(D_n)$.
A  more canonical construction, which also  makes sense of such series (whose index set after all
runs through all the integers)  is given in \cite{looij2013} with a slightly different notation. Since we will
only use these operators as acting on the highest weight representations 
$\tilde \Vcal_{\hat\lambda}(K)$ and $\Vcal_{\hat\lambda}(K)$, the first clause of the following proposition
(whose proof can be found in \cite{br1}, see also \cite{looij2013})  relieves any concerns one might have.

\begin{proposition}[Segal-Sugawara representation]\label{prop:sugawara}
Only finitely many terms in $L(D)$ act nontrivially on any given element of  
$\tilde \Vcal_{\hat\lambda}(K)$, so that we have 
a well-defined action of $L(D)$ on $\tilde \Vcal_{\hat\lambda}(K)$. 
This action defines a representation of the Virasoro algebra $\hat\theta$ on
$\tilde \Vcal_{\hat\lambda}(K)$ which takes the central element $\hbar$  of $\hat\theta$ to scalar 
multiplication with $\ell (\ell+\check{h})^{-1}\dim \gfrak$.
This representation normalizes the one of $\widehat{\glie K}$, in fact $[L(D),\ad (Xf)]= \ad (X\, D\! f)$.
\end{proposition}

We shall need  the following simple property.

\begin{lemma}\label{lemma:kill}
For any  $D\in F^1\theta$ and $v\in V_\lambda$,  we have $L(D)\circ v=0$. 
\end{lemma} 
\begin{proof}
In view of Proposition \ref{prop:sugawara}, we only need to check this for $D=D_n$ with $n\ge 1$. 
This is then obvious, for 
$L(D_n)=\sum_{i+j=n} \sum_{\kappa} :E_\kappa t^i\circ E_\kappa t^j  \circ v$ is in fact a sum over 
$E_\kappa t^i\circ E_\kappa t^j \circ V$ with $j>0$,  so that then $E_\kappa t^j\circ v=0$.
\end{proof}

We will run into a space which strictly contains $\tilde\Vcal^{\hat\lambda}_\llog (K)$, for if 
$\eta\in \omega^{(N+1)}\la \qlog \ra$ and $f\in K$, then
$\res_1 \eta\pi_1^*f$ may have arbitrary poles along the coordinate divisors.
The following lemma shows how this can happen.

\begin{lemma}\label{lemma:} For every $f\in K$, the residue $\res_{t=0} (x-t)^{-1}f(t)dt$ reproduces the polar part of $f(x)$
and $\res_{t=0} (x-t)^{-2}f(t)dt$ reproduces minus the derivative of the polar part of $f(x)$.
\end{lemma}
\begin{proof}
Since differentiating with respect to $x$ commutes with taking $\res_{t=0}$, the second assertion follows from the first. By continuity, it suffices to  check the  first assertion for the basic case $f=t^r$, where $r\in\ZZ$. We then observe that
\[
\res_{t=0}\frac{t^rdt}{x-t}=\res_{t=0}\frac{t^rdt}{x(1-t/x)}= \res_{t=0}\frac{t^rdt}{x}\sum_{n\ge 0} \Big(\frac{t}{x}\Big)^n.
\]
is zero unless $r<0$, in which case we get only a contribution for $n=-r-1$ yielding $x^r$. 
\end{proof}

On the other hand, the pole order along a diagonal will not go up. 
In particular, such a  residue takes $\omega^{(N+1)}\la \qlog \ra$ to $\omega^{(N)}(2)$.
Let us therefore denote  for $r=1, 2, \dots$ by $\tilde\Vcal ^{\hat\lambda}(r,K)$ the space
of sequences of symmetric maps $\xi=(\xi_N: \glie^{\otimes N}\otimes V\to \omega^{(N)}(r))_{N\ge 0}$  
satisfying the residual compatibility properties that  we imposed to define  
$\tilde\Vcal^{\hat\lambda}_\llog (K)$ resp.\  $\tilde\Vcal^{\hat\lambda}_\qlog (K)$ so that 
$\tilde\Vcal^{\hat\lambda}_\llog (K)\subset \tilde\Vcal^{\hat\lambda} (1, K)$ and 
$\tilde\Vcal^{\hat\lambda}_\qlog (K)\subset \tilde\Vcal^{\hat\lambda} (2, K)$.
We define $\Vcal ^{\hat\lambda}(r,K)$ analogously (as the linear maps that vanish on
$\glie^{\otimes *}\otimes X_\theta^{\lambda_0 +1}\otimes 1_\lambda$).

\begin{lemma}\label{lemma:sugawara}
Choose  a uniformizer $t$ for $\Ocal$.
Let $\xi\in \tilde\Vcal^{\hat\lambda}_\qlog (K)$. Given $W\in \glie^{\otimes N}$ and $v\in V_\lambda$, then 
\[
\xi'_{N+2}(W\otimes c\otimes v):=\xi _{N+2}(W\otimes c\otimes v)-\xi_{N+2, \dots, 3}(W\otimes v) \ell \dim\glie \frac{dt_1dt_2}{(t_2-t_1)^2}
\]
has no pole along $\Delta_{12}$. If we pull  this polydifferential back along $\Delta_{12}$, then we obtain a quadratic differential in the $t_{12}$-variable with coefficients in  $\omega^{(N)}(2)$ and  for any $D\in \theta$, we have
\[
\res_{[12]}\iota^{[12]}_{D} \Delta_{12}^*\xi'(W\otimes c\otimes v)=-2(\ell +\check h)\res_2\res_1
\xi_{N+2}(W\otimes L(D)\otimes v),
\]
where $\iota^{[12]}_{D}$ stands for taking the  inner product with $D$ and $\res_{[12]}$ for the residue, both using the first ($t_{12}$) component.  The map
$W\otimes v\in \glie^{\otimes N}\otimes V\mapsto \res_2\res_1
\xi_{N+2}(W\otimes L(D)\otimes v)$ gives rise to a linear map  
\[
\Lscr(D): \tilde\Vcal^{\hat\lambda}_\qlog (K)\to \tilde\Vcal^{\hat\lambda} (2,K)
\]
which is zero when $D\in F^1\theta$. It takes $\Vcal^{\hat\lambda}_\qlog (K)$ to $\Vcal^{\hat\lambda} (2,K)$.
\end{lemma}

\begin{proof}
We  verify this for $D=D_n=t^{n+1}\frac{d}{dt}$.
The definition of $L(D_n)$ shows  that 
\[
-2(\ell +\check h)\res_{2}\res_{1}\xi (W\otimes L(D_n)\otimes v)
= \res_{2}\res_{1} \xi(W\otimes c\otimes v)\big( t_2^{n/2}t_1^{n/2} +
\sum_{\substack {n_1+n_2=n,\\ n_1> n_2}} 2t_2^{n_2} t_1^{n_1}\big),
\]
where the term $t_2^{n/2}t_1^{n/2}$ is suppressed in case $n$ is odd. Note that since $ \xi(W\otimes c\otimes v)$ has a pole along $t_1=0$ of order one at most, the right hand side vanishes when $n\ge 1$.

The symmetry of the tensor $c\in \glie\otimes\glie$ implies that 
$\xi(W\otimes c\otimes v)$ has no residue along along $\Delta_{12}$. Since $\check c$ takes on the tensor $c=\sum_\kappa E_\kappa\otimes E_\kappa$ the value $\dim\glie$,  it follows that we may write
\[
\xi_{N+2}(W\otimes c\otimes v)
=\sum_{r_1\ge 0, r_2\ge 0} 
\psi_{r_1,r_2}t_1^{r_1-1}t_2^{r_2-1}dt_1dt_2 +\ell \dim \glie .\xi_{N+2, \dots, 3}(W\otimes v)\frac{dt_1dt_2}{(t_2-t_1)^2}
\]
for certain $\psi_{r_1,r_2}\in \omega^{(N)}(2)$. Here it must be understood that we first expand as a Laurent series  in $t_1$ and $t_2$, meaning for instance that  if $(t_i-t_1)^{-1}$ with  $i>2$ appears as a factor on the left, then it is expanded as $t_i^{-1}\sum_{r\ge 0} (t_1/t_i)^r$. Since $c$ is symmetric, we have 
$\psi_{r_1,r_2}=\psi_{r_2,r_1}$

Before we substitute this in the preceding expression,
let us note that if we do the same with $(t_i-t_1)^{-2}$ (i.e., first expand in $t_1$), then
\[
t_2^{n_2} t_1^{n_1}\frac{dt_1dt_2}{(t_2-t_1)^2}=t_2^{n_2-2} t_1^{n_1} \frac{dt_1dt_2}{(1-t_1/t_2)^2}=
\sum_{r\ge 0} (r+1)t_2^{-r-2+n_2}t_1^{r+n_1}dt_1dt_2.
\]
So the value of $\res_2\res_1$ on this  bidifferential is zero unless for some integer $r\ge 0$, $n_1=-r-1$ and $n_2=r+1$.
But this will never happen if we also assume that $n_1\ge n_2$. It follows that
\begin{multline*}
-2(\ell +\check h)\res_{2}\res_{1}\xi (W\otimes L(D_n)\otimes v)=\\=
\textstyle \res_{2}\res_{1}\Big(\sum_{r_1\ge 0, r_2\ge 0} \psi_{r_1,r_2}t_1^{r_1-1}t_2^{r_2-1}dt_2dt_1.
\big(t_2^{n/2}t_1^{n/2} +2\sum_{\substack{n_1+n_2=n\\n_1> n_2}} t_2^{n_2} t_1^{n_1}\big)\Big)=\\
\textstyle = \psi_{-n/2,-n/2}+2\sum_{\substack{n_1+n_2=n\\n_1> n_2}}\psi_{-n_1,-n_2}=\sum_{n_1+n_2=n}\psi_{-n_1,-n_2}.
\end{multline*}
The last sum is finite and lies  in $\omega^{(N)}(2)$. It is straightforward to check that $\Lscr\xi$ satisfies the residue properties that makes it an element of $\tilde\Vcal^{\hat\lambda}(2,K)$. If $W$ as above ends with
$X_\theta^{\otimes m}$ in the sense that  $W=W'\otimes X_\theta^{\otimes m}$, then the same is true for
$\res_2\res_1 \xi_{N+2}(W\otimes L(D_n)\otimes v)$. This implies that $\Lscr(D_n)$ induces a homomorphism of
$\Vcal^{\hat\lambda}_\qlog (K)\to \Vcal^{\hat\lambda}(2,K)$

The identity
$\xi '(W\otimes c\otimes v)=\sum_{r_1\ge 0, r_2\ge 0}\psi_{r_1,r_2}t_1^{r_1-1}t_2^{r_2-1}dt_1dt_2$ 
shows that $\xi '(W\otimes c\otimes v)$ has no pole along $\Delta_{21}$.  Let us write $t$ for $t_{21}$.
Then its pull back along $\Delta_{21}$ is
$\sum_{r_1\ge 0, r_2\ge 0}\psi_{r_1,r_2}t^{r_1+r_2-2}(dt)^2$. Its inner 
product with $D_n=t^{n+1}\frac{d}{dt}$ is $\sum_{r_1\ge 0, r_2\ge 0}\psi_{r_1,r_2}t^{r_1+r_2 +n-1} dt$
and the residue of this with respect to $t$ is $\sum_{r_1+r_2=-n}  \psi_{r_1,r_2}$.
This proves the lemma.  
\end{proof}


\subsection{The WZW connection}\label{subsect:wzw} 
We return to the situation of Subsection \ref{subsect:confblocks}. In particular, we are \emph{given} a trivialization of each normal bundle $p^*\theta_{\Ccal/S}$   and we have \emph{chosen} trivializations  
$\phi_p: \Ocal_p\cong \Ocal_S[[t]]$ compatible with these.
The fact that $P\not=\emptyset$  ensures that the morphism $\inj_N\!\mathring  F:(\Ccal/S)^{(N)}\to S$ is affine for every $N$.

We first recall the way the WZW connection was defined in \cite{looij2013}. We first note that the isomorphisms of $\Ocal_S$-algebras $\phi_p: \Ocal_p\cong \Ocal_S[[t]]$ trivialize $\tilde\Vcal_{\hat\lambda}(\Kcal_P)$ as a  $\Ocal_S$-module.  Let $D\in\theta_S$. Since ${\mathring\Ccal}/S$ is affine,  $D$ can be lifted  along $\mathring F$ to a rational vector field $\tilde D$ on $\Ccal$ that is regular on ${\mathring\Ccal}$. In terms of $\phi_p$, it  will decompose at $p$ as
$\tilde D=D\otimes_\CC 1+D_p$, where $D_p$ is a vertical vector field at $p$, but with a 
possible pole along $p$ (in terms of our relative uniformizer $t$ it defines an element of $\Ocal_S\otimes_\CC\theta$). We regard 
$D_P:=\sum_{p\in P} D_p$ as an $\Ocal_S$-derivation of $\Kcal_P$ and interpret $L(D_P)$ as $\sum_{p\in P} L(D_p)$.  Then
\[
\textstyle \hat D:=D\otimes_\CC 1+L(D_P)
\]
acts as an operator acting in $\tilde\Vcal_{\hat\lambda}(\Kcal_P)$ by the following rule: if $\psi=X_Nf_N\circ\cdots \circ X_1f_1\circ v$, then 
\[
\textstyle\hat D \psi = 
X_N Df_N \circ X_{N-1}f_{N_1}\circ\cdots \circ X_1f_1\circ v + L(D_P)\circ \psi
\]
It then follows from Proposition \ref{prop:sugawara} that for $f\in  \Kcal_P$ and $X\in\glie$ the identity $[\hat D , Xf]= X\tilde Df$ holds (as operators acting on $\tilde\Vcal_{\hat\lambda}(\Kcal_P)$). Hence
\[
\hat D\psi=\textstyle \sum_{i=1}^N X_Nf_N\circ \cdots\circ  X_i\tilde D\!f_i\circ  \cdots \circ X_1f_1\circ v+ 
X_Nf_N\circ\cdots \circ X_1f_1\circ L(D_P)\circ v.
\]
This formula also shows that the operator $\hat D$, though dependent on the given trivializations of the normal bundles of the sections,  is independent of the chosen $\phi_p$, for a different choice of $\phi_p$ modifies $D_p$ by a $D'_p$ which has a zero along $p$ (in terms of $\phi_p$ we have  $D'_p\in \Ocal_S\otimes F^0\theta$) and according to Lemma \ref{lemma:kill} such a $D'_p$ kills 
$v$ when considered as an element of $\tilde\Vcal^{\hat\lambda}(\Kcal_P)$.  The formula $[\hat D , Xf]= X\tilde Df$ also shows that $\hat D$ preserves 
$\mathring F_*(\glie\Ocal_{\mathring\Ccal})\tilde\Vcal_{\tilde\lambda}(\Kcal_P)$ and hence acts in $\tilde\Vcal_{\tilde\lambda}(\Ccal/S)$.
The latter action is easily checked to independent of the chosen lift $\tilde D$ of $D$: it  only depends on $D$. This this gives us the action of $D$ on $\tilde\Vcal_{\tilde\lambda}(\Ccal/S)$ under the WZW-connection: it   equals  $\nabla^{WZW}_D$.

Let us  write  $\tilde D_i\psi$ for $X_Nf_N\otimes\cdots\otimes X_i\tilde D\!f_i\otimes \cdots \otimes X_1f_1\otimes v$.  We also abbreviate $X_Nf_N\otimes\cdots \otimes X_1f_1\otimes L(D_P)\otimes v$ by $\Lscr^\dagger(D_P)\psi$. This notation is suggested by Lemma \ref{lemma:sugawara}, as this is an adjoint of $\Lscr(D)$: we have 
\[
\la \xi |\Lscr^\dagger(D_P)\psi\ra= \la \Lscr(D_P)\xi |\psi\ra.
\] 
The preceding formula is now concisely written as 
\[
\textstyle \hat D = \sum_{i=1}^N \tilde D_i + \Lscr^\dagger(D_P).
\]
Since $\tilde D$ defines a derivation in $F_*\Ocal_{\mathring\Ccal}$, the identity $[\hat D ,Xf]= X\tilde Df$  (used in case $f\in F_*\Ocal_{\mathring\Ccal}$) also shows that 
 $L(D)$ descends to an action in both $\tilde\Vcal_{\hat\lambda}(\Ccal/S)$ and $\Vcal_{\hat\lambda}(\Ccal/S)$.  It defines an action of a (central) Virasoro extension of $\theta_S$ by $\Ocal_S\hbar$ in
$\Vcal_{\hat\lambda}(\Ccal/S)$. We may think of this as endowing 
this sheaf with a projectively flat connection.

We transfer this connection to the dual bundle $\Vcal^{\hat\lambda}(\Ccal/S)$ that is characterized  by the property that the natural pairing $\Vcal^{\hat\lambda}(\Ccal/S)\otimes_{\Ocal_S} \Vcal_{\hat\lambda}(\Ccal/S)\to \Ocal_S$ be flat. We identify $\Vcal^{\hat\lambda}(\Ccal/S)$ with $\Vcal^{\hat\lambda}_\qlog(\Ccal/S)$  so that when $\xi$ is a local section of one of these, then
\[
\la \hat D\xi | \psi\ra +\la \xi | \hat D\psi\ra=D\la \xi | \psi\ra. 
\]

In order to see what $\hat D\xi$ is like, we need the following lemma.

\begin{lemma}\label{lemma:residuederivation}
Let $R$ be a $k$-algebra, $\Kcal$ a topological $R$-algebra isomorphic to $R((t))$, $d:\Kcal\to \omega_{\Kcal/R}$ the universal continuous $R$-derivation and $\res : \omega_{\Kcal/R}\to R$ the $R$-linear extension of the usual residue.
Let $\tilde D$ be a continuous $k$-derivation of $\Kcal$ which extends a $k$-derivation $D$ of $R$. Then $\tilde D$
acts as a derivation on $\omega_{\Kcal/R}$ by the usual formula (i.e., as $\tilde D=d\iota_{\tilde D} + \iota_{\tilde D}d$) applied to a lift of $\alpha$ in $\omega_{\Kcal/k}$
and the residue pairing  $\Kcal \otimes_R \omega_{R/S}\to R$ defined by $ f\otimes\alpha \mapsto \res(f\alpha)$ satisfies 
\[
D(\res (f\alpha))=\res (\tilde D f)\alpha + \res f \tilde D\alpha.
\]
In particular,  the right hand side is independent of the choice of $\tilde D$.
\end{lemma}
\begin{proof}
Any two lifts of $\alpha$ in $\omega_{\Kcal/k}$ differ by an element of $\Kcal\otimes_R\Omega_{R/k}$ and so 
the assertion that $d\iota_{\tilde D} + \iota_{\tilde D}d$  defines a derivation in $\omega_{\Kcal/R}$ amounts to the statement that it takes $\Kcal\otimes_R\Omega_{R/k}$  to itself. This left as an exercise.

If $t\in \Kcal$ is such that $\Kcal=R((t))$, then the basic case is 
when $f=\phi t^m$, $\alpha= \psi  t^{n-1}dt$  and $\tilde D=D+ \tau t^{r+1}\p /\p t$ with
$\phi, \psi, \tau\in \Ocal_S$.
The verification is then straightforward:  we find 
\begin{align*}
D(\res(f\alpha)) &=D(\phi \psi)\delta_{m+n,0},\\
\res_p (\tilde D f)\alpha &=\psi D\phi \delta_{m+n,0} +m\tau\phi\psi \delta_{m+n+r,0}, \\ 
\res_p f \tilde D\tilde\alpha &=(n+r)\tau\phi\psi \delta_{m+n+r,0} +\phi D\psi \delta_{n+m,0},
\end{align*}
which clearly implies the assertion.
\end{proof}

We want to apply this  to a multidimensional (polydifferential) situation. To this end, 
we note that the rational lift $\tilde D$ of $D$ extends to lifts $\tilde D^{(1)}, \tilde D^{(2)}, \dots , \tilde D^{(N)}$ for the successive projections
\[
S\leftarrow\Ccal=(\Ccal/S)^{(1)}\leftarrow (\Ccal/S)^{(2)}\leftarrow \cdots  \leftarrow (\Ccal/S)^{(N)}
\]
with $\tilde D^{(i)}$ regular on $(\mathring\Ccal/S)^{(i)}$. Indeed, $\theta_{(\Ccal/S)^{(i)}}$ can be obtained as the fiber product (of vector bundles) defined by  $i$ copies of $\theta_{\Ccal}\to F^*\theta_S$ pulled back to $(\Ccal/S)^{(i)}$ and $(\tilde D^{(i)},D)$ corresponds by taking the pair $(\tilde D, D)$ in each such copy.

\begin{corollary}\label{cor:residuederivation}
Let $\eta$ be a section of  $\Omega_{\Ccal/S}^{(N)}$ over $\inj_N(\mathring\Ccal/S)$.
Then for any $N$-tuple of sections 
$f_1, \dots , f_N$ of $\Kcal_P$ we have  
\[
D (r_N\cdots r_1 \pi_1^*f_1\cdots \pi_N^*f_N\, \eta )=
\sum_{i=1}^N r_N\cdots r_1 \big(\pi_1^*f_1\cdots \pi_i^*\tilde Df_i\cdots\pi_N^*f_N\, \eta \big)+r_N\cdots r_1 \big(\pi_1^*f_1\cdots \pi_N^*f_N D^{(N)}\tilde\eta \big),
\]
where $r=\sum_{p\in P}r ^p :\omega_{\Kcal_P/S}\to \Ocal_S$ stands for the residue sum  as before.
\end{corollary}
\begin{proof}
Assume $N\ge 1$.  We apply Lemma \ref{lemma:residuederivation} to $\alpha:= r_{N-1}\cdots r_1 \pi_1^*f_1\cdots \pi_{N-1}^*f_{N-1}\, \eta$ (regarded as a relative differential on $\Ccal/S$ near $P$),  the projection $\Ccal=(\Ccal/S)^{(1)}\to S$ and $f_N$.  This gives:
\begin{multline*}
D (r_N\cdots r_1 \pi_1^*f_1\cdots \pi_N^*f_N\,\eta )=\\=
r_N (r_{N-1}\cdots r_1 \pi_1^*f_1\cdots \pi_{N-1}^*f_{N-1} \pi_N^*\tilde Df_N\, \eta) +
r_N\big(\pi_N^*f_N.\tilde D^{(1)} (r_{N-1}\cdots r_1 \pi_1^*f_1\cdots \pi_{N-1}^*f_{N-1}\, \eta )\big).
\end{multline*}
We have a similar formula for $\tilde D^{(1)} (r_{N-1}\cdots r_1 \pi_1^*f_1\cdots \pi_{N-1}^*f_{N-1}\, \eta)$, obtained by  applying 
Lemma \ref{lemma:residuederivation} to $\alpha:=r_{N-2}\cdots r_1 \pi_1^*f_1\cdots \pi_{N-2}^*f_{N-2}\, \eta$ (regarded as a relative differential on $(\Ccal/S)^{(2)}$ near $(P/S)^{(2)}$), the projection $(\Ccal/S)^{(2)}\to (\Ccal/S)^{(1)}$ with
$\tilde D=\tilde D^{(2)}$ viewed as a lift of $D=\tilde D^{(2)}$ and for which we take $f$ to be the pull-back of $f_{N-1}$ to the fiber factor of $(\Ccal/S)^{(2)}\to (\Ccal/S)^{(1)}$ (we might denote this by $\pi_{N-1}^*f_{N-1}$). Then the asserted identity follows readily with induction.
\end{proof}

We expect not to cause any  confusion by writing  $\tilde D$ for $(\tilde D^{(N)})_{N\ge 0}$ with $\tilde D^{(N)}$ acting in degree $N$. 
Lemma \ref{lemma:sugawara} has the following global, relative counterpart:

\begin{lemma}\label{lemma:Loperator}
Let  $\xi\in \tilde\Vcal^{\hat\lambda}_\qlog (\Ccal/S)$,   $W\in \glie^{\otimes N}$ and $v\in V_\lambda$. Then 
\[
\xi'_{N+2}(W\otimes c\otimes v):=\xi _{N+2}(W\otimes c\otimes v)-\xi_{N+2, \dots, 3}(W\otimes v) \ell \dim\glie \pi_{2,1}^*\zeta_2.
\]
has no pole along $\Delta_{12}$ and
\begin{equation}\label{eqn:Loperator}
r^p_2 r^p_1
\xi_{N+2}(W\otimes L(D_p)\otimes v)=\frac{-1}{2(\ell +\check h)}r^p_{[12]}\iota^{[12]}_{D_p} \Delta_{12}^*\xi_{N+2}'(W\otimes c\otimes v).
\end{equation}
This is zero when $D_p$ vanishes of order $\ge 2$ along $p$.
By assigning to $\xi$ and $W\otimes v$ the sum over $p\in P$ of the above expression we obtain a $\Ocal_S$-homomorphism
\[
\textstyle \Lscr_N(D_P): \tilde\Vcal^{\hat\lambda}_\qlog (\Ccal/S)\to\Hom_\CC(\glie^{\otimes N}\otimes V,  
F^{(N)}\Omega_{\Ccal/S}(P)^{(N)}(2))
\]
which only depends on the image $[D_P]$ of $D_P$ in $R^1F_*\theta_C(-2P)$. This results in a  $\Ocal_S$-homomorphism
\[
\Lscr_N^u: R^1F_*\theta_C(-2P)\to\Hom_\CC(\glie^{\otimes N}\otimes V,  
F^{(N)}\Omega_{\Ccal/S}(P)^{(N)}(2))
\]
with the property that $\Lscr_N(D_P)=\Lscr_N^u\KS_{\Ccal/S,\vec P}(D)$.
\end{lemma}
\begin{proof}
We only prove the assertions regarding the factorization over the Kodaira-Spencer homomorphism, for the other assertions follow in a straightforward manner from  Lemma \ref{lemma:sugawara}. It is already clear that 
$\Lscr(D_P)$ vanishes in case every $D_p$ lies in $F^1\theta_p$.  On the other hand, if $D_P$ is the restriction of a section if  $\mathring F_*\theta_{\mathring \Ccal/S}$, then $\Lscr(D_P)$ vanishes by the residue formula applied to the right hand side of the identity \ref{eqn:Loperator}. It follows that this expression only depends on the image of $D_P$ in $R^1F_*\theta_C(-2P)$. The factorization property of $\Lscr$ is then a straightforward check. 
\end{proof}

\begin{corollary}\label{cor:wzwdual}
The WZW-action $\nabla^{WZW}_D$ on $\Vcal^{\hat\lambda}_\qlog(\Ccal/S)$ is given by $\hat D =  \tilde D  - \Lscr_\pt^u\KS_{\Ccal/S,\vec P}(D)$.
\end{corollary}
\begin{proof}
Corollary \ref{cor:residuederivation} yields
$D\la \xi_N | \psi\ra=\sum_{i=1}^N \la \xi_N |\tilde D_i\psi\ra +\la \tilde D^{(N)}\xi_{N} |\psi \ra$.
Since $(\hat D\xi)_N= \sum_{i=1}^N \tilde D_i\xi_N+(\Lscr^\dagger(D_P)\xi)_N$, we then find
\[
\textstyle \la \xi | \hat D\psi\ra=\sum_{i=1}^N \la \xi |\tilde D_i\psi\ra+\la \xi|\Lscr^\dagger(D_P)\psi \ra = \big(D\la \xi | \psi\ra -\la \tilde D^{(N)}\xi |\psi \ra\big) +\la \Lscr(D_P)\xi |\psi )\ra 
\]
and so $\hat D\xi (\psi)=  \la \tilde D^{(N)}\xi |\psi \ra - \la \Lscr_N(D_P)\xi |\psi \ra$.  Since $\Lscr_N(D_P)=\Lscr_N^u\KS_{\Ccal/S,\vec P}(D)$, the corollary follows.
\end{proof}

As we shall now explain, once we divide out by the exact relative $N$-polydifferentials on $\mathring\Ccal^{(N)}/S$ (we here view $N$-polydifferentials on $\mathring\Ccal^{(N)}/S$ as closed $N$-forms twisted with a sign character), the  operator $\tilde D$  acquires a  De Rham interpretation.  Since $\inj_N\!\mathring F$ is affine, its direct image functors
are acyclic when applied to the terms to coherent $\Ocal_{\inj_N(\mathring\Ccal/S)}$-modules. If we apply this 
to the terms  of the relative algebraic De Rham complex, we obtain  isomorphisms
\[
\Hcal^N\big(\inj_I\!\mathring F_*\Omega^\pt_{\inj_N(\mathring\Ccal/S)/S}\big)\cong
\Ocal_S\otimes_{\underline\CC} R^N\inj_N\!\mathring F_*{\underline\CC(N)},
\]
where we note that the left hand side is the quotient of $\inj_I\!\mathring F_*\Omega^N_{\inj_N(\mathring\Ccal/S)/S}$ by the $\Ocal_S$-submodule of relative exact forms $d\inj_I\!\mathring F_*\Omega^{N-1}_{\inj_N(\mathring\Ccal/S)/S}$.

The derivation  $\tilde D$ acts on this De Rham resolution in such a manner that the induced action on the right hand side is given by $D\otimes 1$ on the left hand side.  This action will in general not preserve the Hodge filtration, in particular, it need not preserve 
$F^0\big(\Ocal_S\otimes_{\underline \CC} R^N\inj_N\!\mathring F_*\underline\CC(N)\big)$. A similar remark applies to $\Vcal^{\hat\lambda}_\qlog (\Ccal/S)$ (where then $F^0$ is replaced by $F^{-1}$). 

If in the preceding Corollary \ref{cor:wzwdual}, we divide out by the sheaf of exact polydifferentials, we find:

\begin{theorem}\label{thm:gmversuswzw} We have a natural $\Ocal_S$-homomorphism  
\[
\gamma_N: \Vcal^{\hat\lambda}_\qlog (\Ccal/S)\to 
\Ocal_S\otimes_{\underline\CC}\Hom_{\CC}^{\sign_N}(\glie^{\otimes N}\otimes V_\lambda,  
R^N\inj_N\!\mathring F_*\underline {\CC}(N))
\]
whose target is a variation of mixed Hodge structure that takes in fact its values in the $F^{-1}$-part of the Hodge filtration. This is an embedding for $N$ sufficiently large and the product of these maps over $N\ge 0$ is also an embedding of $\Rscr_P$-modules.
Moreover, if  $D\in  \theta_S$, then 
\[
\gamma_N\nabla^{WZW}_D=(D\otimes 1)\gamma_N -\Lscr_N^u\KS_{\Ccal/S,\vec P}(D),
\]
where $D$ acts on $\Ocal_S$ be derivation, $\Lscr_N^u: R^1F_*\theta_C(-2P)\to\Ocal_S\otimes_{\underline\CC}\Hom_{\CC}^{\sign_N}(\glie^{\otimes N}\otimes V_\lambda,  
R^N\inj_N\!\mathring F_*\underline {\CC}(N))$  is as in Lemma  \ref{lemma:Loperator} and $\KS_{\Ccal/S,\vec P}: \theta_S\to R^1F_*\theta_{\Ccal/S}(-2P)$ is the Kodaira-Spencer homomorphism. 
$\square$
\end{theorem}

\begin{remark}\label{rem:noweightfiltration}
The kernel  of  $\gamma_N$  consists of the $\xi\in\Vcal^{\hat\lambda}_\qlog (\Ccal/S)$ with $\xi_N=0$. This is in fact
$W_{-1-N}\Vcal^{\hat\lambda}_\qlog (\Ccal/S)$ and so  these kernels define the opposite of the weight filtration: 
Indeed, $\xi$ lies in this kernel if and only if $\xi_{N+1}$ is without poles and this is equivalent (in view  of the Tate twist with $\ZZ(N+1)$) to $\xi_N$ being of weight $-1-N$. 
While the Gau\ss-Manin connection preserves the weight filtration, we cannot expect this to be so for the  WZW-connection: it might well happen that there exists a $\xi \in \Vcal^{\hat\lambda}_\qlog (\Ccal/S)$ with $\xi_{N-1}=0$, but  for which  $\Lscr_{N-1}(D)\xi$ (which after all depends on $\xi_{N+1}$ and $D$) is nonzero for some $D\in\theta_S$. So there is no weight filtration on the underlying local system.
\end{remark}

\begin{remark}\label{rem:topinterpretation}
Taking our cue from the genus zero case, we expect that the WZW-connection  will be a Gau\ss-Manin connection on the nose if  instead of the ordinary cohomology of the configuration bundle, we take  its cohomology with values in a local system. 
We should perhaps allow this  local system to depend on the argument, in the sense  that we might have a fixed decomposition 
 of $\glie^{\otimes N}\otimes V_\lambda$ into $\glie$-subresentations with a local system specified for every summand and
also allow the cohomology to be subject to perversity (support) conditions. But in the end it should be expressible in terms of the De Rham theory of the fundamental groupoid of the pointed curve restricted to the points at infinity specified by nonzero tangent vectors (when there are no points, this must be replaced by the De Rham theory of the fundamental group up to inner automorphisms). Our hope is that this would then produce a direct summand of a variation of a polarized Hodge structure, which then at the same time comes with the desired unitary structure. But note that for such a program even in the genus zero case work remains to be done.
\end{remark}


\begin{thebibliography}{PP}

\bibitem{beauville} 
A.~Beauville:
\textsl{Conformal blocks, fusion rules and the Verlinde formula, } in: 
Proc.\ Hirzebruch 65 Conf.\  Alg.\ Geom., 75--96, Israel Math. Conf. Proc. \textbf{9}, Bar-Ilan Univ., Ramat Gan (1996).

\bibitem{bd}
A.~Beilinson,  V.~Drinfeld:
\textsl{Affine Kac-Moody Algebras and Polydifferentials,} Internat. Math. Res. Notices  \textbf{1} (1994), 1--11.

\bibitem{bd2004}
A.~Beilinson,  V.~Drinfeld:
\textsl{Chiral algebras.} 
American Mathematical Society Colloquium Publications \textbf{51},  American Mathematical Society, Providence, RI, 2004.

\bibitem{belkale}
P.~Belkale:
\textsl{Unitarity of the KZ/Hitchin connection on conformal blocks in genus 0 for arbitrary Lie algebras,}
J.\ Math.\ Pures Appl.\ (9) \textbf{98} (2012), 367--389.

\bibitem{BBM}
P.~Belkale, P.~Brosnan, S.~Mukhopadhyay:
\textsl{Hyperplane arrangements and tensor product invariants,}
Michigan Math.\ J.\  \textbf{68} (2019), 801--829. 

\bibitem{br1}
I.~Biswas, A.~K.~Raina:
\textsl{Projective structures on a Riemann surface,}
Internat.\ Math.\ Res.\ Notices \textbf{15} (1996), 753--768.

\bibitem{bcfp}
I.~Biswas, E.~Colombo, P.~Frediani, G.~P.~Pirola:
\textsl{A Hodge theoretic projective structure on Riemann surfaces,}
\url{https://arxiv.org/abs/1912.08595}


\bibitem{cohenf95}
F.~R.~Cohen:  \textsl{On configuration spaces, their homology, and Lie algebras,} 
Journal of Pure and Applied Algebra \textbf{100} (1995), 19--42.


\bibitem{CFG}
E.~Colombo, P.~Frediani, A.~Ghigi:
\textsl{On totally geodesic submanifolds in the Jacobian locus,}
Internat.\ J.\  Math. \textbf{26} (2015), no. 1.

\bibitem{deligne:hodge2}
P.~Deligne:
\textsl{Th\'eorie de Hodge II,}
Inst.\ Hautes \'Etudes Sci.\ Publ.\ Math.\ \textbf{40} (1971), 5--57. 

\bibitem{del-gon}
P.~Deligne,  A.~B.~Goncharov:
\textsl{Groupes fondamentaux motiviques de Tate mixte,}
Ann.\ Sci.\ \'Ecole Norm.\ Sup.\ \textbf{38} (2005), 1--56.

\bibitem{segal}
G.~Segal:
\textsl{The definition of conformal field theory,}  in:  Topology, geometry and quantum field theory, 421--577, 
London Math.\ Soc.\ Lecture Note Ser. \textbf{308}, Cambridge Univ. Press, Cambridge, 2004.

\bibitem{kac}
V.~G.~Kac
\textsl{Infinite dimensional Lie algebras,} 
Third edition. Cambridge University Press, Cambridge, 1990.


\bibitem{klyachko}
A.~A.~Klyachko:
\textsl{Lie elements in a tensor algebra,} 
Sibirsk.\ Mat. \v{Z}.\ \textbf{15} (1974), 1296--1304.

\bibitem{looij:2021}
E.~Looijenga:
\textsl{Torelli group action on the configuration space of a surface,} to appear in J.\ of Topology and Analysis.
\url{https://doi.org/10.1142/S1793525321500370}, \url{https://arxiv.org/abs/2008.10556}

\bibitem{looij2013}
E.~Looijenga:
 \textsl{From WZW models to modular functors,} in: Handbook of moduli. Vol. II, 427--466, 
Adv. Lect. Math. (ALM) \textbf{25} Int. Press, Somerville, MA (2013). 

\bibitem{looij2012}
E.~Looijenga:
 \textsl{The KZ system via polydifferentials,} in: Arrangements of hyperplanes--Sapporo 2009, 189--231, 
Adv.\ Stud.\ Pure Math.\ \textbf{62}, Math.\ Soc.\ Japan, Tokyo (2012). 
32G34 (14D07)

\bibitem{looij2010}
E.~Looijenga:
 \textsl{Unitarity of $\SL(2)$-conformal blocks in genus zero,} 
J.\ Geom.\ Phys.\  \textbf{59} (2009), 654--662.

\bibitem{looij1991}
E.~Looijenga:
 \textsl{Cohomology of $\Mcal_3$ and $\Mcal_3^1$. }  in: Mapping class groups and moduli spaces of Riemann surfaces (G\"ottingen, 1991/Seattle, WA, 1991), 205--228, Contemp.~Math.\ \textbf{150}, Amer.\ Math.\ Soc., Providence, RI, 1993.
 
\bibitem{ramadas} 
T.~R.~Ramadas:
\textsl{The ``Harder-Narasimhan trace'' and unitarity of the KZ/Hitchin connection: genus 0,} 
Ann.\ of Math.\ (2) \textbf{169} (2009), 1--39. 
 

\bibitem{sv}
V.~Schechtman, A.~Varchenko:
\textsl{Arrangements of Hyperplanes and Lie Algebra Homology}, Invent. Math. \textbf{106} (1991), 139--194.

\bibitem{totaro}
B.~Totaro: 
\textsl{Configuration spaces of algebraic varieties,}
Topology \textbf{35} (1996), 1057--1067. 

\bibitem{TUY}
A.~Tsuchiya,  K.~Ueno, Y.~Yamada:
\textsl{Conformal field theory on universal family of stable curves with gauge symmetries,} in:  Integrable systems in quantum field theory and statistical mechanics, 459--566, Adv.\ Stud.\ Pure Math.\ \textbf{19}, Academic Press, Boston, MA, 1989.
\end{thebibliography}
\end{document}